\documentclass{article}

\usepackage{fullpage,tikz}
\usepackage[utf8]{inputenc}
\usepackage{color,amsfonts,amsmath,amssymb,amsthm,cases,wasysym}
\usepackage[normalem]{ulem}
\usepackage{svg}
\usepackage{float}
\usepackage{mathrsfs}
\usepackage{hyperref} 

\title{Consistent Dimer Models on Surfaces with Boundary}
\author{Jonah Berggren$^1$ and Khrystyna Serhiyenko$^2$}
\date{
	$^1$University of Kentucky \href{mailto:jrberggren@uky.edu}{jrberggren@uky.edu} 
	\\
	$^2$University of Kentucky \href{mailto:khrystyna.serhiyenko@uky.edu}{khrystyna.serhiyenko@uky.edu}
}

\newtheorem{thm}{Theorem}[section]
\newtheorem{thmIntro}{Theorem}    
\newtheorem{conIntro}[thmIntro]{Corollary}    
\newtheorem{prop}[thm]{Proposition}
\newtheorem{lemma}[thm]{Lemma}
\newtheorem{cor}[thm]{Corollary}

\theoremstyle{definition}
\newtheorem{defn}[thm]{Definition}
\newtheorem{remk}[thm]{Remark}
\newtheorem{example}[thm]{Example}

\newcommand{\wind}{\textup{Wind}}

\newcommand{\A}{{\mathcal A}}
\renewcommand{\AA}{\mathscr A}
\newcommand{\F}{\mathcal F}

\begin{document}
\maketitle

\begin{abstract}
	A dimer model is a quiver with faces embedded in a surface. We define and investigate notions of consistency for dimer models on general surfaces with boundary which restrict to well-studied consistency conditions in the disk and torus case. We define weak consistency in terms of the associated dimer algebra and show that it is equivalent to the absence of bad configurations on the strand diagram.  In the disk and torus case, weakly consistent models are nondegenerate, meaning that every arrow is contained in a perfect matching; this is not true for general surfaces. Strong consistency is defined to require weak consistency as well as nondegeneracy. We prove that the completed as well as the noncompleted dimer algebra of a strongly consistent dimer model are bimodule internally 3-Calabi-Yau with respect to their boundary idempotents. As a consequence, the Gorenstein-projective module category of the completed boundary algebra of suitable dimer models categorifies the cluster algebra given by their underlying quiver. We provide additional consequences of weak and strong consistency, including that one may reduce a strongly consistent dimer model by removing digons and that consistency behaves well under taking dimer submodels.
\end{abstract}

\section{Introduction}

Dimer models were introduced as a model to study phase transitions in solid state physics.
In this setting, a dimer model is a bicolored graph embedded into a surface, representing a configuration of particles which may bond to one another.
The physics of this system is described by \textit{perfect matchings} of the graph, see the survey \cite{Kenyon} and references therein.
Moreover, to a dimer model one may associate a \textit{dimer algebra}, which is the Jacobian algebra of a certain quiver with potential, whose combinatorics and representation theory relate to the physics of the dimer model.
In the physics literature, dimer models on tori have seen the most study, especially those satisfying certain \textit{consistency conditions}~\cite{XHV}. 
Under these conditions, several authors including Mozgovoy and Reineke~\cite{XMR},  Davison~\cite{XDavison}, and Broomhead~\cite{XBroomhead2009}, showed that the dimer algebra is 3-Calabi-Yau. 
Ishii and Ueda~\cite{XIU2007} showed that the moduli space $\mathcal M_\theta$ of stable representations of the dimer algebra with dimension vector $(1,\dots,1)$ and a generic stability condition $\theta$ in the sense of King~\cite{XKing} is a smooth toric Calabi-Yau 3-fold. 
Moreover, the center $Z$ of the dimer algebra $A_Q$ is a Gorenstein affine 3-fold, 
the dimer algebra $A_Q$ is a non-commutative crepant resolution of $Z$, and $\mathcal M_\theta$ is a crepant resolution of $Z$~\cite{XIU2007}.
{Properties of the category of coherent sheaves over $\mathcal M_\theta$ may be understood through the combinatorics of the dimer model, opening a rich connection to mirror symmetry~\cite{XBocklandt2015,XBocklandt2013,FHKVW,XFU}.}

Many equivalent definitions of consistency have been introduced for torus dimer models. See, for example,~\cite[Theorem 10.2]{XBocklandt2011},~\cite{XBocklandt2009,XBroomhead2009,Kennaway,XMR}.
In particular, consistency of a dimer model is equivalent to the absence of certain bad configurations in the strand diagram of the dimer model~\cite[Theorem 1.37]{XBocklandt2015}. 

Dimer models on disks have been studied separately, and are of particular interest to the theory of cluster algebras. Postnikov introduced plabic graphs and strand diagrams in~\cite{Postnikov}. Scott~\cite{XScott} showed that the homogeneous coordinate ring of the Grassmannian $\textup{Gr}(k,n)$ is a cluster algebra, in which certain seeds are indexed by $(k,n)$-Postnikov diagrams.
Jensen-King-Su~\cite{XJKS} gave an additive categorification for this cluster structure, and Baur-King-Marsh~\cite{BKMX} interpreted this categorification as the Gorenstein-projective module category over the completed boundary algebra of the associated dimer model. Pressland extended these results to arbitrary Postnikov diagrams in~\cite{XPressland2019} and observed that a dimer model coming from a Postnikov diagram satisfies a \textit{thinness condition}, which is analogous to the algebraic consistency conditions in the torus literature.

{A systematic study on dimer models on more general surfaces was initiated by Franco in~\cite{BFTFDBPTSA}. This study is largely concerned with the master and mesonic moduli spaces on dimer models, which may be computed using the combinatorics of perfect matchings.
Operations such as removing an edge and the dual untwisting map were investigated in~\cite{BFTFDBPTSA,NDIBFT,BFTFDB}. Quiver mutation and square moves were connected with cluster mutation in~\cite{CTFBFT}, and further connected with combinatorial mutation of polytopes in~\cite{Twin}.
Dimer models on general surfaces were connected with matroid polytopes and used to obtain a partial matroid stratification of the Grassmannian, generalizing the place of dimer models in disks in the matroid stratification of the Grassmannian~\cite{TGOOSD, BFTCAATG, NPOSD, SOSCVTGONPOSD, HCCAQFTD}.
Various generalizations of the notions of consistency in the disk and torus case have been considered in this body of work.} 

\nocite{BBMTCYFADQ}
\nocite{BFTDBI}
\nocite{BFT2}
\nocite{BB}

We define a new notion of consistency for dimer models on compact orientable surfaces with or without boundary.
We call a dimer model \textit{path-consistent} if, for any fixed vertices $v_1$ and $v_2$ and a fixed homotopy class $C$ of paths from $v_1$ to $v_2$, there is a unique (up to path-equivalence) \textit{minimal path} $r$ from $v_1$ to $v_2$ in $C$ such that any path from $v_1$ to $v_2$ in $C$ is equivalent to $r$ composed with some number face-paths. When $Q$ is a dimer model on a torus, path-consistency is equivalent to the many consistency conditions in the literature. When $Q$ is a dimer model on a higher genus surface without boundary, path-consistency is equivalent to the weaker notions of consistency rather than the stronger algebraic consistency. See~\cite[Theorem 10.2]{XBocklandt2011}. When $Q$ is on a disk, path-consistency is the thinness condition appearing in Pressland~\cite{XPressland2019}.

We associate a strand diagram to a dimer model and define \textit{bad configurations}. We say that a dimer model is \textit{strand-consistent} if it has no bad configurations. This matches the notion of zigzag consistency of general dimer models on surfaces with boundary briefly considered in the first section of~\cite{XBocklandt2015}. In particular, it agrees with the well-studied notions of consistency in the torus case.

Our first main theorem is as follows, where we exclude the case of a sphere without boundary as such a dimer model is never strand-consistent.  A key idea of the proof is to observe that either notion of consistency of a dimer model is equivalent to consistency of its (possibly infinite) universal cover model, which enables the assumption of simple connectedness.

\begin{thmIntro}[Theorem~\ref{thm:cons-alg-str}]\label{thm:main1}
	Let $Q$ be a dimer model not on a sphere. The following are equivalent:
\begin{enumerate}
	\item\label{q1} The dimer model $Q$ is path-consistent.
	\item\label{q2} The dimer model $Q$ is strand-consistent.
	\item The dimer algebra $A_Q$ is cancellative.
\end{enumerate}
\end{thmIntro}

We may thus say that a \textit{weakly consistent dimer model} is a model not on a sphere satisfying any of the above equivalent conditions. This generalizes results in the case of the torus~\cite[Theorem 10.1]{XBocklandt2011},~\cite{XIU}.
This was also shown for dimer models on the disk corresponding to $(k,n)$-diagrams in~\cite{BKMX}. The implication~\eqref{q2}$\implies$\eqref{q1} for general dimer models on disks appears in~\cite[Proposition 2.11]{XPressland2019}. 
A corollary of our result is the reverse direction in the disk case (Corollary~\ref{cor:main-cor-disk}).

As an application, we use the strand diagram characterization of consistency to prove that \textit{dimer submodels} of weakly consistent dimer models are weakly consistent (Corollary~\ref{cor:dimer-submodel-consistent}). 
This gives us practical ways to get new weakly consistent models from old and to understand equivalence classes of minimal paths. 

Next, we study \textit{perfect matchings} of weakly consistent dimer models. In the torus case, perfect matchings of the dimer model feature prominently~\cite{XIU,XBroomhead2009,XBocklandt2011}. Perfect matchings of a torus dimer model generate the cone of R-symmetries, which have applications in physics. Perfect matchings may be used to calculate the perfect matching polygon of the dimer model, which is related to the center of the dimer algebra.
Perfect matchings of a dimer model on a disk~\cite{CKP,XLam2015} are the natural analog and may be connected with certain perfect matching modules of the completed dimer algebra to understand the categorification given by the boundary algebra of a dimer model on a disk~\cite{CKP}.
{Over arbitrary compact surfaces with boundary, perfect matchings may be used to describe the master and mesonic moduli spaces associated to the dimer model. Moreover, perfect matchings of a dimer model on a general surface may be calculated by taking determinants of Kasteleyn matrices~\cite{HaK}, \cite[\S5]{BFTFDBPTSA}.}
In Theorem~\ref{thm:geo-consistent-then-almost-perfect-matching}, we show that any (possibly infinite) simply connected weakly consistent dimer model has a perfect matching. This means that the universal cover model of any weakly consistent dimer model has a perfect matching. On the other hand, we give an example of a (non-simply-connected) weakly consistent dimer model which has no perfect matching (Example~\ref{ex:cons-tor-no-perf}).
One important notion for dimer models in the disk and torus is \textit{nondegeneracy}, which requires that all arrows are contained in a perfect matching. We extend this definition to general surfaces and show that nondegeneracy gives a positive grading to the dimer algebra.
We define a \textit{strongly consistent} dimer model as one which is weakly consistent and nondegenerate. In the disk and torus case, weak and strong consistency are equivalent, but this is not true for more general surfaces.
We then use~\cite{XPressland2019} to prove the following result. 
\begin{thmIntro}[{Theorem~\ref{thm:Calabi-Yau}}]\label{thm:AAAz}
	Let $Q$ be a finite strongly consistent dimer model. Then the dimer algebra $A_Q$ and the completed dimer algebra $\widehat A_Q$ are bimodule internally 3-Calabi-Yau with respect to their boundary idempotents.
\end{thmIntro}
When $Q$ is a dimer model on a disk, we recover~\cite[Theorem 3.7]{XPressland2019}. When $Q$ has no boundary, this translates to the algebra being 3-Calabi-Yau~\cite[Remark 2.2]{XPressland2015}. Hence, we recover the statement in the torus (and closed surface of higher genus) literature that consistent dimer models are 3-Calabi-Yau proven in~\cite[Corollary 4.4]{XDavison}.
Using~\cite[Theorem 4.1 and Theorem 4.10]{AIRX}, Theorem~\ref{thm:AAAz} immediately implies the following. 
\begin{conIntro}[{Corollary~\ref{thm:cat-thm}}]\label{cor:GFDS}
	Let $Q$ be a strongly consistent, Noetherian, and boundary-finite (Definition~\ref{defn:bf}) dimer model with no 1-cycles or 2-cycles. Then the Gorenstein-projective module category of the completed boundary algebra of $Q$ categorifies the cluster algebra given by the ice quiver of $Q$.
\end{conIntro}
We use the term ``categorification'' for brevity during the introduction; see Corollary~\ref{thm:cat-thm} for a more rigorous statement.
We give some examples of strongly consistent dimer models on annuli satisfying the requirements of Corollary~\ref{cor:GFDS}. 

{Next, we use the theory of dimer submodels to get some interesting results about equivalence classes of minimal paths in (weakly and strongly) consistent dimer models. We prove that in a weakly consistent dimer model, minimal \textit{leftmost} and \textit{rightmost} paths in a given homotopy class between two vertices are unique when they exist. If we further assume nondegeneracy, then they always exist.}

Finally, we study the reduction of dimer models. In the disk case, consistent dimer models with at least three boundary vertices may be \textit{reduced} in order to obtain a dimer model with no 2-cycles and an isomorphic dimer algebra~\cite[\S2]{XPressland2019}. 
We show in Proposition~\ref{prop:reduce-dimer-model} that a similar process may be used to remove certain, but not all, digons in the non-simply-connected case. 
Figure~\ref{fig:non-red-annulus} gives a weakly (but not strongly) consistent dimer model with a digon which may not be removed in this way. 
On the other hand, Corollary~\ref{cor:degenreduce} states that if we require strong consistency, then we may remove all digons from a dimer model.

The article is organized as follows. In Section~\ref{sec:first-section}, we define dimer models and prove that path-consistency is equivalent to cancellativity. We also show that these notions behave well when passing to the universal cover of a dimer model.
In Section~\ref{sec:mac}, we develop some technical theory of basic and cycle-removing morphs in order to prove that a path-consistent and simply connected dimer model has no \textit{irreducible pairs} (Theorem~\ref{prop:no-irreducible-pairs}). This result is used in Section~\ref{sec:pdsc} to complete the proof of Theorem~\ref{thm:main1} by showing that path-consistency and strand-consistency are equivalent.
Next, in Section~\ref{sec:ds}, we introduce dimer submodels and prove that dimer submodels of weakly consistent dimer models are weakly consistent (Corollary~\ref{cor:dimer-submodel-consistent}). 
This gives us practical ways to get new weakly consistent models from old and to understand equivalence classes of minimal paths. 
Section~\ref{sec:perfect-matching} is dedicated to perfect matchings of weakly consistent dimer models. 
In Section~\ref{sec:3cy}, we prove that the noncompleted and completed dimer algebras of a strongly consistent dimer model are bimodule internally 3-Calabi-Yau with respect to their boundary idempotents (Theorem~\ref{thm:AAAz}). 
As a result, we obtain our categorification result in Corollary~\ref{cor:GFDS}.
In Section~\ref{sec:eraec}, we use the results of Section~\ref{sec:ds} to understand the equivalence classes of minimal paths in (weakly and/or strongly) consistent dimer models.
Lastly, in Section~\ref{sec:red}, we discuss the process of reducing a dimer model by removing digons. We prove that if $Q$ is strongly consistent, then all digons may be removed.

\subsection*{Acknowledgments}

The authors thank Matthew Pressland for useful discussions and for comments on preliminary versions. The authors also thank two anonymous referees for their careful review of the manuscript.
The authors were supported by the NSF grant DMS-2054255.

\section{Covers and Consistency}\label{sec:first-section}

In this section we define a dimer model on an arbitrary surface with boundary. We introduce {path-consistency}, which generalizes notions of consistency of dimer models on the disk and torus. We show that path-consistency is equivalent to cancellativity. Moreover, we prove that these notions work well with taking the universal cover of a dimer model.

\subsection{Dimer Models}

We begin by defining dimer models, following~\cite[\S3]{BKMX}.
A \textit{quiver} is a directed graph. A \textit{cycle} of $Q$ is a nonconstant oriented path of $Q$ which starts and ends at the same vertex.  A cycle of length $a$ is called an $a$-\emph{cycle}.
Two cycles $\alpha_1\dots\alpha_a$ and $\beta_1\dots\beta_b$ are \emph{cyclically equivalent} if $a=b$ and there is some $j\in\mathbb Z$ such that $\alpha_i=\beta_{i+j}$ (where the subscript addition is calculated modulo $a$) for all $i\in[a]$.
If $Q$ is a quiver, we write $Q_{\text{cyc}}$ for the set of cycles in $Q$ of length at least two up to cyclic equivalence.

\begin{defn}
	A \textit{quiver with faces} is a triple $Q=(Q_0,Q_1,Q_2)$, where $(Q_0,Q_1)$ are the vertices and arrows of a quiver and $Q_2\subseteq Q_{\text{cyc}}$ is a set of \textit{faces} of $Q$.
\end{defn}

A \textit{digon} of $Q$ is a face in $Q_2$ consisting of two arrows.
Given a vertex $i\in Q_0$, we define the \textit{incidence graph} of $Q$ at $i$ to be the graph whose vertices are given by the arrows incident to $i$ and whose arrows $\alpha\to\beta$ correspond to paths
\[\xrightarrow{\alpha}i\xrightarrow{\beta}\]
which occur in faces of $Q$.

\begin{defn}\label{defn:dimer-model}
	A (locally finite, oriented) \textit{dimer model with boundary} is given by a quiver with faces $Q=(Q_0,Q_1,Q_2)$, where $Q_2$ is written as a disjoint union $Q_2=Q_2^{cc}\cup Q_2^{cl}$, satisfying the following properties:
    \begin{enumerate}
        \item Each arrow of $Q_1$ is in either one face or two faces of $Q$. An arrow which is in one face is called a \textit{boundary arrow} and an arrow which is in two faces is called an \textit{internal arrow}.
	\item Each internal arrow appears once in one cycle bounding a face in $Q_2^{cc}$ and once in one cycle bounding a face in $Q_2^{cl}$.
	{\item\label{ddm:4} The incidence graph of $Q$ at each vertex is connected.}
	\item Any vertex of $Q$ is incident with a finite number of arrows.
    \end{enumerate}
\end{defn}

A vertex of $Q$ is called boundary if it is adjacent to a boundary arrow, and otherwise it is called internal. 

Given a dimer model with boundary $Q$ we may associate each face $F$ of $Q$ with a polygon whose edges are labeled by the arrows in $F$ and glue the edges of these polygons together as indicated by the directions of the arrows to form a surface with boundary $S(Q)$ into which $Q$ may be embedded~\cite[Lemma 6.4]{XBocklandt2009}. The surface $S(Q)$ is oriented such that the cycles of faces in $Q_2^{cc}$ are oriented positive (or counter-clockwise) and the cycles of faces in $Q_2^{cl}$ are oriented negative (or clockwise). The boundary of $S(Q)$ runs along the boundary arrows of $Q$. If $S(Q)$ is a disk, then we say that $Q$ is a \textit{dimer model on a disk}. If $S(Q)$ is simply connected, then we say that $Q$ is a \textit{simply connected dimer model}.

A dimer model $Q$ is \textit{finite} if its vertex set is finite. Note that $Q$ is finite if and only if $S(Q)$ is compact.

Suppose that $Q$ is a finite quiver 
such that every vertex has finite degree. Suppose further that $Q$ has an embedding into an oriented surface $\Sigma$ with boundary such that the complement of $Q$ in $\Sigma$ is a disjoint union of discs, each of which is bounded by a cycle of $Q$. We may then view $Q$ as a dimer model with boundary by declaring $Q_2^{cc}$ (respectively $Q_2^{cl}$) to be the set of positively (respectively, negatively) oriented cycles of $Q$ which bound a connected component of the complement of $Q$ in $\Sigma$. All dimer models may be obtained in this way.

Let $Q$ be a dimer model and let $p$ be a path in $Q$. We write $t(p)$ and $h(p)$ for the start and end vertex of $p$, respectively. If a path $q$ can be factored in the form $q=q_2pq_1$, where $h(q_1)=t(p)$ and $t(q_2)=h(p)$, we say that $p$ is in $q$ or that $q$ contains $p$ as a subpath and we write $p\in q$. Corresponding to any vertex $v$ is a \textit{constant path} $e_v$ from $v$ to itself which has no arrows. 

Any arrow $\alpha$ in $Q$ is associated with at most one clockwise and one counter-clockwise face of the dimer model. We refer to these faces as $F_\alpha^{cl}$ and $F_\alpha^{cc}$, respectively, when they exist. Let $R_\alpha^{cl}$ (respectively $R_\alpha^{cc}$) be the subpath of $F_\alpha^{cl}$ (respectively $F_\alpha^{cc}$) beginning at $h(\alpha)$ and ending at $t(\alpha)$, and consisting of all arrows in $F_\alpha^{cl}$ (respectively $R_\alpha^{cc}$) except for $\alpha$. A path in $Q$ of the form $R_\alpha^{cl}$ (respectively $R_\alpha^{cc}$) for some $\alpha$ is called a \textit{clockwise return path} (respectively a \textit{counterclockwise return path}) of $\alpha$. 

\begin{defn}\label{defn:dimer-algebra}
    Given a dimer model with boundary $Q$, the \textit{dimer algebra} $A_Q$ is defined as the quotient of the path algebra $\mathbb CQ$ by the relations
	\[R_\alpha^{cc}-R_\alpha^{cl}\]
    for every internal arrow $\alpha\in Q_1$.
\end{defn}

We now make more definitions. We say that two paths $p$ and $q$ in $Q$ are \textit{path-equivalent} if their associated elements in the dimer algebra $A_Q$ are equal. If $p$ is a path in $Q$, we write $[p]$ for the path-equivalence class of $p$ under these relations.

The set of \textit{left-morphable} (respectively \textit{right-morphable}) arrows for $p$ is the set of internal arrows $\alpha\in Q_1$ such that $R_\alpha^{cc}$ (respectively $R_\alpha^{cl}$) is in $p$. The set of \textit{morphable arrows} for $p$ is the set of arrows which are left-morphable or right-morphable for $p$. 
Let $\alpha$ be a morphable arrow for $p$. We also say that $p$ \textit{has the morphable arrow $\alpha$}. Then $p$ contains $R_\alpha^{cl}$ or $R_\alpha^{cc}$ as a subpath, and may possibly contain multiple such subpaths. If $p'$ is a path obtained from $p$ by replacing a single subpath $R_\alpha^{cl}$ with $R_\alpha^{cc}$ ($R_\alpha^{cc}$ with $R_\alpha^{cl}$, respectively), then we say that $p'$ is a \textit{basic right-morph} (respectively, \textit{basic left-morph}) of $p$. We omit the word ``basic'' when the context is clear. If $p$ only has one subpath which is a copy of $R_\alpha^{cl}$ or $R_\alpha^{cc}$, then we say that $p'$ is an \textit{unambiguous basic (right or left) morph} of $p$ and we write $p'=m_\alpha(p)$. We say that $\alpha$ is an \emph{unambiguous} morphable arrow for $p$ in this case.
Since the relations of $A_Q$ are generated by the relations $\{R_\alpha^{cc}-R_\alpha^{cl}\ :\ \alpha\textup{ is an internal arrow of }Q\}$, two paths $p$ and $q$ are path-equivalent if and only if 
there is a sequence of paths $p=r_1,\dots,r_m=q$ such that $r_{i+1}$ is a basic morph of $r_i$ for $i\in[m-1]$. 

Suppose $p$ is a cycle in $Q$ which starts and ends at some vertex $v$ and travels around a face of $Q$ once. Then we say that $p$ is a \textit{face-path} of $Q$ starting at $v$. The terminology is justified by the following observation which follows from the defining relations.

\begin{remk}\label{faces-are-equivalent}
    Any two face-paths of $Q$ starting at $v$ are path-equivalent.
\end{remk}

\begin{defn}
    For all $v\in Q_0$, fix some face-path $f_v$ at $v$. Then define
    \begin{equation}
	    [f]:=\sum_{v\in Q_0}[f_v].
    \end{equation}
\end{defn}

If $|Q_0|$ is finite, then $[f]$ is an element of $A_Q$. It follows from Remark~\ref{faces-are-equivalent} that the path-equivalence class $[f]$ is independent of the choice of $f_v$ for all $v\in Q_0$. Moreover, the dimer algebra relations imply that $[f]$ commutes with every arrow. Hence, if $|Q_0|$ is finite, then $[f]$ is in the center of $A_Q$. 
If $|Q_0|$ is not finite, then $[f]$ 
is not an element of the dimer model $A_Q$.
However, every element $x$ of $A_Q$ has a well-defined product with $f$, so we use notation such as $[xf^m]$ in this case as well. 

The \textit{completed path algebra} $\langle\langle\mathbb CQ\rangle\rangle$ is the completion of $\mathbb CQ$ with respect to the arrow ideal. The completed path algebra has as its underlying vector space the \textit{possibly infinite} linear combinations of (distinct) finite paths in $Q$. Multiplication in $\langle\langle\mathbb CQ\rangle\rangle$ is induced by composition. See~\cite[Definition 2.6]{XPressland2020}.

\begin{defn}\label{defn:completed-dimer-algebra}
	The \textit{completed dimer algebra} $\widehat A_Q$ is the quotient of the completed path algebra $\mathbb C\langle\langle Q\rangle\rangle$ by the closure $\widehat I_Q$ of the ideal generated by the relations $R_\alpha^{cc}-R_\alpha^{cl}$ for each internal arrow $\alpha$ with respect to the arrow ideal.
\end{defn}

Elements of $\widehat A_Q$ are possibly infinite linear combinations of (finite) paths of $Q$, with multiplication induced by composition.

\subsection{Path-Consistency}\label{sec:hc}

We now define path-consistency, which is a nice 
condition on the equivalence classes of paths between two vertices. We prove some short lemmas about path-consistent models.

A path $p$ in $Q$ is also a path in the surface $S(Q)$. We thus say that paths $p$ and $q$ of $Q$ are \textit{homotopic} if they are homotopic as paths in $S(Q)$.

\begin{defn}
	A path $p$ in a dimer model $Q$ is \textit{minimal} if we may not write $[p]=[qf^m]$ for any $m\geq1$.
\end{defn}

\begin{defn}\label{defn:algebraic-consistency}
	A dimer model $Q=(Q_0,Q_1,Q_2)$ is \textit{path-consistent} if it satisfies the following \textit{path-consistency condition}:

         \indent\indent\hangindent=1cm Fix vertices $v_1$ and $v_2$ of $Q$. For any homotopy class $C$ of paths in $S(Q)$ from $v_1$ to $v_2$ there is a minimal path $p_{v_2v_1}^C$, unique up to path-equivalence, with the property that any path $p$ from $v_1$ to $v_2$ in $Q$ in the homotopy class $C$ satisfies $[p]=[f^mp_{v_2v_1}^C]$ for a unique nonnegative integer $m$.  We call $m$ the \textit{c-value} of $p$.
\end{defn}

We remark that equivalent paths of a general dimer model must be homotopic, so in some sense this is the ``lowest number of path equivalence classes'' that one could hope for in a dimer model.

\begin{lemma}\label{lem:c-values-compose}
    If $p$ and $q$ are paths in a path-consistent dimer model $Q$ with $h(p)=t(q)$, then the c-value of the composition $qp$ is greater than or equal to the c-value of $p$ plus the c-value of $q$.
\end{lemma}
\begin{proof}
    If $p$ and $q$ are paths in a path-consistent dimer model $Q$ with $h(p)=t(q)$, then we may write $[p]=[f^{m_p}r_p]$ and $[q]=[f^{m_q}r_q]$ for some minimal paths $r_p$ and $r_q$. Then $m_p$ is the c-value of $p$ and $m_q$ is the c-value of $q$. Then, using the fact that $[f]$ is central, we calculate
    \[[qp]=[f^{m_q}r_qf^{m_p}r_p]=[f^{m_q+m_p}r_qr_p].\]
	We have shown that $[f^{m_q+m_p}]$ may be factored out of $[qp]$, hence the c-value of $[qp]$ is greater than or equal to $m_q+m_p$.
\end{proof}

Two paths are equivalent if and only if there is a sequence of basic morphs taking one to the other. Since a basic morph cannot remove some arrows without replacing them with other arrows, the constant path is the unique minimal path from a vertex to itself. This leads to the following remark.

\begin{remk}\label{rem:cycle-cant-be-constant}
    If $Q$ is path-consistent and $p$ is a nonconstant null-homotopic cycle, then the c-value of $p$ is positive.
\end{remk}

It is an important fact that all face-paths of a dimer model are null-homotopic. This lets us show the following.

\begin{lemma}\label{lem:subpath-of-facepath-is-thin}
    Let $Q$ be a path-consistent dimer model. Any proper subpath of a face-path of $Q$ is minimal.
\end{lemma}
\begin{proof}
	Suppose $p$ is a proper subpath of a face-path $f_v$ starting at $v:=t(p)$. Let $p'$ be the subpath of $f_v$ such that $p'p=f_v$. If $p$ is not minimal, then by definition of path-consistency, $[p]=[rf^m]$ for some positive integer $m$ and some minimal path $r$ from $v$ to $h(p)$ homotopic to $p$. 
	Then $[p'p]=[p'rf^m]=[p'rf_v^m]$. Moreover, $r$ is homotopic to $p$, hence $p'r$ is homotopic to the face-path $p'p$, hence is null-homotopic. Then by Remark~\ref{rem:cycle-cant-be-constant} it has some positive c-value of $m'$. By definition of path-consistency, $[p'r]=[f_v^{m'}]$. It follows that 
	\[[f_v]=[p'p]=[p'rf_v^m]=[f_v^{m'}f_v^m]=[f_v^{m'+m}],\] which is a contradiction since $m+m'\geq1+1=2$ but all face-paths trivially have a c-value of 1. It follows that $p$ is minimal.
\end{proof}

\begin{lemma}\label{lem:subpath-of-facepath-not-cycle}
	Let $Q$ be a simply connected path-consistent dimer model. No proper subpath of a face-path of $Q$ is a cycle.
\end{lemma}
\begin{proof}
	Suppose $p$ is a proper subpath of a face-path of $Q$ which is a cycle. By Lemma~\ref{lem:subpath-of-facepath-is-thin}, the path $p$ is minimal. Since $Q$ is simply-connected, $p$ is null-homotopic. The only minimal null-homotopic path from a vertex to itself is the constant path, so this is a contradiction.
\end{proof}

\subsection{Universal Covers}

We define the notion of a \textit{universal cover dimer model} and show that it behaves well with respect to path-consistency and the cancellation property.

Let $Q$ be a dimer model. We construct a dimer model $\widetilde Q$ over the universal cover $\widetilde{S(Q)}$ of $S(Q)$. We consider $Q$ to be embedded into $S(Q)$, so that a vertex $v\in Q$ may be considered as a point of $S(Q)$. Similarly, we describe $\widetilde Q$ embedded into $\widetilde{S(Q)}$.

The vertices of $\widetilde Q$ are the points $\tilde v\in\widetilde{S(Q)}$ which descend to vertices $v$ of $S(Q)$. For any arrow $\alpha$ from $v$ to $w$ in $Q$ and any $\tilde v\in\widetilde Q_0$, there is an arrow $\alpha_{\tilde v}$ obtained by lifting $\alpha$ as a path in $Q$ up to a path in $\widetilde Q$ starting at $\tilde v$. The face-paths of $\widetilde Q$ are similarly induced by lifting the face-paths of $Q$. It is not hard to see that $\widetilde Q$ is a (locally finite) dimer model. 
The following facts follow by universal cover theory.
\begin{enumerate}
	\item If $S(Q)$ is not a sphere, then $\widetilde{S(Q)}$ is not a sphere.
	\item The surface $\widetilde{S(Q)}$ is simply connected.
	\item Let $\tilde p$ and $\tilde q$ be paths in $\widetilde Q$ with the same start and end vertices which are lifts of paths $p$ and $q$ of $Q$, respectively. Then $[p]=[q]$ in $A_Q$ if and only if $[\tilde p]=[\tilde q]$ in $A_{\widetilde Q}$.
\end{enumerate}

Universal covers are useful to consider because simple cycles on the universal cover have well-defined interiors. The following remark gives another advantage of universal covers.

\begin{remk}\label{remk:Q-hat-Q-path-correspondence}
	Choose vertices $\tilde v_1$ and $\tilde v_2$ of $\widetilde Q$. Any two paths from $\tilde v_1$ to $\tilde v_2$ are homotopic, hence descend to homotopic paths in $Q$. Then this choice of vertices gives a homotopy class $C$ of paths between the corresponding vertices $v_1$ and $v_2$ of $Q$. The paths from $\tilde v_1$ to $\tilde v_2$ in $\widetilde Q$ correspond precisely to the paths from $v_1$ to $v_2$ in the homotopy class $C$. Equivalence classes of paths in the dimer algebra are respected by this correspondence.
\end{remk}

Remark~\ref{remk:Q-hat-Q-path-correspondence} relates $Q$ and $\widetilde Q$ in a useful way.
Many of our technical results require simple connectedness. Passing to the universal cover model allows us to prove things about general dimer models $Q$ by considering their simply connected universal cover models.
In particular, we may study path-consistency of $Q$ by studying path-consistency of $\widetilde Q$. 

\begin{prop}\label{prop:Q-consistent-iff-hat-Q-consistent}
        A dimer model $Q$ is path-consistent if and only if $\widetilde Q$ is path-consistent.
\end{prop}
\begin{proof}
	Suppose $Q$ is path-consistent. Choose vertices $\tilde v_1$ and $\tilde v_2$ of $\widetilde Q$; these correspond to vertices $v_1$ and $v_2$ of $Q$ and induce a homotopy class $C$ of paths between them. By Remark~\ref{remk:Q-hat-Q-path-correspondence}, the paths in $\widetilde Q$ from $\tilde v_1$ to  $\tilde v_2$ correspond to the paths in $Q$ from $v_1$ to $v_2$ in $C$. By path-consistency of $Q$, each such path in $Q$ is equivalent to $p_{v_2v_1}^C$ composed with some power of a face-path, hence $\widetilde Q$ is path-consistent.

	The other direction is similar.
\end{proof}

It follows from Lemma~\ref{lem:subpath-of-facepath-not-cycle} and Proposition~\ref{prop:Q-consistent-iff-hat-Q-consistent} that a path-consistent dimer model cannot have contractible loops.

\begin{defn}
	A dimer algebra $A_Q=\mathbb C Q/I$ is called a \textit{cancellation algebra} (or \textit{cancellative}) if for paths $p,q,a,b$ of $Q$ with $h(a)=t(p)=t(q)$ and $t(b)=h(p)=h(q)$, we have $[pa]=[qa]\iff [p]=[q]$ and $[bp]=[bq]\iff [p]=[q]$. We call this the \textit{cancellation property}.
\end{defn}

\begin{lemma}\label{lem:Q-canc-iff-hat-Q-canc}
	$A_Q$ is a cancellation algebra if and only if $A_{\widetilde Q}$ is a cancellation algebra.
\end{lemma}
\begin{proof}
	This follows because $[p]=[q]$ in $A_Q$ if and only if $[\tilde p]=[\tilde q]$ in $A_{\widetilde Q}$, where $\tilde p$ and $\tilde q$ are any lifts of $p$ and $q$ to $\widetilde Q$ with $t(\tilde p)=t(\tilde q)$.
\end{proof}

\begin{lemma}\label{lem:many-cycles-factor-through-path}
	Let $p$ be a path in $Q$ of length $m$. Then the composition of face-paths $f^m_{t(p)}$ is equivalent to a path beginning with $p$.
\end{lemma}
\begin{proof}
	Let $p=\gamma_m\dots\gamma_1$ be a product of arrows.
	For each $\gamma_i$, let $R_{\gamma_i}$ be a return path of $\gamma_i$. The path $R_{\gamma_1}\dots R_{\gamma_m}\gamma_m\dots\gamma_1$ is equivalent to $f^m_{t(p)}$ and begins with $p$.
\end{proof}

We now show that the notions of path-consistency and cancellativity coincide.

\begin{thm}\label{thm:consistent-iff-cancellation}
	A dimer model $Q$ is path-consistent if and only if $A_Q$ is a cancellation algebra. 
\end{thm}
\begin{proof}
	By Lemma~\ref{lem:Q-canc-iff-hat-Q-canc} and Proposition~\ref{prop:Q-consistent-iff-hat-Q-consistent}, it suffices to show the result on the universal cover $\widetilde Q$.

	Suppose that $\widetilde Q$ is path-consistent. 
	We prove that $A_{\widetilde Q}$ is a cancellation algebra. Accordingly, take paths $p,q,a$ of ${\widetilde Q}$ with $h(a)=t(p)=t(q)$ and $h(p)=h(q)$. We show that $[pa]=[qa]\implies [p]=[q]$. The case of left composition is symmetric. By path-consistency, we may write $[p]=[rf^{m_p}]$ and $[q]=[rf^{m_q}]$, where $r$ is a minimal path from $t(p)$ to $h(p)$, necessarily homotopic to $p$ and $q$ by simple connectedness. Given $[pa]=[qa]$, we have $[f^{m_p}ra]=[f^{m_q}ra]$. Then $m_p=m_q$ by path-consistency. We have shown that
	\[[q]=[rf^{m_q}]=[rf^{m_p}]=[p]\]
	and the proof of this direction is complete.

	Suppose now that $A_{\widetilde Q}$ is a cancellation algebra. 
	We first show that only a finite number of face-paths may be factored out of any path $p$, and that this number is bounded by the number of arrows in $p$. Suppose to the contrary that there is some path $p$ of $Q$ with $m$ arrows such that we may write
		$[p]=[p'f^{m'}]$ for some $m'>m$. By Lemma~\ref{lem:many-cycles-factor-through-path}, $[p'f^{m'}]=[lp]$ for some nonconstant cycle $l$ at $h(p)$. 
	Applying the cancellation property to the equation $[p]=[p'f^{m'}]=[lp]$ gives that $l$ is equivalent to the constant path, which is a contradiction. This shows that only a finite number of face-paths may be factored out of any path of ${\widetilde Q}$.

	Then any path $p$ of ${\widetilde Q}$ is equivalent to $rf^m$ for a minimal path $r$ and a nonnegative integer $m$. Suppose $[p]=[rf^m]=[r'f^{m'}]$ for some nonnegative integers $m$ and $m'$ and minimal paths $r$ and $r'$. Without loss of generality suppose $m\leq m'$. By the cancellation property, $[r]=[r'f^{m'-m}]$. Then if $m'>m$ we have factored a face-path out of $r$, contradicting minimality of $r$, hence $m'=m$ and $[r]=[r']$. 

	Then if ${\widetilde Q}$ is not path-consistent, there must be minimal paths $p$ and $q$ between the same vertices which are not equivalent. Take $m$ which is greater than the length of $p$ and the length of $q$. 
		By Lemma~\ref{lem:many-cycles-factor-through-path}, $[pf^m]=[pq'q]$ for some path $q'$ from $h(q)$ to $t(q)$. Suppose we have shown that $pq'$ is equivalent to $f^{m'}$ for some $m'$. Then 
				\[[pf^m]=[pq'q]=[qf^{m'}].\]
			By the cancellation property, we get either $[p]=[qf^{m'-m}]$ or $[q]=[pf^{m-m'}]$. 
			Since $p$ and $q$ are minimal, we have $m=m'$ and $[p]=[q]$, contradicting our initial assumption. The proof is then complete if we show that any cycle is equivalent to a composition of face-paths. We do so now.

			Suppose to the contrary and take a simple cycle $l=\gamma_{s'}\dots\gamma_1$ which is not equivalent to a composition of face-paths and such that every simple cycle inside the disk bounded by $l$ is equivalent to a composition of face-paths.  Note that if   $\widetilde{S(Q)}$ is a sphere then choose one of the two regions that $l$ bounds as being the disk. 
			For any arrow $\gamma_i\in l$, let $R_{\gamma_i}$ be the return path of $\gamma_i$ inside of the disk bounded by $l$. 
				As in Lemma~\ref{lem:many-cycles-factor-through-path}, set $l':=R_1\dots R_{s'}$. Then $l'l=R_1\dots R_{s'}\gamma_{{s'}}\dots\gamma_1$ is equivalent to $f_{t(l)}^{s'}$. Moreover, $l'$ is a cycle lying in the area bounded by $l$.
			If $l'$ is a simple cycle strictly contained in $l$ then $l'$ is equivalent to a composition of face-paths by choice of $l$. If $l'$ is not a simple cycle, then one-by-one we remove simple proper subcycles of $l'$, each of which is strictly contained in the area defined by $l$ and hence is equivalent to a composition of face-paths. Then we replace them with a composition of face-paths until 
			we get $[l']=[f_{t(p)}^{s}]$ for some $s$.

			Either way, $l'$ is equivalent to some composition of face-paths $f_{t(p)}^s$. Then $[f_{t(p)}^{s'}]=[l'l]=[f_{t(p)}^sl]$. Since $l'$ is a subpath of $l'l$, we must have that ${s'}\geq s$. Then the cancellation property gives $[f_{t(p)}^{{s'}-s}]=[l]$ and $l$ is equivalent to a composition of face-paths, contradicting the choice of $l$. This completes the proof that all cycles are equivalent to a composition of face-paths and yields the theorem.
\end{proof}

\subsection{Winding Numbers}\label{ssec:wnt}

In later sections, we will make use of the {winding number} of an (undirected) cycle around a point in a simply connected dimer model. We now set up notation and prove a lemma.

A \textit{signed arrow} $\alpha^{\varepsilon}$ of $Q$ is an arrow $\alpha$ along with a sign $\varepsilon\in\{1,-1\}$. We consider 
\[h(\alpha^\varepsilon):=\begin{cases}
h(\alpha)&\varepsilon=1\\t(\alpha)&\varepsilon=-1\end{cases} \text{ and }
t(\alpha^\varepsilon):=\begin{cases}
t(\alpha)&\varepsilon=1\\h(\alpha)&\varepsilon=-1\end{cases}.\]
A \textit{walk} on $Q$ is a string of signed arrows $p:=\alpha_m^{\varepsilon_m}\dots\alpha_1^{\varepsilon_1}$ of $Q$ such that $h(\alpha_j^{\varepsilon_j})=t(\alpha_{j+1}^{\varepsilon_{j+1}})$ for all $j\in[m-1]$. The walk $p$ is a \textit{cycle-walk} if $h(\alpha_m^{\varepsilon_m})=t(\alpha_1^{\varepsilon_1})$.  Furthermore, we write $p^{-1}=\alpha_1^{-\varepsilon_1}\dots \alpha_m^{-\varepsilon_m}$.

\begin{defn}\label{defn:winding-number}
	Let $p$ be a path in $\widetilde Q$ and let $F$ be a face of $\widetilde Q$. Let $q$ be a walk on $\widetilde Q$ from $h(p)$ to $t(p)$. We write $\wind(qp,F)$ for the winding number of the path $qp$, considered as a path on the surface $\widetilde{S(Q)}$,  around some point in the interior of $F$.
\end{defn}

\begin{lemma}\label{lem:rotation-number-formula}
    Let $p$ be a path in $\widetilde Q$ and let $q$ be a walk on $\widetilde Q$ such that $qp$ is a cycle-walk. Let $\alpha$ be a left-morphable arrow for $p$. Then for any face $F$ of $\widetilde Q$,
    \begin{equation}
        \wind(m_\alpha(p)q,F)=
            \begin{cases}
		\wind(qp,F)-1 & F\in\{F_\alpha^{cl},F_\alpha^{cc}\}\\
                \wind(qp,F) & \text{else}.
            \end{cases}
    \end{equation}
\end{lemma}
\begin{proof}
    If $F=F_\alpha^{cc}$ or $F=F_\alpha^{cl}$, then $R_\alpha^{cc}$ winds around $F$ for some positive angle $0\leq\theta\leq2\pi$, while $R_\alpha^{cl}$ winds around $F$ for an angle of $\theta-2\pi$. Left-morphing at $\alpha$ switches the former for the latter, leading to a net decrease of $2\pi$ radians. Then $\wind(m_\alpha(p)q,F)=\wind(qp,F)-1$ in this case.
    If $F$ is any other face, then $R_\alpha^{cl}$ and $R_\alpha^{cc}$ do not wind differently around $F$ and the winding number does not change.
\end{proof}

\section{Morphs and Chains}\label{sec:morphsandchains}\label{sec:actual-argument}\label{sec:mac}

In this section, we prove some technical results about basic morphs with the goal of proving that a path-consistent and simply connected dimer model has no {irreducible pairs} (Theorem~\ref{prop:no-irreducible-pairs}). This result will be used in Section~\ref{sec:pdsc} to characterize path-consistency in terms of the strand diagram of a dimer model.
We start with a definition.

In the preceding, we have used the fact that two paths $p$ and $q$ are equivalent if and only if there is a sequence of basic morphs taking $p$ to $q$. We introduce the idea of a \textit{chain} of morphable arrows that allows us to talk about sequences of morphs applied to a path in special cases.

Recall that a morphable arrow $\alpha$ for $p$ is \emph{unambiguous} if $p$ only has one subpath which is a copy of $R_\alpha^{cl}$ or $R_\alpha^{cc}$. In this case, there is a unique path $p'=m_\alpha(p)$ obtained by replacing the subpath $R_\alpha^{cl}$ with $R_\alpha^{cc}$, or vice versa.
If $\alpha_1,\dots,\alpha_r\in Q_1$ with each $\alpha_i$ an unambiguous morphable arrow for $m_{\alpha_{i-1}}\circ\dots\circ m_{\alpha_1}(p)$ for all $i\leq r$, we call the sequence $a=\alpha_r\dots\alpha_1$ a \textit{morphable chain}, or simply a \textit{chain}, for $p$. 
We introduce the notation $m_{a}(p):=m_{\alpha_r}\circ\dots\circ m_{\alpha_1}(p)$ and we say that $a$ is a chain from $p$ to $m_a(p)$. 
For some $i\in[r]$, we say that $\alpha_i$ is a left-morph (respectively right-morph) of $a$ if $\alpha_i$ is left-morphable (respectively right-morphable) for $m_{\alpha_{i-1}\dots\alpha_{1}}(p)$. Note that since $\alpha_i$ is an \textit{unambiguous} morphable arrow for this path, $\alpha_i$ is either a left-morph or a right-morph of $a$, but not both. If $\alpha_i$ is a left-morph (respectively right-morph) for all $i$, we say that $a$ is a \textit{left-chain} (respectively \textit{right-chain}). 
Two chains $a$ and $b$ of $p$ are \textit{equivalent} if $m_a(p)=m_b(p)$. 

Since we require morphable arrows of a chain to be unambiguous, it may be the case that paths $p$ and $q$ are equivalent despite there being no chain from $p$ to $q$. For example, this is true if $p$ and $q$ are equivalent but distinct and every morphable arrow for $p$ is ambiguous. In reasonable circumstances, however, the notion of a chain is often sufficient. For example, minimal paths (in path-consistent dimer models) have no unambiguous morphs, so two minimal paths are equivalent if and only if there is a chain from one to the other.

\subsection{Cycle-Removing Morphs}

In this subsection, we let $\widetilde Q$ be a path-consistent and simply connected dimer model and we define a new type of morph which weakly decreases the c-value of a path and preserves the property of being an elementary path, which we define now.

\begin{defn}\label{defn:elementary}
	A \textit{elementary} path in a dimer model $Q$ is a (possibly constant) path which is not a face-path and which contains no cycles as proper subpaths.
\end{defn}

Note that an elementary path may never contain all arrows in a given face-path. 
Then if $p$ is elementary, no morphable arrow for $p$ is in $p$. Moreover, every morphable arrow for $p$ is unambiguous. We also have the following.

\begin{defn}\label{defn:cycle-removing-morph}
    Let $p$ be an elementary path in a 
	path-consistent dimer model ${\widetilde Q}$. Let $\alpha$ be a right-morphable arrow for $p$. If $m_\alpha(p)$ is elementary, we define the \textit{cycle-removing right-morph} $\omega_\alpha(p)$ to be $m_\alpha(p)$.

    If not, write $p=p''R_\alpha^{cl}p'$ for subpaths $p'$ and $p''$ of $p$. This decomposition is unique because morphable arrows for elementary paths are unambiguous.
	Let $v_0:=h(\alpha)$ and number the vertices of $F_\alpha^{cc}$ counter-clockwise as $h(\alpha)=v_0,v_1,\dots,v_m=t(\alpha)$. Let $a$ be the largest integer less than $m$ such that $v_a\in p'$. Note that if $v_m\in p'$, then $p''$ is constant by elementariness. Let $b$ be the smallest integer greater than 0 such that $v_b\in p''$.
		
    Since $p=p''R_\alpha^{cl}p'$ is elementary, $p'$ and $p''$ do not intersect except for possibly at the endpoints $t(p'),h(p'')$ if they coincide. Then any proper subcycle of $m_\alpha(p)=p''R_\alpha^{cc}p'$ must involve some $v_i$ for $i\in\{1,\dots,m-1\}$, hence either $a>0$ or $b<m$ or both. Moreover, $a\leq b$. 
    
    Let $q'$ be the subpath of $p'$ from $t(p')$ to $v_a$. Let $R'$ be the subpath of $R_\alpha^{cc}$ from $v_a$ to $v_b$. Let $q''$ be the subpath of $p''$ from $v_b$ to $h(p'')$. 
	If $q''R'q'$ is not a face-path, define the \textit{cycle-removing right-morph} $\omega_\alpha(p)$ to be $q''R'q'$. Otherwise, define $\omega_\alpha(p)$ to be the constant path. For example, see Figures~\ref{fig:cycle-removing-example} and~\ref{fig:cycle-removing-to-constant}.

    \def\svgwidth{300pt}
    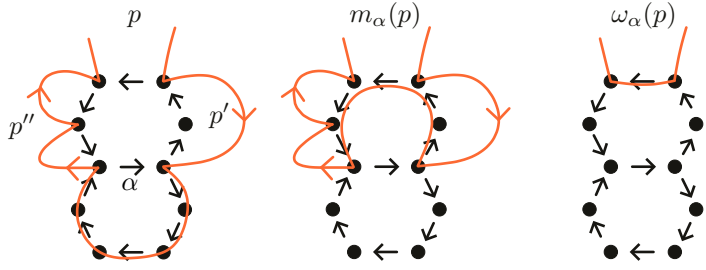
\begin{figure}
        \centering
\begingroup%
  \makeatletter%
  \providecommand\color[2][]{%
    \errmessage{(Inkscape) Color is used for the text in Inkscape, but the package 'color.sty' is not loaded}%
    \renewcommand\color[2][]{}%
  }%
  \providecommand\transparent[1]{%
    \errmessage{(Inkscape) Transparency is used (non-zero) for the text in Inkscape, but the package 'transparent.sty' is not loaded}%
    \renewcommand\transparent[1]{}%
  }%
  \providecommand\rotatebox[2]{#2}%
  \newcommand*\fsize{\dimexpr\f@size pt\relax}%
  \newcommand*\lineheight[1]{\fontsize{\fsize}{#1\fsize}\selectfont}%
  \ifx\svgwidth\undefined%
    \setlength{\unitlength}{1189.7208252bp}%
    \ifx\svgscale\undefined%
      \relax%
    \else%
      \setlength{\unitlength}{\unitlength * \real{\svgscale}}%
    \fi%
  \else%
    \setlength{\unitlength}{\svgwidth}%
  \fi%
  \global\let\svgwidth\undefined%
  \global\let\svgscale\undefined%
  \makeatother%
  \begin{picture}(1,0.3207988)%
    \lineheight{1}%
    \setlength\tabcolsep{0pt}%
    \put(0.53075578,0.31099507){\color[rgb]{0.09019608,0.08627451,0.07058824}\makebox(0,0)[lt]{\lineheight{1.25}\smash{\begin{tabular}[t]{l}$m_\alpha(p)$\end{tabular}}}}%
    \put(0.25127845,0.3105748){\color[rgb]{0.09019608,0.08627451,0.07058824}\makebox(0,0)[lt]{\lineheight{1.25}\smash{\begin{tabular}[t]{l}$p$\end{tabular}}}}%
    \put(0.35432779,0.18340186){\color[rgb]{0.09019608,0.08627451,0.07058824}\makebox(0,0)[lt]{\lineheight{1.25}\smash{\begin{tabular}[t]{l}$p’$\end{tabular}}}}%
    \put(0.86117363,0.31099507){\color[rgb]{0.09019608,0.08627451,0.07058824}\makebox(0,0)[lt]{\lineheight{1.25}\smash{\begin{tabular}[t]{l}$\omega_\alpha(p)$\end{tabular}}}}%
    \put(0.10184171,0.16949692){\color[rgb]{0.09019608,0.08627451,0.07058824}\makebox(0,0)[lt]{\lineheight{1.25}\smash{\begin{tabular}[t]{l}$p^{\prime\prime}$\end{tabular}}}}%
    \put(0,0){\includegraphics[width=\unitlength,page=1]{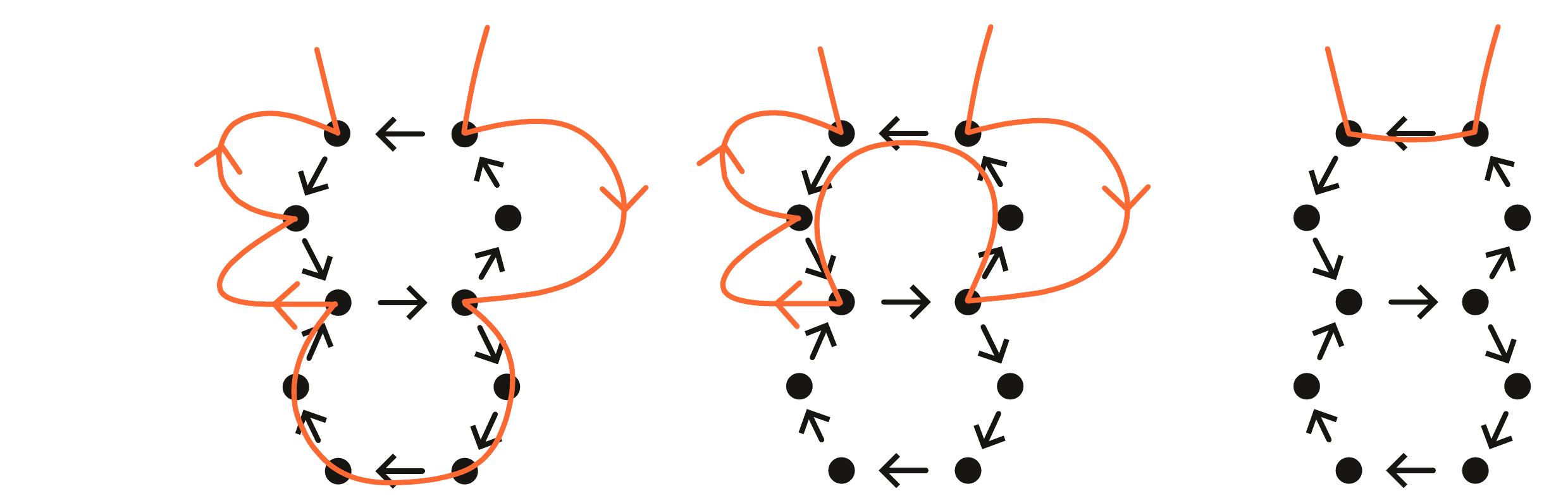}}%
    \put(0.24323677,0.09841856){\color[rgb]{0,0,0}\makebox(0,0)[lt]{\lineheight{1.25}\smash{\begin{tabular}[t]{l}$\alpha$\end{tabular}}}}%
  \end{picture}%
\endgroup%

	    \caption{An example of a cycle-removing morph.}
        \label{fig:cycle-removing-example}
    \end{figure}

	\def\svgwidth{300pt}
	\begin{figure}
		\centering
		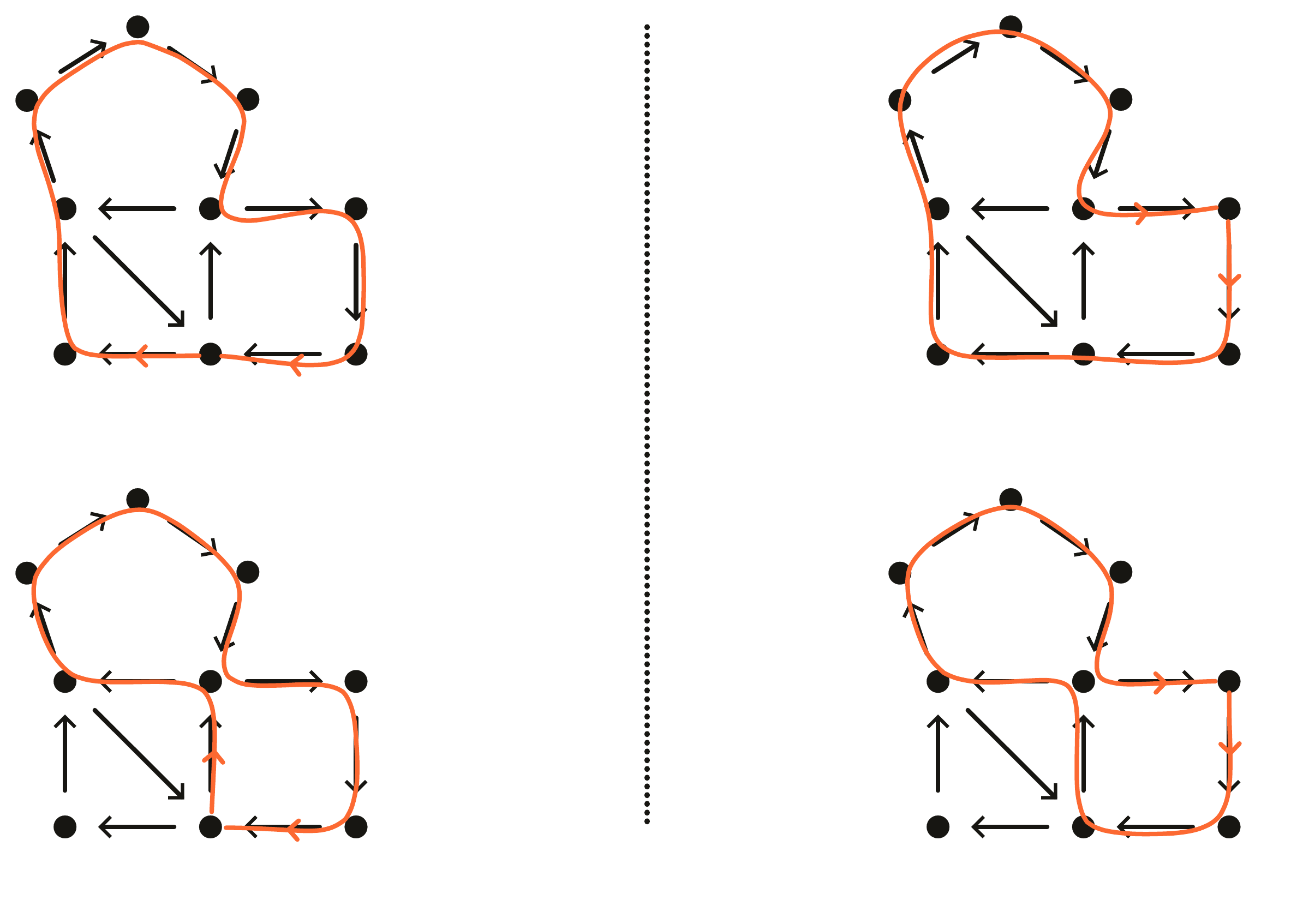
		\caption{The top left shows a clockwise cycle $p$ at $v$. The bottom left shows $m_\alpha(p)$. In the notation of Definition~\ref{defn:cycle-removing-morph}, $a=b$ and the paths $p',q',R',q''$ are all constant paths at $v$, hence $\omega_\alpha(p)$ is constant. The top right shows a clockwise cycle $q$ at a different $v$, and the bottom right shows $m_\alpha(q)$. In this case, $q''R'q'$ is the clockwise square face containing $v$, so $\omega_\alpha(q)$ is defined to be constant.}
		\label{fig:cycle-removing-to-constant}
	\end{figure}
    We similarly define \textit{cycle-removing left-morphs}.
\end{defn}

Intuitively, the cycle-removing right-morph $\omega_\alpha(p)$ is obtained by removing the proper subcycles from $m_\alpha(p)$ to get an elementary path. Since $\widetilde Q$ is path-consistent, any cycle is equivalent to a composition of face-paths, hence $m_\alpha(p)$ is equivalent to $\omega_\alpha(p)f^m$ for some $m\geq0$.

We define cycle-removing morphs only for simply connected and path-consistent dimer models because without these hypotheses, there may be cycles which are not equivalent to a composition of face-paths. Hence, we may have that $m_\alpha(p)$ is not equivalent to $\omega_\alpha(p)f^m$ for any $m\geq0$. If $Q$ is path-consistent but not simply connected, we can pass to $\widetilde Q$, do a cycle-removing morph, and pass back to $Q$. The result is that we do the corresponding basic morph and remove \textit{null-homotopic} cycles of $Q$.

If $p$ is elementary and $m_\alpha(p)$ contains a proper subcycle, we say that $\alpha$ \textit{creates a proper subcycle} of $p$. Observe that $\alpha$ creates a proper subcycle if and only if $\omega_\alpha(p)\neq m_\alpha(p)$.

\begin{lemma}\label{remk:cycle-removing-results}
    Let $p$ be an elementary path in a path-consistent quiver ${\widetilde Q}$ and let $\alpha$ be a right-morphable arrow for $p$. Then we have the following:
    \begin{enumerate}
        {\item\label{lrr:1} The cycle-removing right-morph $\omega_\alpha(p)$ is elementary.}
        {\item\label{lrr:2} The cycle-removing right-morph $\omega_\alpha(p)$ contains some arrow of $R_\alpha^{cc}$ if and only if $\omega_\alpha(p)$ is nonconstant.}
        {\item\label{lrr:3} The arrow $\alpha$ creates a proper subcycle if and only if $\omega_\alpha(p)$ does not contain all of $R_\alpha^{cc}$.}
    \end{enumerate}
\end{lemma}
\begin{proof}
	Parts~\eqref{lrr:1} and~\eqref{lrr:3} follow from the definitions. We prove~\eqref{lrr:2}. 
	Certainly if $\omega_\alpha(p)$ is constant then it contains no arrow of $R_\alpha^{cc}$. On the other hand, if $\omega_\alpha(p)$ is nonconstant, then in the notation of Definition~\ref{defn:cycle-removing-morph} we must have $a<b$. Then some arrow of $R_\alpha^{cc}$ is contained in $R'$, and hence is contained in $\omega_\alpha(p)=q''R'q'$.
\end{proof}

\begin{remk}
    Many of the definitions and results appearing above as well as later in the text are symmetric. If one switches ``left'' for ``right'' and ``clockwise'' for ``counter-clockwise'' in the statements and proofs, the analogous arguments and results hold. We will refer to these as dual results without stating them separately.
\end{remk}

\subsection{Left, Right, Good, and Bad}

For the remainder of Section~\ref{sec:morphsandchains} we assume that $\widetilde{S(Q)}$ is not a sphere.  We now define a notion of one path being to the right of another.
We obtain some conditions under which cycle-removing morphs behave well with respect to this concept of left and right.

\begin{defn}\label{defn:left-right-disks-stuff}
	Suppose $p$ and $q$ are elementary paths in a simply connected dimer model ${\widetilde Q}$ which is not on a sphere with $t(p)=t(q)$ and $h(p)=h(q)$. We say that $p$ is \textit{to the right of} $q$ if the following conditions are satisfied. 
	\begin{enumerate}
		\item The shared vertices of $p$ and $q$ may be ordered $v_1,\dots,v_m$ such that $v_i$ is the $i$th vertex among $\{v_1,\dots,v_m\}$ to appear in $p$ and is the $i$th such vertex to appear in $q$. We remark that if $p$ is an elementary cycle and $q$ is trivial, then $m=2$ and $v_1=v_2$.
		\item For all $i\in[m-1]$, if $p_i$ (respectively $q_i$) is the subpath of $p$ (respectively $q$) from $v_i$ to $v_{i+1}$, then either $p_i$ and $q_i$ are the same arrow or $q_i^{-1}p_i$ is a counter-clockwise (necessarily simple) cycle-walk. In the latter case, we say that $p$ and $q$ bound the disk $q_i^{-1}p_i$.
	\end{enumerate}
	See Figure~\ref{ex:disksbounded}. We remark that if $q$ is constant and $p$ is an elementary counter-clockwise cycle, then $p$ is to the right of $q$ with $m=2$ and $v_1=v_2$. Similarly, an elementary clockwise cycle at $v$ is to the left of the constant path at $v$.
\end{defn}

We warn the reader that Definition~\ref{defn:left-right-disks-stuff} does not form a partial order on paths in $\widetilde Q$ with the same start and end vertices as the relation is not transitive. See Figure~\ref{fig:lr-not-trans} for an example.

\def\svgwidth{200pt}
\begin{figure}[htbp]
    \centering
\begingroup%
  \makeatletter%
  \providecommand\color[2][]{%
    \errmessage{(Inkscape) Color is used for the text in Inkscape, but the package 'color.sty' is not loaded}%
    \renewcommand\color[2][]{}%
  }%
  \providecommand\transparent[1]{%
    \errmessage{(Inkscape) Transparency is used (non-zero) for the text in Inkscape, but the package 'transparent.sty' is not loaded}%
    \renewcommand\transparent[1]{}%
  }%
  \providecommand\rotatebox[2]{#2}%
  \newcommand*\fsize{\dimexpr\f@size pt\relax}%
  \newcommand*\lineheight[1]{\fontsize{\fsize}{#1\fsize}\selectfont}%
  \ifx\svgwidth\undefined%
    \setlength{\unitlength}{856.22900391bp}%
    \ifx\svgscale\undefined%
      \relax%
    \else%
      \setlength{\unitlength}{\unitlength * \real{\svgscale}}%
    \fi%
  \else%
    \setlength{\unitlength}{\svgwidth}%
  \fi%
  \global\let\svgwidth\undefined%
  \global\let\svgscale\undefined%
  \makeatother%
  \begin{picture}(1,0.5682265)%
    \lineheight{1}%
    \setlength\tabcolsep{0pt}%
    \put(0,0){\includegraphics[width=\unitlength,page=1]{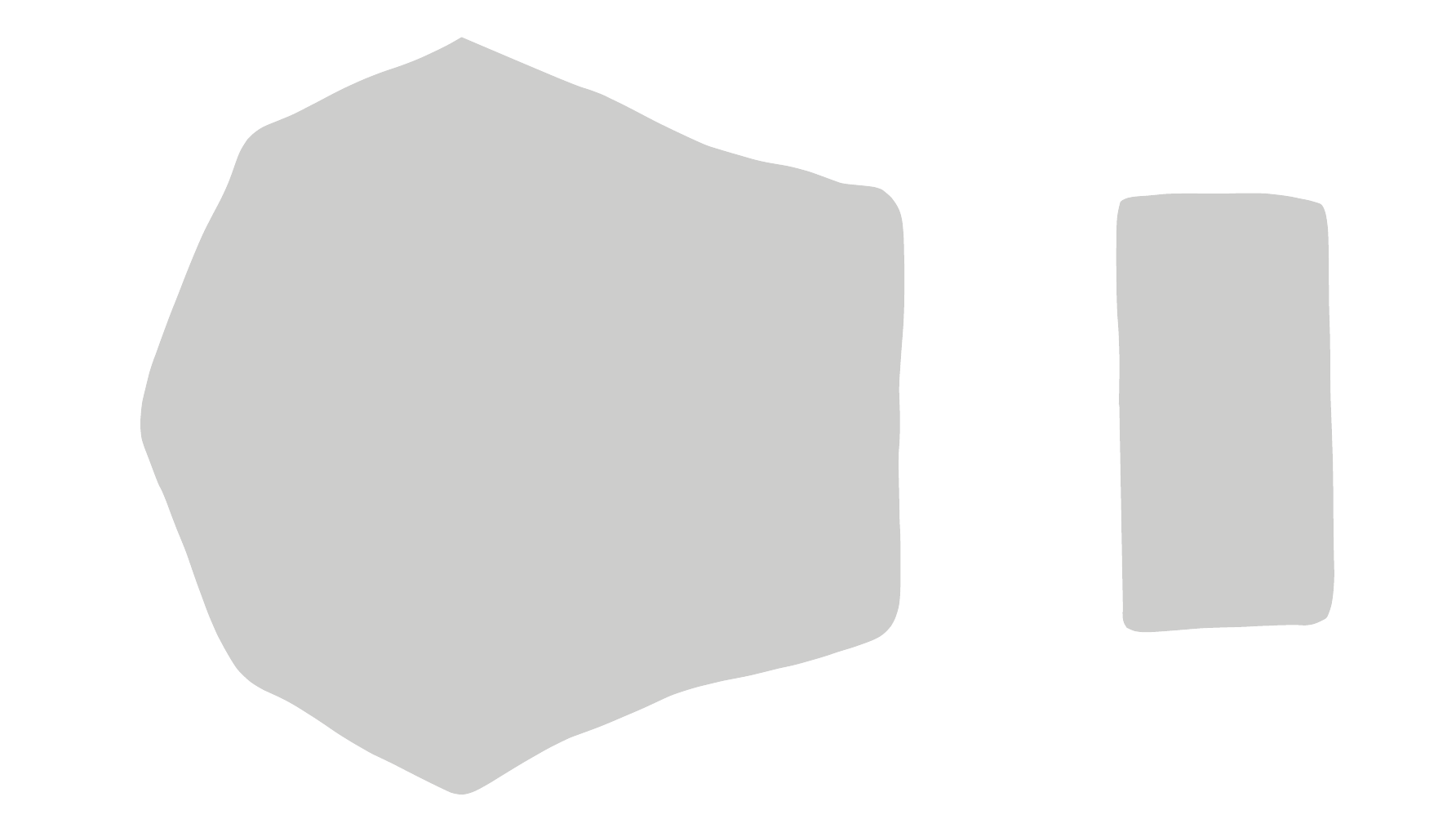}}%
    \put(0.47481343,0.52233166){\color[rgb]{0.09019608,0.08627451,0.07058824}\makebox(0,0)[lt]{\lineheight{1.25}\smash{\begin{tabular}[t]{l}$p$\end{tabular}}}}%
    \put(0.47446385,0.04900092){\color[rgb]{0.09019608,0.08627451,0.07058824}\makebox(0,0)[lt]{\lineheight{1.25}\smash{\begin{tabular}[t]{l}$q$\end{tabular}}}}%
    \put(0.94000226,0.32410945){\color[rgb]{0.09019608,0.08627451,0.07058824}\makebox(0,0)[lt]{\lineheight{1.25}\smash{\begin{tabular}[t]{l}$t(p)$\end{tabular}}}}%
    \put(0.00980345,0.31915178){\color[rgb]{0.09019608,0.08627451,0.07058824}\makebox(0,0)[lt]{\lineheight{1.25}\smash{\begin{tabular}[t]{l}$h(p)$\end{tabular}}}}%
    \put(0,0){\includegraphics[width=\unitlength,page=2]{disks_bounde.pdf}}%
  \end{picture}%
\endgroup%

    \caption{The paths $p$ and $q$ bound two disks and $p$ is to the right of $q$.}
	\label{ex:disksbounded}
\end{figure}

\def\svgwidth{100pt}
\begin{figure}	
	\centering
\begingroup%
  \makeatletter%
  \providecommand\color[2][]{%
    \errmessage{(Inkscape) Color is used for the text in Inkscape, but the package 'color.sty' is not loaded}%
    \renewcommand\color[2][]{}%
  }%
  \providecommand\transparent[1]{%
    \errmessage{(Inkscape) Transparency is used (non-zero) for the text in Inkscape, but the package 'transparent.sty' is not loaded}%
    \renewcommand\transparent[1]{}%
  }%
  \providecommand\rotatebox[2]{#2}%
  \newcommand*\fsize{\dimexpr\f@size pt\relax}%
  \newcommand*\lineheight[1]{\fontsize{\fsize}{#1\fsize}\selectfont}%
  \ifx\svgwidth\undefined%
    \setlength{\unitlength}{559.41601562bp}%
    \ifx\svgscale\undefined%
      \relax%
    \else%
      \setlength{\unitlength}{\unitlength * \real{\svgscale}}%
    \fi%
  \else%
    \setlength{\unitlength}{\svgwidth}%
  \fi%
  \global\let\svgwidth\undefined%
  \global\let\svgscale\undefined%
  \makeatother%
  \begin{picture}(1,0.93513416)%
    \lineheight{1}%
    \setlength\tabcolsep{0pt}%
    \put(0,0){\includegraphics[width=\unitlength,page=1]{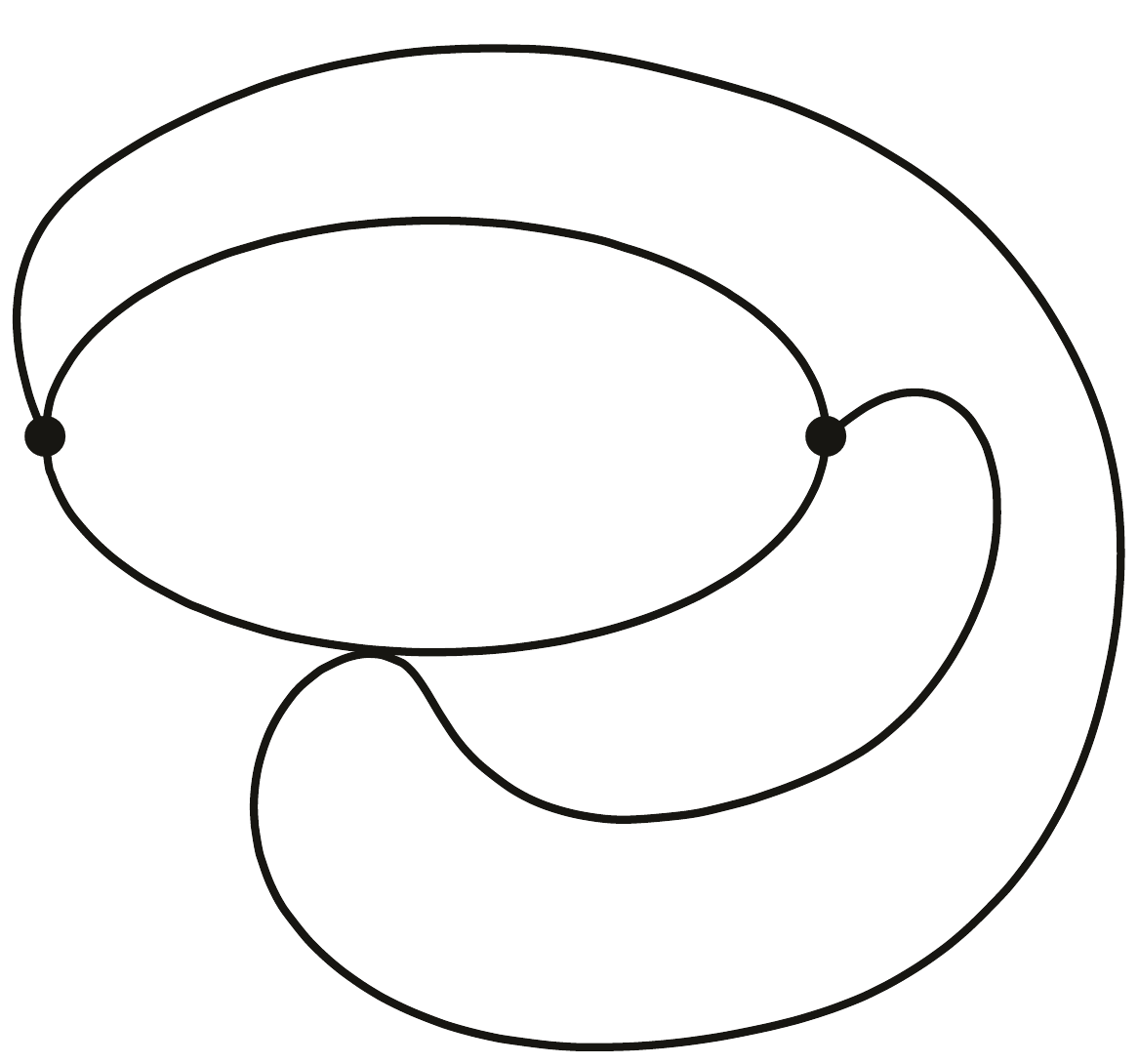}}%
    \put(0.47866166,0.42643225){\color[rgb]{0.09019608,0.08627451,0.07058824}\makebox(0,0)[lt]{\lineheight{1.25}\smash{\begin{tabular}[t]{l}$p$\end{tabular}}}}%
    \put(0.47399501,0.65720862){\color[rgb]{0.09019608,0.08627451,0.07058824}\makebox(0,0)[lt]{\lineheight{1.25}\smash{\begin{tabular}[t]{l}$q$\end{tabular}}}}%
    \put(0.74349462,0.7251366){\color[rgb]{0.09019608,0.08627451,0.07058824}\makebox(0,0)[lt]{\lineheight{1.25}\smash{\begin{tabular}[t]{l}$r$\end{tabular}}}}%
    \put(0,0){\includegraphics[width=\unitlength,page=2]{lrntX.pdf}}%
  \end{picture}%
\endgroup%

	\caption{The path $q$ is to the right of $p$ and $r$ is to the right of $q$, but $r$ is not to the right of $p$.}
\label{fig:lr-not-trans}
\end{figure}

\begin{remk}
	Definition~\ref{defn:left-right-disks-stuff} relies on the notion of a simple cycle-walk being clockwise or counter-clockwise. This is only well-defined when the surface is not a sphere. In the following, we will prove that path-consistency implies strand-consistency while making heavy use of the notions of left and right as well as those of clockwise and counter-clockwise, culminating in the proof of
Proposition~\ref{thm:not-a-cons-then-not-g-conss}. 
	These arguments require $\widetilde Q$ not to be on a sphere, and indeed dimer models on spheres may not be strand-consistent but may be cancellative, so Proposition~\ref{thm:not-a-cons-then-not-g-conss} does not hold in this case. 
\end{remk}

\begin{defn}\label{defn:rightmost-leftmost-paths}
    A path $p$ is a \textit{rightmost path} (respectively \textit{leftmost path}) there are no right-morphable (respectively left-morphable) arrows for $p$.
\end{defn}

\begin{defn}\label{defn:irreducible-pair}
	An \textit{irreducible pair} is a pair of paths $(p,q)$ in $\widetilde Q$ such that $q^{-1}p$ is a simple counter-clockwise cycle-walk, $p$ is leftmost, and $q$ is rightmost.
\end{defn}

If $(p,q)$ is an irreducible pair, then $p$ is to the right of $q$.
Note that $p$ or $q$ may be constant. If this is true, then the other path in the pair cannot be a face-path by elementariness.
The notion of an irreducible pair appears in~\cite{XBocklandt2011} when $Q$ is a dimer model on a torus. 
We now make preparations to prove that irreducible pairs may not occur in path-consistent dimer models.

\begin{defn}
	Let $p$ be an elementary path of $\widetilde Q$ and let $\alpha$ be a morphable arrow for $p$. We say that $\alpha$ is \textit{good} if $\alpha$ is a left-morphable arrow or if $\alpha$ is a right-morphable arrow such that $m_\alpha(p)$ does not contain a proper counter-clockwise subcycle. If $\alpha$ is a right-morphable arrow for $p$ such that $m_\alpha(p)$ contains a proper counter-clockwise subcycle, then we say that $\alpha$ is \textit{bad}.
	A cycle-removing chain $a=\alpha_r\dots\alpha_1$ of $p$ is \textit{good} if each $\alpha_i$ is a good morphable arrow for $\omega_{\alpha_{i-1}\dots\alpha_{1}}(p)$. Otherwise, $a$ is \textit{bad}. 
\end{defn}

There is an asymmetry to our definitions of good and bad cycle-removing morphs. This is because we intend to apply them only to the leftmost path $p$ in an irreducible pair $(p,q)$. Our strategy will be to obtain a good cycle-removing chain taking $p$ to a minimal path $p'$, where $q^{-1}p$ is contained in the area enclosed by the simple counter-clockwise cycle-walk $q^{-1}p'$. Dually, one obtains a minimal path $q'$ so that $(q')^{-1}p'$ is a simple counter-clockwise cycle-walk, which we show to be impossible in the path-consistent case.

Good morphs are useful because a good cycle-removing morph behaves reasonably well with respect to the notion of one path being to the left or right of another.
Consider the following.
\begin{lemma}\label{lem:right-left}
	{ \renewcommand\labelenumi{(\theenumi)}
	\begin{enumerate}
		\item\label{lrl:1} Let $q$ be an elementary path in $\widetilde Q$ and let $p$ be any elementary path to the right of $q$. \label{rlrrl1} Let $\beta$ be a left-morphable arrow for $p$ which is not a left-morphable arrow for $q$. Then $\omega_\beta(p)$ is to the right of $q$.
		\item\label{lrl:2} Let $p$ be an elementary counter-clockwise cycle in $\widetilde Q$ and let $\alpha$ be a right-morphable arrow for $p$ such that $m_\alpha(p)$ has no proper counter-clockwise subcycle. Then $p$ is contained in the area enclosed by $\omega_\alpha(p)$.
	\end{enumerate}}
\end{lemma}
\begin{proof}
	To see~\eqref{lrl:1}, it suffices to reduce to the case where $p$ and $q$ share only their start and end vertices. In this case, they bound one disk $q^{-1}p$ and a 
	left-morph of $p$ at $\beta$ results in a path contained in the area enclosed by $q^{-1}p$.

	We now prove~\eqref{lrl:2}. 
	We show that $\wind(\omega_\alpha(p),F)\geq\wind(p,F)$ for any face $F$. It follows that $\omega_\alpha(p)$ is also an elementary counter-clockwise cycle. Since any face $F$ is in the interior of $p$ (respectively $\omega_\alpha(p)$) if and only if $\wind(p,F)>0$ (respectively $\wind(\omega_\alpha(p),F)>0$), it further follows that if $F$ is in the interior of $p$ then $F$ is in the interior of $\omega_\alpha(p)$ and the statement is proven.

	Let $p$ be an elementary counter-clockwise cycle in $\widetilde Q$ and let $\alpha$ be a right-morphable arrow for $p$ such that $m_\alpha(p)$ has no proper counter-clockwise subcycle. By Lemma~\ref{lem:rotation-number-formula}, $\wind(m_\alpha(p),F)\geq\wind(p,F)$. The path $\omega_\alpha(p)$ is obtained from $m_\alpha(p)$ by deleting some number of proper elementary subcycles. All of these are clockwise, and hence have a winding number less than or equal to zero around $F$, by assumption. It follows that their deletion can only increase the winding number around $F$ and we have
	\[\wind(\omega_\alpha(p),F)\geq\wind(m_\alpha(p),F)\geq\wind(p,F).\]

		If $F$ is enclosed within $p$, then $\wind(p,F)=1$ and the above inequality forces $\wind(\omega_\alpha(p),F)=1$. It follows that $F$ is enclosed within $\omega_\alpha(p)$. This ends the proof.
\end{proof}

Note that the conditions of~\eqref{lrl:1} and~\eqref{lrl:2} of Lemma~\ref{lem:right-left} necessitate that the morphable arrows considered are good.
Some caution must be shown when considering whether cycle-removing morphs move paths to the right or left, particularly when the morphs are {bad}.
{For example, see the bad right-morph of Figure~\ref{fig:right-creates-counterclockwise}. On the left is a path $p$ with a right-morphable arrow $\alpha$, and on the right is the path $\omega_\alpha(p)$. Note that $p$ is not contained in the area enclosed by $\omega_\alpha(p)$, justifying the limited scope of Lemma~\ref{lem:right-left}~\eqref{lrl:2}.}

\def\svgwidth{200pt}
\begin{figure}[htbp]
    \centering
	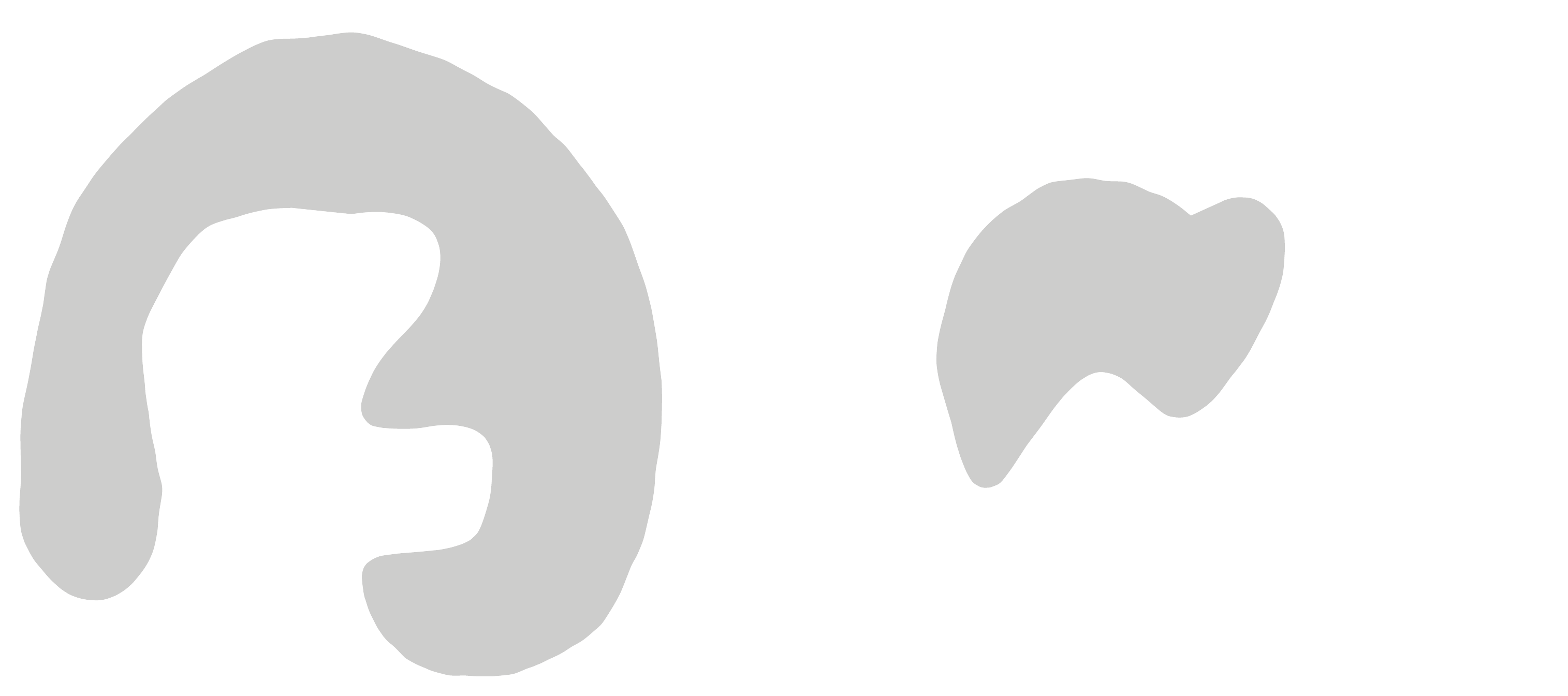
	\caption{On the left is a counter-clockwise cycle $p$ at $v$. Cycle-removing right-morphing $p$ at the arrow $\alpha$ results in a \textit{clockwise} cycle which does not enclose $p$.}
	\label{fig:right-creates-counterclockwise}
\end{figure}

\subsection{Irreducible Pairs}

The goal of this section is to show that {irreducible pairs} (Definition~\ref{defn:irreducible-pair}) cannot appear in simply connected path-consistent dimer models. We begin with some technical lemmas about cycle-removing morphs.

\begin{lemma}\label{lem:cycleless-switcheroo}
	Let $p$ be an elementary path in a path-consistent quiver $\widetilde Q$. Let $\alpha$ be a right-morphable arrow for $p$ and let $\beta$ be a left-morphable arrow for $\omega_\alpha(p)$ distinct from $\alpha$. Then $\beta$ is a left-morphable arrow for $p$.
\end{lemma}
\begin{proof}
	Any subpath of $\omega_\alpha(p)$ which is not a subpath of $p$ must contain some arrow $\gamma$ of $R_\alpha^{cc}$. Since $\omega_\alpha(p)$ does not contain $\alpha$, the only counter-clockwise return path which $\gamma$ could be a part of is $R_\alpha^{cc}$. This shows that $\omega_\alpha(p)$ contains no counter-clockwise return paths which are not in $p$, other than possibly $R_\alpha^{cc}$. The result follows.
\end{proof}

\begin{defn}
	We say that a (cycle-removing) chain $a=\alpha_r\dots\alpha_1$ of an elementary path $p$ is \textit{crossing-creating} if $m_{\alpha_i\dots\alpha_1}(p)$ is elementary for $i<r$ but $m_a(p)$ is not elementary.
\end{defn}

If a path $p$ of a path-consistent dimer model is not minimal then there must be a crossing-creating chain of $p$. Our goal is now to show that there must be a \textit{good} crossing-creating chain of $p$ (Proposition~\ref{lem:cycle-removing-no-right-cc-cycle}). Our first step is Lemmas~\ref{lem:right-no-creates-cc-left-ccc-cycle-version} and~\ref{lem:right-creates-cc-left-ccc-cycle-version}. These show that an elementary path must have a crossing-creating left-chain under certain technical conditions.

\begin{lemma}\label{lem:right-no-creates-cc-left-ccc-cycle-version}
	Let $p$ be an elementary counter-clockwise cycle in $\widetilde Q$. Let $\alpha$ be a right-morphable arrow for $p$ such that $m_\alpha(p)$ has no proper counter-clockwise subcycle. Suppose there is a crossing-creating left-chain of $\omega_\alpha(p)$. Then there is a crossing-creating left-chain of $p$.
\end{lemma}
\begin{proof}
	Let $b=\beta_s\dots\beta_1$ be a crossing-creating left-chain of $\omega_\alpha(p)$.
	Let $j$ be maximal such that $b':=\beta_j\dots\beta_1$ is a left-chain of $p$ as well as $\omega_\alpha(p)$. If $b'$ is crossing-creating for $p$, then we are done, so suppose $b'$ is not a crossing-creating chain of $p$.

	Suppose first that $b$ is a left-chain of $p$, and hence that $j=s$ and $b'=b$. Since $b$ creates a crossing of $\omega_\alpha(p)$ but not of $p$, there must be a proper subcycle of $m_{b}(\omega_\alpha(p))$ starting at a vertex $v$ of $\omega_\alpha(p)$ which is not a vertex of $p$. Then left-morphing $m_{\beta_{s-1}\dots\beta_{1}}(p)$ at $\beta_s$ must add $v$ to $p$. By Lemma~\ref{lem:right-left}~\eqref{lrl:2}, the area bounded by $p$ is contained in the area bounded by $\omega_\alpha(p)$. In particular, vertices of $\omega_\alpha(p)$ which are not vertices of $p$ are not in the area bounded by $p$. 
	Since left-morphing $p$ at $b$ does not create any crossings, $p=\omega_b(\omega_b(p))$ is obtained by applying the (cycle-removing) right-chain $b$ to $\omega_b(p)$. Then Lemma~\ref{lem:right-left}~\eqref{lrl:2} applied to $\omega_b(p)$ shows that $\omega_b(p)$ is strictly contained in the area bounded by $p$. Since $v$ is not in the area bounded by $p$, the vertex $v$ cannot be in $\omega_b(p)$. This is a contradiction.

	It follows that $b$ is not a left-chain of $p$ and hence that $j<s$. 
	We claim that $R_\alpha^{cl}\in m_{\beta_{j'}\dots\beta_{1}}(p)$ and that no arrow of $F_\alpha^{cl}$ is in $m_{\beta_{j'}\dots\beta_{1}}(\omega_\alpha(p))$ for any $1\leq j'\leq j$.
	Since $R_\alpha^{cl}\in p$ and $\beta_1$ does not create a crossing of $p$, we must have $\beta_1\not\in F_\alpha^{cl}$. Since no arrow of $F_\alpha^{cl}$ is in $\omega_\alpha(p)$ and $R_{\beta_1}^{cc}\in\omega_\alpha(p)$, for any $\alpha'\in F_\alpha^{cl}$ we must have $\beta_1\not\in F_{\alpha'}^{cc}$. It follows that $R_\alpha^{cl}$ is in $m_{\beta_1}(p)$ and that no arrow of $F_\alpha^{cl}$ is in $m_{\beta_1}(\omega_\alpha(p))$. We repeat this argument to see that $R_\alpha^{cl}\in m_{\beta{j'}\dots\beta_{1}}(p)$ and that no arrow of $F_\alpha^{cl}$ is in $m_{\beta{j'}\dots\beta_{1}}(\omega_\alpha(p))$ for any $1\leq j'\leq j$, proving the claim.

	First, we suppose that $m_\alpha(p)$ is not elementary and hence that $\omega_\alpha(p)$ does not contain all of $R_\alpha^{cc}$. In particular, either the arrow $\alpha'$ following $\alpha$ in $R_\alpha^{cc}$ or the arrow $\alpha''$ preceding $\alpha$ in $R_\alpha^{cc}$ (or both) are absent from $\omega_\alpha(p)$. Suppose that $\alpha'$ is not in $\omega_\alpha(p)$; the $\alpha''$ case is the same. 
	Since $b'$ is a left-chain of $p$ and $R_\alpha^{cl}$ is in $m_{\beta_{j'}\dots\beta_{1}}(p)$  by the claim above, the arrows $\alpha'$ and $\alpha$ may not be added by any morph in $b'$ without creating a crossing in $p$, which would contradict our choice of $b'$. Hence, neither $\alpha'$ nor $\alpha$ is in $m_{\beta_j\dots\beta_1}(\omega_\alpha(p))$. In particular, no return path of an arrow in $F_\alpha^{cc}$ is a subpath of $m_{\beta_j\dots\beta_1}(\omega_\alpha(p))$. 
	Since $\beta_{j+1}$ is a left-morphable arrow for $m_{b'}(\omega_\alpha(p))$ but not of $m_{b'}(p)$, it must be the case that $R_{\beta_{j+1}}^{cc}$ is a subpath of $m_{b'}(\omega_\alpha(p))$ but not of $m_{b'}(p)$. Since the only such subpaths contain some arrow of $R_\alpha^{cc}$, we must have that $\beta_{j+1}\in F_\alpha^{cc}$. This contradicts the fact that no return path in $F_\alpha^{cc}$ is a subpath of $m_{\beta_j\dots\beta_1}(\omega_\alpha(p))$.

	On the other hand, suppose that $m_\alpha(p)$ is elementary, and hence that $\omega_\alpha(p)=m_\alpha(p)$. 
		As above, the arrow $\beta_{j+1}$ must be in $F_\alpha^{cc}$. 
		If $\beta_{j+1}\neq\alpha$, then $\alpha\in m_{\beta_j\dots\beta_1}(\omega_\alpha(p))$, contradicting the claim in the third paragraph, so $\beta_{j+1}=\alpha$.
	Consider the chain $\beta_{j+1}\beta_j\dots\beta_1\alpha$ 
	of $p$. The left-morph $\beta_{j+1}$ cancels out the right-morph $\alpha$ in this chain, hence the left-chain $\beta_{j+2}\beta_j\dots\beta_1$ 
	of $p$ is equivalent to $\beta_{j+2}\dots\beta_2\beta_1\alpha$. 
	Then $\beta_s\dots\beta_{j+2}\beta_j\dots\beta_1$ 
	is a crossing-creating left-chain of $p$. This completes the proof.

\end{proof}

\begin{lemma}\label{lem:right-creates-cc-left-ccc-cycle-version}
	Let $p$ be an elementary 
	path in $\widetilde Q$. Let $\alpha$ be a right-morphable arrow for $p$ and let $l$ be a proper elementary counter-clockwise subcycle of $m_\alpha(p)$. Suppose there is a crossing-creating left-chain of $l$. Then there is a crossing-creating left-chain of $p$.
\end{lemma}
\begin{proof}
	The proof is the same as that of Lemma~\ref{lem:right-no-creates-cc-left-ccc-cycle-version}, with $l$ taking the place of $\omega_\alpha(p)$.
\end{proof}

\begin{lemma}\label{lem:right-morph-no-cc-facepath}
	Let $p$ be an elementary path and let $\alpha$ be a right-morphable arrow for $p$. Then $m_\alpha(p)$ does not contain a counter-clockwise face-path.
\end{lemma}
\begin{proof}
	Any arrow added by right-morphing at $\alpha$ is only a part of one counter-clockwise face, which is $F_\alpha^{cc}$. Since $p$ is elementary, $\alpha$ is not in $p$ and hence is not in $m_\alpha(p)$. The result follows.
\end{proof}

Before finally proving that any elementary path has a good crossing-creating chain (Proposition~\ref{lem:cycle-removing-no-right-cc-cycle}), we first prove the special case that any counter-clockwise elementary cycle has a good crossing-creating chain.

\begin{prop}\label{prop:simple-cc-cycle-c-left-chain}
	Let $p$ be a counter-clockwise elementary cycle in a path-consistent dimer model $\widetilde Q$. Then $p$ has a crossing-creating left-chain.
\end{prop}
\begin{proof}
	We induct on the c-value of $p$. 
	For the base case, let $p$ be an elementary counter-clockwise cycle with a minimal c-value among elementary counter-clockwise cycles.
	By path-consistency, $p$ must be equivalent to a composition of face-paths and hence must have a crossing-creating chain. If $p$ has a good crossing-creating chain $a=\alpha_r\dots\alpha_1$ such that $\alpha_r$ is a right-morph, then $\omega_a(p)$ is an elementary counter-clockwise cycle by Lemma~\ref{lem:right-left}~\eqref{lrl:2}, and necessarily has a lower c-value than $p$, contradicting our choice of $p$.
	Suppose $p$ has a good crossing-creating chain $a=\alpha_r\dots\alpha_1$ such that $\alpha_r$ is a left-morph. If $a$ is a left-chain, then there is nothing to show; otherwise, let $j$ be maximal such that $\alpha_j$ is a right-morph. Applying Lemma~\ref{lem:right-no-creates-cc-left-ccc-cycle-version} to $m_{\alpha_1\dots\alpha_{j-1}}(p)$ shows that there is a crossing-creating left-chain of $m_{\alpha_1\dots\alpha_{j-1}}(p)$. Repeating this process for each right-morph of $a$ gives that there is a crossing-creating left-chain of $p$.
	If $p$ has a bad crossing-creating chain $a=\alpha_r\dots\alpha_1$, then $\alpha_r$ is a right-morph and $m_a(p)$ contains a proper subpath $l$ which is a simple counter-clockwise cycle. By Lemma~\ref{lem:right-morph-no-cc-facepath}, $l$ is not a face-path and hence is elementary. This contradicts our choice of $p$, since $l$ has a strictly lower c-value than $p$. This completes the proof of the base case.

	Now suppose that $p$ is an elementary counter-clockwise cycle which does not have a minimal c-value.
	There must be a crossing-creating chain $c=\gamma_r\dots\gamma_1$ of $p$. We first show that there must be a crossing-creating chain $ab$ of $p$ where $b$ is a left-chain of $p$ and $a$ is a (possibly empty) right-chain of $m_b(p)$. Suppose $c$ is not already of this form and let $j$ be minimal such that $\gamma_{j-1}$ is a right-morph in $c$ and $\gamma_j$ is a left-morph in $c$. 

	If $\gamma_{j-1}$ and $\gamma_j$ are the same arrow, then $\gamma_j$ cannot create a crossing and the term $\gamma_{j}\gamma_{j-1}$ may be removed from $c$ to get an equivalent chain $\gamma_r\dots\gamma_{j+1}\gamma_{j-2}\dots\gamma_1$. 
	If $\gamma_{j-1}$ and $\gamma_j$ are not the same arrow, then by Lemma~\ref{lem:cycleless-switcheroo},
	$\gamma_j$ is a left-morphable arrow for 
	$m_{\gamma_{j-2}\dots\gamma_1}(p)=\omega_{\gamma_{j-2}\dots\gamma_1}(p)$. If this left-morph creates a crossing, then let 
	$c_1:=\gamma_j\gamma_{j-2}\dots\gamma_1$. Otherwise, note that $\gamma_j\gamma_{j-1}$ and $\gamma_{j-1}\gamma_{j}$ are equivalent chains of 
	$m_{\gamma_{j-2}\dots\gamma_1}(p)$, hence 	
	\[c_1:=\gamma_r\dots\gamma_{j+1}\gamma_{j-1}\gamma_j\gamma_{j-2}\dots\gamma_1\]
	is a crossing-creating chain of $p$ equivalent to $c$.
	We apply the above logic repeatedly to move all left-morphs in $c$ to the front. We end up with some crossing-creating chain $c_m=ab$ where $b$ is a left-chain of $p$ and $a$ is a right-chain of $m_b(p)$. If $a$ is trivial, then $b$ is a crossing-creating left-chain and we are done. 
	Otherwise, let $a=\alpha_r\dots\alpha_1$.
	Lemma~\ref{lem:right-creates-cc-left-ccc-cycle-version} or Lemma~\ref{lem:right-no-creates-cc-left-ccc-cycle-version} (depending on whether $m_{ba}(p)$ contains a proper counter-clockwise subcycle) along with the induction hypothesis shows that 
	$m_{\alpha_{r-1}\dots\alpha_1b}(p)$ has a crossing-creating left-chain. We now repeatedly apply Lemma~\ref{lem:right-no-creates-cc-left-ccc-cycle-version} to see that $m_b(p)$ has a crossing-creating left-chain $b'$. Then $b'b$ is a crossing-creating left-chain of $p$.

\end{proof}

\begin{prop}\label{lem:cycle-removing-no-right-cc-cycle}
	Let $p$ be an elementary path in a path-consistent dimer model $\widetilde Q$. Then there exists a good cycle-removing chain from $p$ to a minimal path.
\end{prop}
\begin{proof}
	We prove by induction on the c-value of $p$. The base case when $p$ is minimal is trivial. 
	Suppose the result has been shown for paths with a strictly lower c-value than that of some non-minimal path $p$. We first show that $p$ has a good crossing-creating chain.
	Since $p$ is not minimal, there is some crossing-creating chain $a=\alpha_r\dots\alpha_1$ of $p$. If $\alpha_r$ is a left-morph or if $m_a(p)$ contains no proper counter-clockwise subcycle, then $a$ is a good crossing-creating chain of $p$. If not, then $m_a(p)$ contains a proper counter-clockwise simple subcycle $l$, which is not a face-path by Lemma~\ref{lem:right-morph-no-cc-facepath}. By Proposition~\ref{prop:simple-cc-cycle-c-left-chain}, $l$ has a crossing-creating left-chain. By Lemma~\ref{lem:right-creates-cc-left-ccc-cycle-version}, $m_{\alpha_{r-1}\dots\alpha_{1}}(p)$ has a crossing-creating left-chain $a'$. Then $a'\alpha_{r-1}\dots\alpha_1$ is a good crossing-creating chain of $p$. Then we may suppose that $p$ has a good crossing-creating chain $a$. Since $\omega_a(p)$ is an elementary path with a lower c-value than $p$, by the induction hypothesis there is a good cycle-removing chain $a''$ from $\omega_a(p)$ to a minimal path. Then $a''a$ is a good cycle-removing chain from $p$ to a minimal path.

\end{proof}

We are now ready to prove the main result of this section.

\begin{lemma}\label{lem:ffl}
	Let $(p,q)$ be an irreducible pair of a path-consistent dimer model $\widetilde Q$. Let $r$ be an elementary path from $t(p)$ to $h(p)$ which does not enter the interior of $q^{-1}p$ and let $\alpha$ be a morphable arrow for $r$. Then $m_\alpha(r)$ does not enter the interior of $q^{-1}p$.
\end{lemma}
\begin{proof}
	Suppose to the contrary that $m_\alpha(r)$ enters the interior of $q^{-1}p$. Since $p$ is leftmost and $q$ is rightmost, it must be the case that $\alpha$ is a right-morph at an arrow of $p$ or a left-morph at an arrow of $q$. Suppose the former is true; the latter case is symmetric. The segment $r'$ of $r$ from $t(p)$ to $h(\alpha)$ either winds right or left around $q^{-1}p$. See Figure~\ref{fig:fffff}. In either case, as can be seen from Figure~\ref{fig:fffff}, the path $R_\alpha^{cl}r'$ may not be completed to a path from $t(p)$ to $h(p)$ without entering the interior of $q^{-1}p$ or causing a self-intersection.
	\def\svgwidth{300pt}
	\begin{figure}
		\centering
		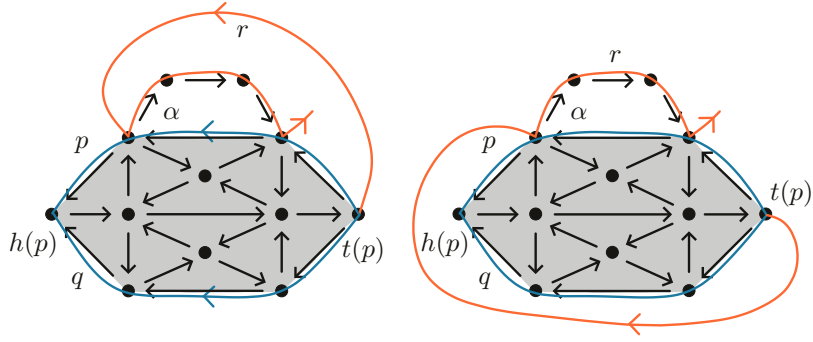
		\caption{On the left, $r$ winds to the right from $t(p)$ to $t(\alpha)$. On the left, $r$ winds to the left from $t(p)$ to $h(\alpha)$. In both cases, there is no way to complete the beginning of $r$ (pictured) to a path from $t(p)$ to $h(p)$ without breaking elementariness or entering $q^{-1}p$.}
		\label{fig:fffff}
	\end{figure}
\end{proof}

\begin{thm}\label{prop:no-irreducible-pairs}
	A simply connected path-consistent dimer model, that is not on a sphere, has no irreducible pairs. 
\end{thm}
\begin{proof}
	Let $\widetilde Q$ be path-consistent and simply connected dimer model that is not on a sphere. 
	Take a pair $(p,q)$ of paths such that $q^{-1}p$ is an elementary counter-clockwise cycle-walk. We show that $p$ has a left-morphable arrow or $q$ has a right-morphable arrow.
	If $p$ is a cycle and $q$ is trivial, then Proposition~\ref{prop:simple-cc-cycle-c-left-chain} shows the desired result. If $p$ is trivial and $q$ is a cycle, the dual of Proposition~\ref{prop:simple-cc-cycle-c-left-chain} shows the desired result. We may now assume that $p$ and $q$ are not cycles. Suppose for the sake of contradiction that $(p,q)$ is an irreducible pair.
	
	Choose a face $F$ in the interior of $q^{-1}p$. By Proposition~\ref{lem:cycle-removing-no-right-cc-cycle}, there is a good cycle-removing chain $a=\alpha_r\dots\alpha_1$ from $p$ to some minimal path $\omega_a(p)$. By repeated application of Lemma~\ref{lem:ffl}, no intermediate path of this chain enters the interior of $q^{-1}p$. In particular, no arrow $\alpha_i$ is an arrow of $F$.
	Then Lemma~\ref{lem:rotation-number-formula} tells us that $\wind\big(q^{-1}m_{\alpha_i}(\omega_{\alpha_{i-1}\dots\alpha_{1}}(p)),F\big)=\wind\big(q^{-1}\omega_{\alpha_{i-1}\dots\alpha_{1}}(p),F\big)$ for each $i$.
	There may be some clockwise cycles removed from $m_{\alpha_i}(\omega_{\alpha_{i-1}\dots\alpha_{1}}(p))$ to get $\omega_{\alpha_i\dots\alpha_1}(p)$, but since $a$ is good no counter-clockwise cycles are removed. Hence,
	\[\wind\big(q^{-1}\omega_{\alpha_i\dots\alpha_1}(p),F\big)\geq\wind\big(q^{-1}m_{\alpha_i}(\omega_{\alpha_{i-1}\dots\alpha_1}(p)),F\big)=\wind(\omega_{\alpha_{i-1}\dots\alpha_1}(p),F).\]
	By applying this result for each $i$, we see that
	\[\wind(q^{-1}\omega_a(p),F)\geq\wind(q^{-1}p,F)=1.\]
	Dually, there is a cycle-removing chain $b=\beta_s\dots\beta_1$ of $q$ removing only counter-clockwise cycles such that $\omega_b(q)$ is minimal and
	\[\wind(q^{-1}\omega_b(q),F)\leq\wind(q^{-1}q,F)=0.\]
	By path-consistency, $\omega_a(p)$ and $\omega_b(q)$ are equivalent. Then there is a chain $c=\gamma_t\dots\gamma_1$ such that $m_c(\omega_a(p))=\omega_b(q)$. As above, Lemma~\ref{lem:ffl} shows that no arrow $\gamma_i$ of $c$ is an arrow of $F$, hence repeated application of Lemma~\ref{lem:rotation-number-formula} gives that $1\leq \wind(q^{-1}\omega_a(p),F)=\wind(q^{-1}\omega_b(p),F)\leq 0$, a contradiction.
\end{proof}

\section{Strand Diagrams and Strand-Consistency}\label{sec:pdsc}

In this section, we use \textit{zigzag paths} to associate a dimer model to a \textit{strand diagram} on its surface. Our goal is to prove that, excluding the case of a sphere, a finite dimer model is path-consistent if and only if there are no bad configurations in its strand diagram. This generalizes ideas of Bocklandt~\cite{XBocklandt2011} and Ishii and Ueda~\cite{XIU}. The theory of cycle-removing morphs developed in Section~\ref{sec:actual-argument}, in particular Theorem~\ref{prop:no-irreducible-pairs}, is necessary to prove this result. 

\subsection{Strand Diagrams}

We define strand diagrams and connect them to dimer models. The below definition is a reformulation of~\cite[Definition 1.10]{XBocklandt2015}.

\begin{defn}\label{defn:postnikov-diagram}
	Let $\Sigma$ be an oriented surface with or without boundary with a discrete set of marked points on its boundary. A (connected) \textit{strand diagram} $D$ in $\Sigma$ consists of a collection of smooth directed curves drawn on the surface $\Sigma$, called \textit{strands}, each of which is either an \textit{interior cycle}
	 contained entirely in the interior of $D$ or starts and ends at marked boundary points, subject to the following conditions.
    \begin{enumerate}
        {\item Each boundary marked point is the start point of exactly one strand, and the end point of exactly one strand.\label{pd:1}}
        {\item Any two strands intersect in finitely many points, and each intersection involves only two strands. Each intersection not at a marked boundary point is transversal.\label{pd:2}}
	{\item\label{pd:3} Moving along a strand, the signs of its crossings with other strands alternate. This includes intersections at a marked boundary point. See the figure below, where the bold segment is boundary. 
		\def\svgwidth{60pt}
		\begin{figure}[H]			\centering
			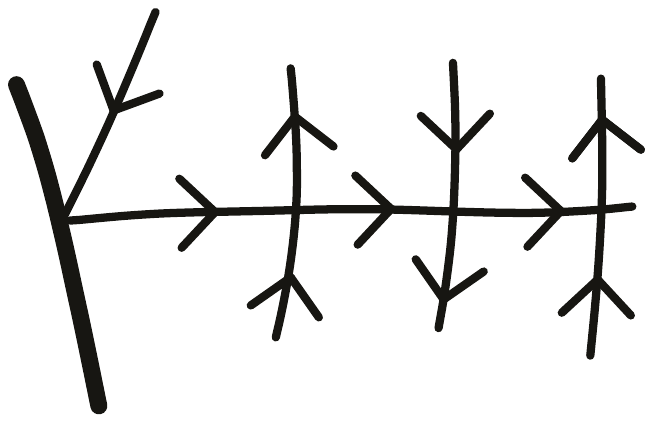
\label{fig:alternating-strands}

		\end{figure}}
	{\item\label{dpd:4} 
		    Any connected component $C$ of the complement of $D$ in the interior of $\Sigma$ is an open disk. The boundary of $C$ may contain either zero or one one-dimensional ``boundary segment'' of the boundary of $\Sigma$. In the former case, $C$ is an \textit{internal region}. In the latter, $C$ is a \textit{boundary region}. 
It follows from~\eqref{pd:3} that $C$ is either an \textit{oriented region} (i.e., all strands on the boundary of the component are oriented in the same direction) or an \textit{alternating region} (i.e., the strands on the boundary of the component alternate directions). See the left and right sides of Figure~\ref{fig:oriented-alternating}, respectively.
Note that by~\eqref{pd:3}, any boundary region with multiple strands must be alternating. We consider a boundary region with a single strand to be alternating.

	\def\svgwidth{150pt}
		\begin{figure}			\centering
			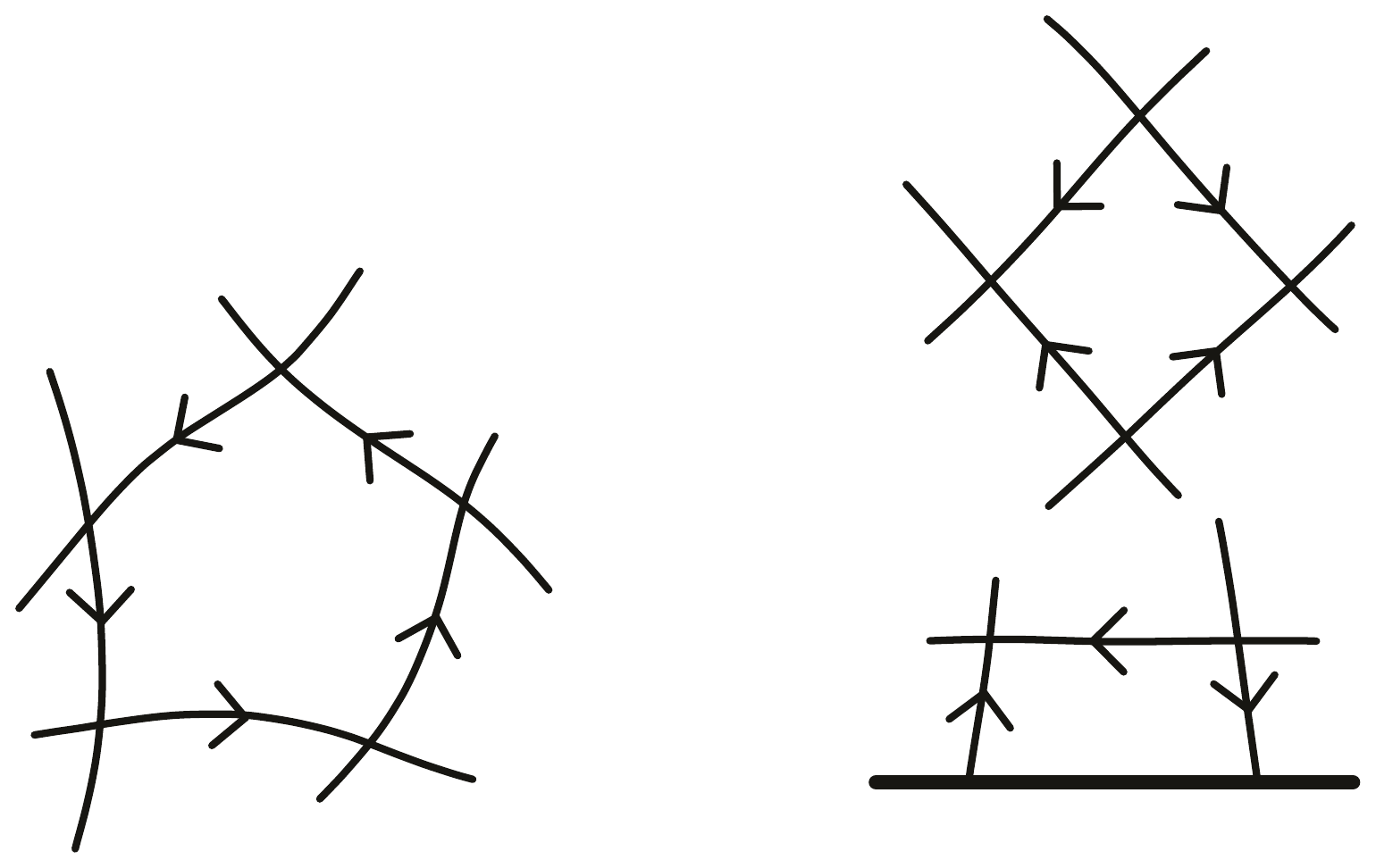
			\caption{One oriented (left) and two alternating (right) regions. The bold segment is boundary.}
			\label{fig:oriented-alternating}
		\end{figure}}
	\item\label{pd:5} The union of the strands is connected.
    \end{enumerate}
    The diagram $D$ is called a \textit{Postnikov diagram} if in addition it satisfies the following conditions
    \begin{enumerate}
        {\item No subpath of a strand is a null-homotopic interior cycle.\label{pd:6nointeriorcycles}}
	{\item\label{spd:2} If two strand segments intersect twice and are oriented in the same direction between these intersection points, then they must not be homotopic.\label{pd:5badlens}} 
    \end{enumerate}
	In other words, \textit{bad configurations} shown in Figure~\ref{fig:bad-configurations} and described below must not appear in order for $D$ to be a Postnikov diagram:
	\def\svgwidth{150pt}
	\begin{figure}		\centering
		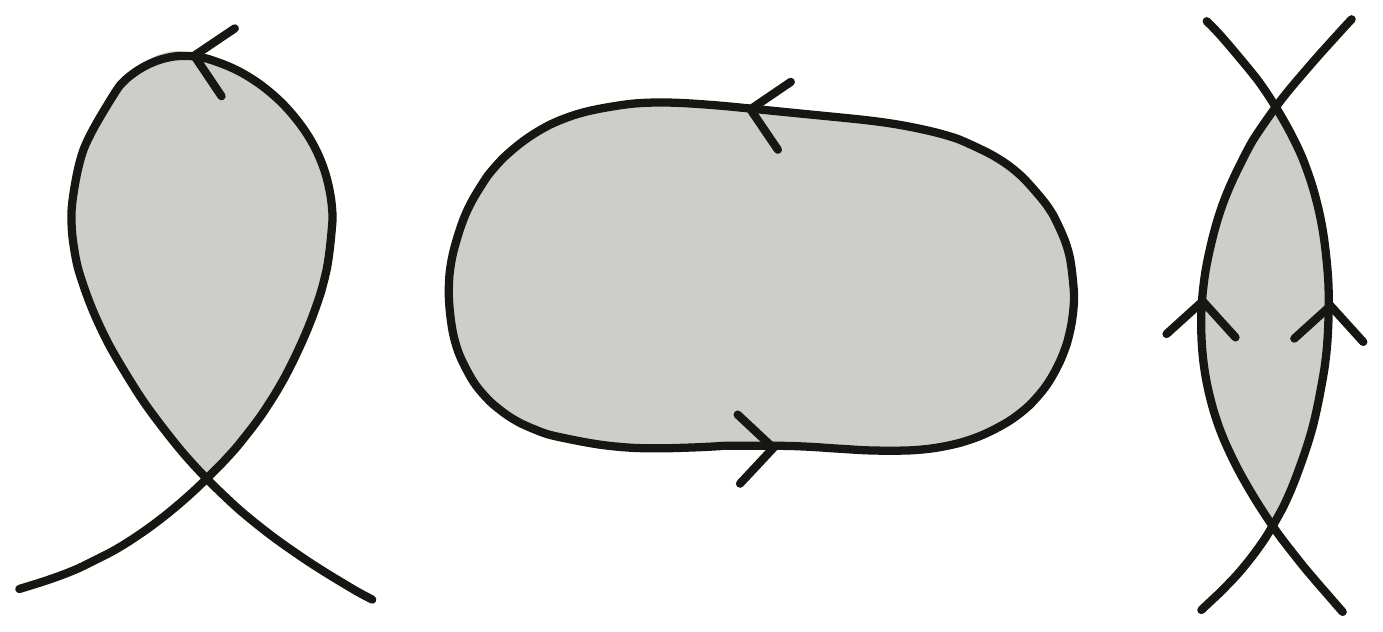
		\caption{The three bad configurations. The shaded areas are contractible.}
		\label{fig:bad-configurations}
	\end{figure}
	\begin{enumerate}
		\item\label{gg1} A strand which intersects itself through a null-homotopic cycle as forbidden in~\eqref{pd:6nointeriorcycles}, called a \textit{null-homotopic self-intersecting strand}. 
		\item\label{gg2} A null-homotopic strand in the interior as forbidden in~\eqref{pd:6nointeriorcycles}, called a \textit{null-homotopic interior cycle}. 
		\item\label{gg3} Two strand segments which intersect in the same order null-homotopically as forbidden in~\eqref{pd:5badlens}, called a \textit{bad lens}.
	\end{enumerate}
	When $Q=\widetilde Q$ is simply connected, any cycle is null-homotopic and we often omit ``null-homotopic'' from~\eqref{gg1} and~\eqref{gg2}.
\end{defn}

\begin{remk}
	In fact, if $\Sigma$ is simply connected then conditions~\eqref{pd:5} and~\eqref{pd:3} of a strand diagram along with the lack of bad lenses imply that if there are multiple strands, then there is no strand which starts and ends at the same marked boundary point, and hence no strand contains a cycle. 
	To see this, suppose there is a strand $z$ which starts and ends at the same marked boundary point. Consider the first time this strand intersects with another strand $w$. Then by connectedness~\eqref{pd:3}, $w$ must enter the area defined by $z$ at this intersection. Then $w$ must eventually leave the area bounded by $z$, creating a bad lens.
\end{remk}

\begin{defn}~\label{defn:dimer-model-from-strand-diagram}
	Let $D$ be a strand diagram in a surface. We associate to $D$ a dimer model $Q_D$ as follows. The vertices of $Q_D$ are the alternating regions of $D$. When the closures of two different alternating regions $v_1$ and $v_2$ meet in a crossing point between strands of $D$, or at one of the marked boundary points, we draw an arrow between $v_1$ and $v_2$, oriented in a way consistent with these strands, as shown in Figure~\ref{fig:strand-arrows}. 
	The counter-clockwise (respectively clockwise) faces of $Q_D$ are the arrows around a counter-clockwise (respectively clockwise) region of $D$.
	\def\svgwidth{200pt}
	\begin{figure}		\centering
		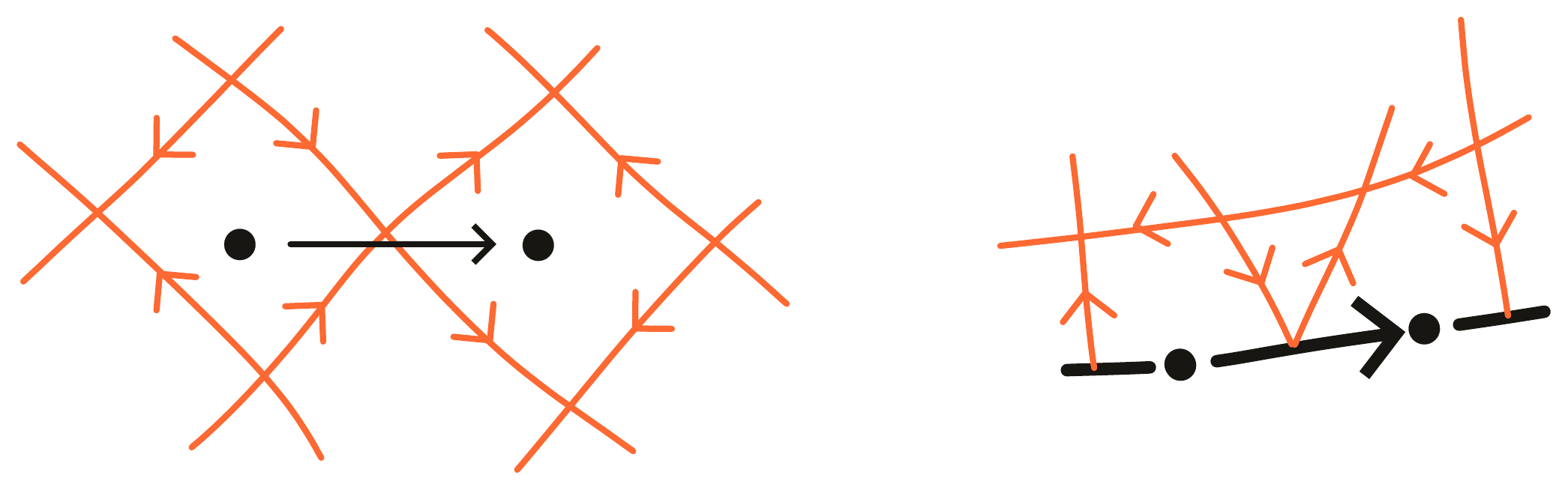
		\caption{The arrows between two alternating faces. The bold arrow is a boundary arrow.}
		\label{fig:strand-arrows}
	\end{figure}
\end{defn}

We may also go in the other direction.

\begin{defn}\label{defn:strand-diagram-from-dimer-model}
	Let $Q$ be a dimer model. We associate a strand diagram $D_Q$ to $Q$ embedded in the surface $S(Q)$ as follows.
	To any arrow $\alpha$ of $Q$, let $v_\alpha$ be the point in the center of $\alpha$ in the embedding of $Q$ into $S(Q)$.

	For any two arrows $\alpha$ and $\beta$ of $Q$ such that $\beta\alpha$ is a subpath of a face-path, we draw a path from $v_\alpha$ to $v_\beta$ along the interior of the face containing $\beta\alpha$. 
	The union of these strand segments forms a connected strand diagram $D_Q$~\cite[\S0.4 -- 0.5]{XBocklandt2015}. See Figure~\ref{fig:submodel-disk-annulus} for an example on the disk and annulus.
    
\end{defn}

The above constructions are mutual inverses, and hence establish a correspondence between strand diagrams and dimer models. This is implicit in the work of Bocklandt~\cite{XBocklandt2015}.

\begin{defn}
	A dimer model $Q$ is \textit{strand-consistent} if its strand diagram $D_Q$ does not have any bad configurations. In other words, $Q$ is strand-consistent if $D_Q$ is a Postnikov diagram.
\end{defn}

\begin{remk}\label{remk:alg-cons-no-sphere}
	A dimer model on a sphere is never strand-consistent since there must be a null-homotopic interior cycle.
	On the other hand, a dimer model on a sphere can still satisfy the path-consistency condition of Definition~\ref{defn:algebraic-consistency}. Hence, in order to prove that the notions of path-consistency and strand-consistency are equivalent (Theorem~\ref{thm:not-a-cons-then-not-g-cons}),
it is necessary to throw out the case where $S(Q)$ is a sphere. 
\end{remk}

Since a cycle on a surface $\Sigma$ lifts to a cycle on the universal cover if and only if it is null-homotopic in $\Sigma$, bad configurations in $Q$ correspond precisely to bad configurations in $\widetilde Q$. The following result is a consequence of this observation.

\begin{prop}\label{prop:Q-g-consistent-iff-hat-Q-g-consistent}
        $Q$ is strand-consistent if and only if $\widetilde Q$ is strand-consistent.
\end{prop}

Proposition~\ref{prop:Q-g-consistent-iff-hat-Q-g-consistent} is useful because bad configurations are easier to work with on the simply connected dimer model $\widetilde Q$. In particular, the null-homotopic conditions appearing in each of the three bad configurations may be ignored in the simply connected case.
As such, we often pass to the universal cover models when working with strand diagrams.

\subsection{Path-Consistency Implies Strand-Consistency}

We now prove that path-consistency implies strand-consistency for finite dimer models. We must first define zigzag paths and their return paths.  The notion of a zigzag path is based on work by Kenyon in \cite{Kenyon} and Kenyon and Schlenker in \cite{KenSch}.

\begin{defn}\label{defn:zigzag-path}
	Let $Q$ be a dimer model. A \textit{zigzag path} of $Q$ is a maximal (possibly infinite) path $z=\dots\gamma_{i+1}\gamma_{i}\dots$ such that one of the following holds.

    \begin{enumerate}
	    \item $\gamma_{i+1}\gamma_{i}$ is part of a counter-clockwise face if $i$ is odd and a clockwise face if $i$ is even. In the former case, $\gamma_{i+1}\gamma_{i}$ is called a \textit{zig} and $h(\gamma_i)$ is a \textit{zig vertex}. In the latter, $\gamma_{i+1}\gamma_{i}$ is a \textit{zag} and $h(\gamma_i)$ is a \textit{zag vertex}.
        \item $\gamma_{i+1}\gamma_i$ is part of a clockwise face if $i$ is odd and a counter-clockwise face if $i$ is even. In the former case, $\gamma_{i+1}\gamma_i$ is called a \textit{zag} and $h(\gamma_i)$ is a \textit{zag vertex}. In the latter, $\gamma_{i+1}\gamma_i$ is a \textit{zig} and $h(\gamma_i)$ is a \textit{zig vertex}.
    \end{enumerate}

	If $z$ is a finite path, then we write $z=\gamma_t\dots\gamma_0$ and we note that $\gamma_0$ and $\gamma_t$ must be boundary arrows. we similarly may have infinite zigzag paths $\gamma_0\gamma_{-1}\dots$ or $\dots\gamma_{1}\gamma_0$ ending or starting at boundary vertices, respectively.
\end{defn}

It is immediate by the constructions of Definitions~\ref{defn:strand-diagram-from-dimer-model} and~\ref{defn:dimer-model-from-strand-diagram} that zigzag paths of $Q$ correspond to strands of $D_Q$.
Intersections of strands in $D_Q$ correspond to shared arrows of zigzag paths in the quiver. We may thus view the bad configurations of Definition~\ref{defn:postnikov-diagram} 
as bad configurations of zigzag paths. 
\begin{enumerate}
	\item A zigzag path has a \textit{null-homotopic self-intersection} if it passes through the same arrow twice, first as the start of a zig and then as the start of a zag (or vice versa), and the segment between these occurrences is null-homotopic. 
	\item A zigzag path is a \textit{null-homotopic interior cycle} if it is cyclic, infinitely repeating, and null-homotopic. 
	\item Two homotopic segments of (possibly the same) zigzag paths $z'$ and $w'$ form a \textit{bad lens} if $z'=\beta z_s\dots z_0\alpha$ and $w'=\beta w_t\dots w_0 \alpha$ and $z_i\neq w_j$ for all $i$ and $j$.
\end{enumerate}

The following definition appears in the literature on dimer models in tori.
See, for example,~\cite{XBocklandt2011} and~\cite{XBroomhead2009}.

\begin{defn}\label{defn:return-path}
	Take a subpath $z':=\gamma_t\dots\gamma_1$ of a zigzag path $z$ such that $\gamma_i\neq\gamma_j$ for $i\neq j$. 
	Suppose that $t(z')$ and $h(z')$ are both zag vertices of $z$. 
	For each zig $\gamma_{j+1}\gamma_j$ (for $j<t$ odd), let $p_j$ be the subpath of $F_{\gamma_j}^{cc}$ from $h(\gamma_{j+1})$ to $t(\gamma_j)$. Then $p_j$ consists of all arrows in $F_{\gamma_j}^{cc}$ except for $\gamma_{j+1}$ and $\gamma_j$.

	The \textit{right return path} 
	is the composition $p_{j_{t-1}}\dots p_{j_1}$ of all such $p_j$'s. The \textit{elementary right return path} is the path obtained by removing all proper elementary subcycles of the right return path in order of their appearance. If the result is a face-path, then the elementary right return path is constant.
	We dually define \textit{(elementary) left return paths}.
\end{defn}

For examples, see Figure~\ref{fig:triple}. In particular, the left of this figure gives an example where the right return path and the elementary right return path differ. It is immediate that right return paths are leftmost and left return paths are rightmost.

\begin{thm}\label{thm:good-quiver-no-bad-configurations}
	A path-consistent dimer model not on a sphere is strand-consistent.
\end{thm}
\begin{proof}
	Given a path-consistent dimer model $Q$, Proposition~\ref{prop:Q-consistent-iff-hat-Q-consistent} implies that $\widetilde Q$ is path-consistent. By Proposition~\ref{prop:Q-g-consistent-iff-hat-Q-g-consistent}, it suffices to show that $\widetilde Q$ is strand-consistent. We will do this by showing that self-intersecting strands, interior cycles, and bad lenses in the strand diagram give rise to irreducible pairs, which cannot occur in strand-consistent models by Theorem~\ref{prop:no-irreducible-pairs}, since $\widetilde Q$ is not a sphere. See Figure~\ref{fig:triple}.
	\def\svgwidth{400pt}
	\begin{figure}		\centering
		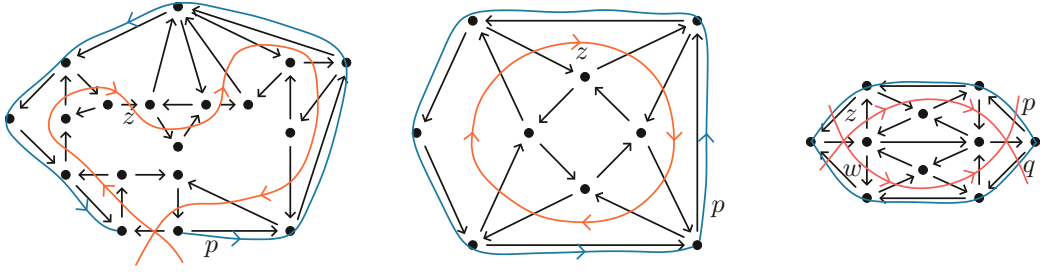
		\caption{An example of a self-intersection (left), interior cycle (middle), and bad lens (right). All of these give rise to irreducible pairs.}
		\label{fig:triple}
	\end{figure}

	First, suppose there is some self-intersecting strand $z$ of $D_{\widetilde Q}$. Recall from the discussion following Definition~\ref{defn:zigzag-path} that intersections of the strand $z$ correspond to multiple incidences of an arrow in its associated zigzag path in ${\widetilde Q}$. Then there is some segment $C=\gamma_0\gamma_t\dots\gamma_1\gamma_0$ of the zigzag path associated to $z$ such that $\gamma_i\neq\gamma_j$ for $i\neq j$. Suppose that the segment of the strand $z$ corresponding to $\gamma_t\dots\gamma_0$ is a clockwise cycle; the counter-clockwise case is symmetric. See the left of Figure~\ref{fig:triple}. Let $v_1:=h(\gamma_0)$ and $v_2:=t(\gamma_0)$. Since the cycle runs clockwise, $v_1$ and $v_2$ are both zag vertices of $z$. Let $z':=\gamma_t\dots\gamma_1$ and let $p$ be the elementary right return path of $z'$. Then $p$ is a leftmost path from $v_2$ to $v_1$ and $\gamma_0^{-1}p$ is an elementary counter-clockwise cycle-path winding counter-clockwise around $z'$, hence $(p,\gamma_0)$ is an irreducible pair. This contradicts Theorem~\ref{prop:no-irreducible-pairs}.

	Now suppose there is a strand $z$ of $D_{\widetilde Q}$ which is an interior cycle. By the above, we may suppose that $z$ contains no self-intersections. Then we may realize $z$ as a path $z':=\gamma_t\dots\gamma_0$ of ${\widetilde Q}$ such that $t(\gamma_t)=h(\gamma_0)$ is a zag vertex of $z$ and $\gamma_i\neq\gamma_j$ for $i\neq j$. See the middle of Figure~\ref{fig:triple}. As above, we assume that $z'$ winds clockwise. Let $p$ be the elementary right return path of $z'$. Then $p$ winds counter-clockwise around $z'$ and thus is a nontrivial counter-clockwise path which is leftmost, contradicting Theorem~\ref{prop:no-irreducible-pairs}.

		Now, suppose there is a bad lens in $D_{\widetilde Q}$. Accordingly, we may take subpaths of zigzag paths $z':=\beta\gamma_t\dots\gamma_1\alpha$ and $w':=\beta\delta_s\dots\delta_1\alpha$ such that $\gamma_i\neq\delta_j$ for $i\neq j$. By the above, the strands have no self-intersections and hence $z'$ and $w'$ have no repeated arrows. Suppose without loss of generality that $z'$ is to the left of $w'$. Then $t(\alpha)$ and $h(\beta)$ are zig vertices of $w$ and zag vertices of $z$. See the right of Figure~\ref{fig:triple}. Let $p$ be the right elementary return path of $z'$ and let $q$ be the left elementary return path of $w'$. 

		Let $p_0:=p$ and $q_0:=q$. Choose a face $F$ in the interior of the bad lens. Then $p_0$ is a leftmost elementary path and $q_0$ is a rightmost elementary path such that $\wind(q_0^{-1}p_0,F)=1>0$, since $q^{-1}p$ winds counter-clockwise around the lens. If $q_0^{-1}p_0$ is simple, then $(p_0,q_0)$ is an irreducible pair and we are done. If not, then $q_0^{-1}p_0$ has some simple proper subcycle-walk $l$. The paths $p_0$ and $q_0$ are elementary, so $l$ must be of the form $(q'_0)^{-1}p'_0$, where $p'_0,q'_0,p_1,q_1$ are paths such that $p_0=p'_0p_1$ and $q_0=q_1q'_0$. If $(q'_0)^{-1}p'_0$ is counter-clockwise, then $(p'_0,q'_0)$ forms an irreducible pair, since any subpath of $p$ is leftmost and any subpath of $q$ is rightmost. If not, then the removal of $(q'_0)^{-1}p'_0$ from $q_0^{-1}p_0$ may only increase the winding number around $F$, hence $\wind(q_1^{-1}p_1,F)\geq\wind(q_0^{-1}p_0,F)>0$. We now start the process over with $p_1$ and $q_1$ in place of $p_0$ and $q_0$. This process must eventually terminate when some $(p'_i,q'_i)$ forms an irreducible pair, contradicting Theorem~\ref{prop:no-irreducible-pairs}.
\end{proof}

\subsection{Strand-Consistency Implies Path-Consistency}

We now prove the converse of Theorem~\ref{thm:good-quiver-no-bad-configurations}, completing the proof that the notions of path-consistency and strand-consistency are equivalent for finite dimer models on surfaces which are not spheres.
In the case where $Q$ is a dimer model on a disk, this is proven in~\cite[Proposition 2.15]{CKP}. In the case where $Q$ is a dimer model on a compact surface without boundary, it appears in~\cite[Theorem 10.1, Theorem 10.2]{XBocklandt2011}.
First, we need the following definition.

\begin{defn}\label{defn:dimer-submodel}
	Let $Q=(Q_0,Q_1,Q_2)$ be a dimer model. Let $\mathcal F\subseteq Q_2$ be a set of faces of $Q_2$ which form a connected surface with boundary which is a subspace of the surface $S(Q)$ of $Q$.  

	We define the \textit{dimer submodel $Q^\mathcal F$ of $Q$ induced by $\mathcal F$} as the dimer model $Q^\mathcal F=(Q^\mathcal F_0,Q^\mathcal F_1,\mathcal F)$, where $Q^\mathcal F_0$ and $Q^\mathcal F_1$ are the sets of vertices and edges of $Q$ appearing in some face of $\mathcal F$.
\end{defn}

Intuitively, $Q^\mathcal F$ is obtained by deleting all faces of $Q$ which are not in $\mathcal F$. See Figure~\ref{fig:submodel-disk-annulus}. It is immediate that if $\mathcal G\subseteq\mathcal F\subseteq Q_2$ induce dimer submodels, then $Q^\mathcal G=(Q^\mathcal F)^\mathcal G$. A \textit{disk submodel} is a dimer submodel which is a dimer model on a disk.

\begin{prop}\label{thm:not-a-cons-then-not-g-conss}
	A strand-consistent dimer model is path-consistent.
\end{prop}
\begin{proof}
	By Theorem~\ref{thm:consistent-iff-cancellation}, $Q$ is path-consistent if and only if $A_Q$ is cancellative. By Lemma~\ref{lem:Q-canc-iff-hat-Q-canc}, this is true if and only if $A_{\widetilde Q}$ is cancellative. Hence, it suffices to show that $A_{\widetilde Q}$ is a cancellation algebra.
	Accordingly, suppose $p,q,a$ are paths of $\widetilde Q$ such that $h(a)=t(p)=t(q)$ and $h(p)=h(q)$ and $[pa]=[qa]$. Then there is a finite sequence of morphs taking $pa$ to $qa$ in $\widetilde Q$. Let $Q'$ be a disk submodel of $\widetilde Q$ containing every intermediate path in this sequence. Since $D_{Q'}$ is a restriction of $D_{\widetilde Q}$, which has no bad configurations, the former also has no bad configurations. By~\cite[Proposition 2.15]{CKP}, $Q'$ is path-consistent. By Theorem~\ref{thm:consistent-iff-cancellation}, $A_{Q'}$ is cancellative, hence $[p]=[q]$ in $Q'$. Then there is a sequence of morphs taking $p$ to $q$ in $Q'$; this is necessarily a sequence of morphs in $\widetilde Q$, so $[p]=[q]$ in $\widetilde Q$.

	It may similarly be shown that if $p,q,b$ are paths of $\widetilde Q$ such that $t(b)=h(p)=h(q)$ and $t(p)=t(q)$, then $[bp]=[bq]$ implies $[p]=[q]$. This completes the proof that $A_{\widetilde Q}$ is cancellation, and thus that $Q$ is path-consistent.
\end{proof}

\begin{thm}\label{thm:not-a-cons-then-not-g-cons}\label{thm:cons-alg-str}
	Let $Q$ be a dimer model not on a sphere. The following are equivalent:
	\begin{enumerate}
		\item $Q$ is path-consistent,
		\item $Q$ is strand-consistent,
		\item The dimer algebra $A_Q$ is cancellative.
	\end{enumerate}
\end{thm}
\begin{proof}
	Path-consistency implies strand-consistency by Theorem~\ref{thm:good-quiver-no-bad-configurations}. Strand-consistency implies path-consistency by Proposition~\ref{thm:not-a-cons-then-not-g-conss}.
	Path-consistency is equivalent to cancellativity by Theorem~\ref{thm:consistent-iff-cancellation}.
\end{proof}
Theorem~\ref{thm:not-a-cons-then-not-g-cons} is known for dimer models on tori, for example see~\cite{XBocklandt2011} and references therein.
It was shown for dimer models on the disk corresponding to $(k,n)$-diagrams in~\cite{BKMX}. The implication~\eqref{q2}$\implies$\eqref{q1} for general dimer models on disks appears in~\cite[Proposition 2.11]{XPressland2019}. The authors are not aware of a proof in the other implication in the case of the disk, hence we include this corollary. 
\begin{cor}\label{cor:main-cor-disk}
	Let $Q$ be a dimer model in a disk. Then $Q$ is path-consistent if and only if $Q$ is strand-consistent.
\end{cor}

In light of Theorem~\ref{thm:cons-alg-str}, we use the word \textit{weakly consistent} to refer to both path-consistent and strand-consistent dimer models that are not on a sphere. 
We will define strong consistency in Section~\ref{sec:perfect-matching}.

\section{Dimer Submodels}\label{sec:ds}

Recall the definition of a {dimer submodel} in Definition~\ref{defn:dimer-submodel}. We now show that the dimer submodel of a weakly consistent model is weakly consistent, and moreover that the path equivalence classes of a dimer submodel may be understood in terms of the original dimer model. This is a useful result that will lead to some nice corollaries in Section~\ref{sec:thin-quivers}.

\def\svgwidth{300pt}
\begin{figure}	\centering
	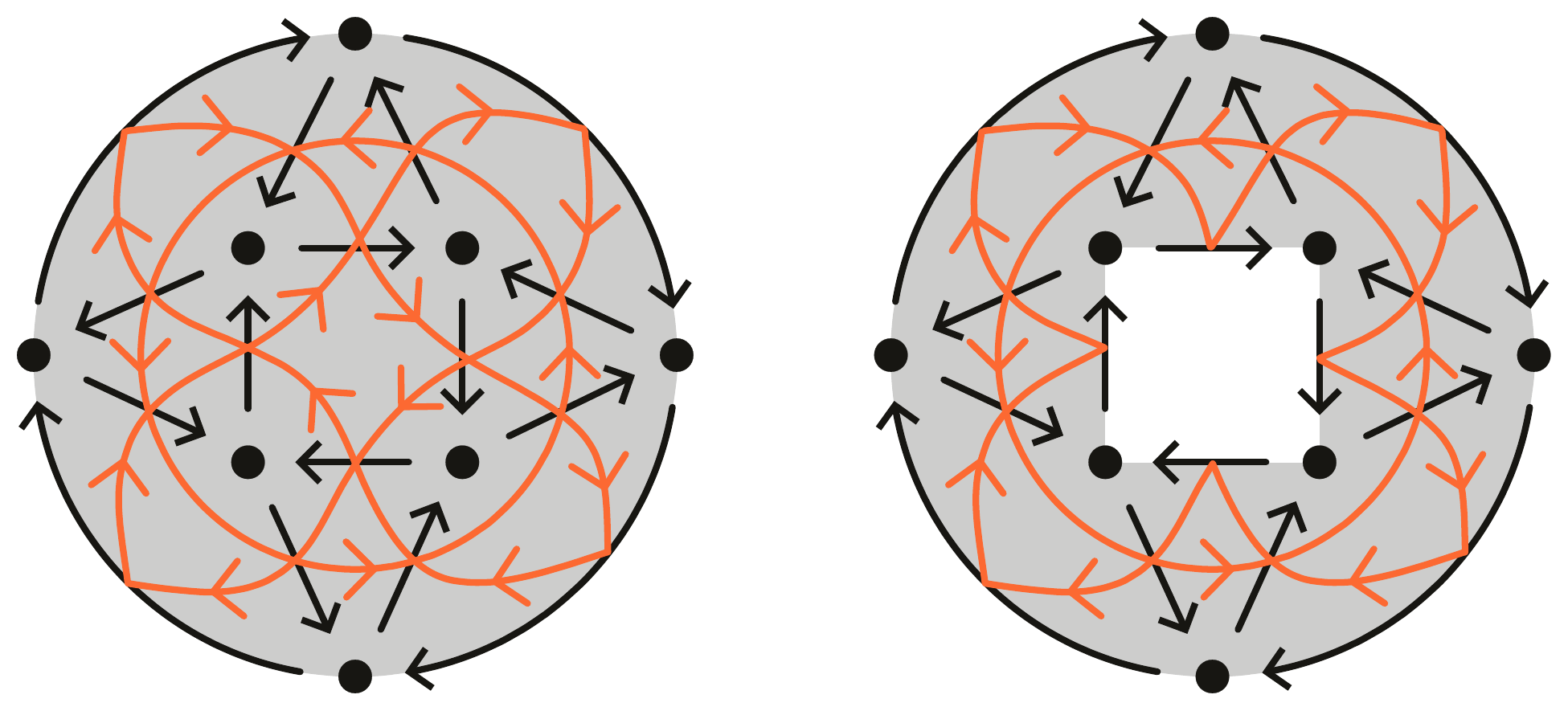
	\caption{The disk model on the left is not weakly consistent, since it has a homologically trivial interior cycle. When we delete the center face by taking the submodel induced by all other faces, this cycle still exists but is no longer homologically trivial. The result is a weakly consistent dimer model on an annulus.}
	\label{fig:submodel-disk-annulus}
\end{figure}

\begin{cor}\label{cor:dimer-submodel-consistent}
	Let $Q$ be a dimer model. Let $\mathcal F$ be a set of faces of $Q$ forming a surface $S$ such that the restriction of the strand diagram $D_Q$ to $S$ has no bad configurations. Then $Q^\mathcal F$ is weakly consistent. In particular, if $Q$ is weakly consistent then any dimer submodel of $Q$ is weakly consistent. 
\end{cor}
\begin{proof}
	Since weak consistency is characterized by the absence of bad configurations by Theorem~\ref{thm:not-a-cons-then-not-g-cons}, the first statement is trivial. 
	Passing from $Q$ to a dimer submodel $Q^{\mathcal F}$ corresponds to restricting the strand diagram of $Q$ to the surface given by the union of the faces in $\mathcal F$. This cannot create any bad configurations, so the second statement follows.
\end{proof}

See Figure~\ref{fig:submodel-disk-annulus} for an example of how Corollary~\ref{cor:dimer-submodel-consistent} may be used in practice to obtain weakly consistent dimer models from existing (not necessarily weakly consistent) models.

\begin{thm}\label{thm:submodel-path-equivalence}
	Let $Q$ be a weakly consistent dimer model and let $Q^{\mathcal F}$ be a (necessarily weakly consistent) dimer submodel of $Q$. Then two paths in $Q^{\mathcal F}$ are equivalent in $Q^{\mathcal F}$ if and only if they are equivalent in $Q$ and homotopic in $S(Q^{\mathcal F})$.
\end{thm}
\begin{proof}
	If $[p]=[q]$ in $Q^\mathcal F$, then certainly $[p]=[q]$ in $Q$. Moreover, in this case $p$ is homotopic to $q$ in $S(Q^\mathcal F)$, hence in $Q$.
		On the other hand, suppose that $[p]=[q]$ in $Q$ and that $p$ is homotopic to $q$ in $S(Q^{\mathcal F})$. By path-consistency of $Q^\mathcal F$, without loss of generality we have $[p]=[qf^m]$ in $Q^{\mathcal F}$ for some nonnegative integer $m$. Then we have $[p]=[qf^m]$ in $Q$. By path-consistency of $Q$, since $[p]=[q]$ in $Q$ we must have $m=0$ and hence $[p]=[q]$ in $Q^\mathcal F$. 
\end{proof}

\begin{remk}
	Theorem~\ref{thm:submodel-path-equivalence} could be stated more generally without changing the proof. We don't need $Q$ to be weakly consistent; we merely need to be able to ``cancel face-paths'' in $Q$. In other words, we require that $[pf^m]=[p]$ cannot hold for positive $m$. This is a weaker condition than cancellativity (and hence weak consistency) and is satisfied, for example, if $Q$ has a perfect matching. See Section~\ref{sec:perfect-matching} and  Lemma~\ref{lem:perfect-matching-intersection-num}.
\end{remk}

\section{Perfect Matchings}\label{sec:perfect-matching}
	A \textit{perfect matching} of a dimer model $Q$ is a collection of arrows $\mathcal M$ of $Q$ such that every face of $Q$ contains exactly one arrow in $\mathcal M$.
	See Figure~\ref{fig:plabic-quiver-strand} for two examples.

	Dimer models on the torus with perfect matchings have been studied in, for example,~\cite{XIU},~\cite{XBroomhead2009},~\cite{XBocklandt2011} and are often called~\textit{dimer configurations}. The Gorenstein affine toric threefold obtained by putting the \emph{perfect matching polygon} at height one is the center of the dimer algebra, and the dimer algebra is viewed as a non-commutative crepant resolution of this variety, for example see~\cite{XBocklandt2015, XBroomhead2009, XIUM, Mozg}.
Perfect matchings of dimer models on a more general surface, often called \textit{almost perfect matchings}, are a natural generalization. {Perfect matching polygons may be extended to dimer models over arbitrary compact surfaces with boundary, and capture the data of the master and mesonic moduli spaces~\cite{BFTFDBPTSA}.}
Perfect matchings may be calculated by taking determinants of Kasteleyn matrices, see \cite{HaK} and \cite[\S5]{BFTFDBPTSA}.

In the present paper, we give some basic results, prove existence of perfect matchings in weakly consistent simply connected dimer models, and give a counterexample to the existence of perfect matchings in arbitrary weakly consistent dimer models.

We use combinatorial theory of matchings of (undirected) graphs in order to prove the main result of this section. In order to do so, we associate to $Q$ a bipartite \textit{plabic graph} $\mathcal G=(\mathcal G_0^b,\mathcal G_0^w,\mathcal G_1)$ embedded into $S(Q)$ as follows.
To each face $F$ of $Q$, we associate an \textit{internal vertex} $v_F$ embedded in the interior of $F$. If $F$ is a clockwise face, we say that $v_F$ is a \textit{black vertex} and we write $v_F\in\mathcal G_0^b$. If $F$ is a counter-clockwise face, we say that $v_F$ is a \textit{white vertex} and we write $v_F\in\mathcal G_0^w$. 
For any boundary arrow $\alpha$ of $Q$, we draw a \textit{boundary vertex} $v_\alpha$ embedded in the middle of $\alpha$. We consider $v_\alpha$ to be a \textit{white vertex} if $\alpha$ is part of a clockwise face, and a \textit{black vertex} if $\alpha$ is part of a counter-clockwise face.
For each internal arrow $\alpha$ of $Q$, we draw an edge between $v_{F_\alpha^{cl}}$ and $v_{F_\alpha^{cc}}$. For each boundary arrow $\alpha$ of $Q$, we draw an edge between $v_\alpha$ and the vertex corresponding to the unique face incident to $\alpha$. See Figure~\ref{fig:plabic-quiver-strand}.
\def\svgwidth{200pt}
\begin{figure}
	\centering
	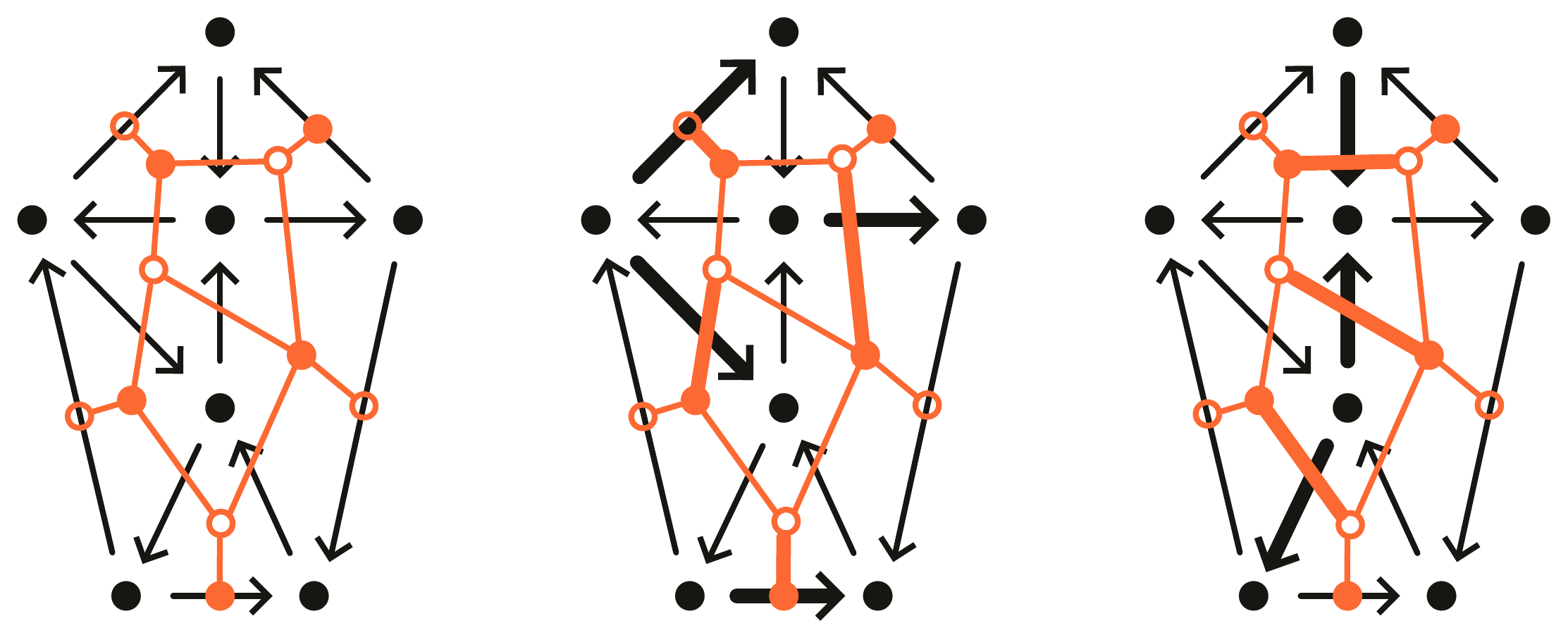
	\caption{On the left, a dimer model on a disk with its plabic graph overlayed is pictured. The two pictures on the right show two different perfect matchings, both as collections of arrows of the quiver and as collections of edges of the plabic graph.}
	\label{fig:plabic-quiver-strand}
\end{figure}

A perfect matching $\mathcal M$ of $Q$ is dual to a collection of edges $\mathcal N$ of the plabic graph $\mathcal G$ of $Q$ such that every internal vertex of $\mathcal G$ is contained in exactly one edge of $\mathcal N$. We refer to both $\mathcal M$ and $\mathcal N$ as perfect matching of $Q$. If $Q$ has no boundary, then a perfect matching is dual to a perfect matching of the dual (plabic) graph of $Q$ in the usual sense. See Figure~\ref{fig:plabic-quiver-strand}.

If $\mathcal M$ is a perfect matching of $Q$, we say that the \textit{intersections} of a path $p$ with $\mathcal M$ are the arrows of $p$ which are in $\mathcal M$. The \textit{intersection number $\mathcal M(p)$ of $p$ with $\mathcal M$} is the number of intersections of $p$ with $\mathcal M$ (counting ``multiplicities'' if $p$ has multiple instances of the same arrow).

\begin{lemma}\label{lem:perfect-matching-intersection-num}
	Suppose $Q$ has a perfect matching $\mathcal M$. If $[p]=[q]$, then $\mathcal M(p)=\mathcal M(q)$.
\end{lemma}
\begin{proof}
	Since $R_\alpha^{cl}$ contains an arrow of $\mathcal M$ if and only if $R_\alpha^{cc}$ does, it follows that basic morphs preserve intersection number. The result follows.
\end{proof}

\begin{prop}\label{prop:perfect-matching-intersection-num}
	Suppose $Q$ is weakly consistent and that $Q$ has a perfect matching $\mathcal M$. Let $p$ and $q$ be paths of $Q$ with the same start vertex, end vertex, and homotopy class. Then $[p]=[q]$ if and only if $\mathcal M(p)=\mathcal M(q)$.
\end{prop}
\begin{proof}
	Lemma~\ref{lem:perfect-matching-intersection-num} shows that if $[p]=[q]$ then $\mathcal M(p)=\mathcal M(q)$. On the other hand, if $[p]\neq[q]$, then without loss of generality $[p]=[qf^m]$ for some $m>0$. Since any face-path contains exactly one arrow of $\mathcal M$, we have $\mathcal M(p)=m+\mathcal M(q)$, ending the proof.
\end{proof}

\begin{lemma}\label{lem:perfect-matching-factor-out-finite}
	Suppose $Q$ has a perfect matching $\mathcal M$ and let $p$ be a path in $Q$.
	The set
	\[\{m\ |\ [p]=[f^mq]\textrm{ for some path }q:t(p)\to h(p)\}\]
	is bounded above by $\mathcal M(p)$. In particular, only a finite number of face-paths can be factored out of $p$.
\end{lemma}
\begin{proof}
	Any face-path has an intersection number of one with any perfect matching. Hence, if $[p]=[f^mq]$, then by Lemma~\ref{lem:perfect-matching-intersection-num} $\mathcal M(p)$ must be at least $m$.
\end{proof}

The condition that only a finite number of face-paths can be factored out of any path $p$ is implied by path-consistency. In fact, it is a strictly weaker property than weak consistency. 

We show the stronger statement that the existence of a perfect matching does not imply weak consistency. Indeed, Figure~\ref{fig:perf-match-not-cons}, which shows a dimer model which has a perfect matching but is not weakly consistent.
More generally, if $Q$ is a weakly consistent dimer model, then let $Q'$ be the dimer model obtained by replacing some internal arrow $\alpha$ of $Q$ with two consecutive arrows $\gamma\beta$ such that $h(\gamma)=h(\alpha)$ and $t(\beta)=t(\alpha)$ and $h(\beta)=t(\gamma)$ is a new vertex of $Q'$. Then the strand diagram of $Q'$ has a bad digon, hence is not weakly consistent. On the other hand, $Q$ has a perfect matching by Theorem~\ref{thm:geo-consistent-then-almost-perfect-matching}, hence $Q'$ has a perfect matching.

\subsection{Consistency and Existence of Perfect Matchings}

We have seen that the existence of a perfect matching does not guarantee weak consistency. We now investigate whether weak consistency guarantees the existence of a perfect matching.
We show in Theorem~\ref{thm:geo-consistent-then-almost-perfect-matching} that perfect matchings exist in simply connected weakly consistent dimer models. On the other hand, we see in Example~\ref{ex:cons-tor-no-perf} that perfect matchings need not exist in arbitrary weakly consistent dimer models.

\begin{defn}
	Let $(V_1,V_2,E)$ be a possibly infinite bipartite graph. Let $\{i,j\}=\{1,2\}$. Let $S\subseteq V_i$. A \textit{matching} from $S$ into $V_j$ is a set $\mathcal N$ of disjoint edges in $E$ such that every vertex of $S$ is incident to precisely one edge in $\mathcal N$.
\end{defn}

A perfect matching of a dimer model, then, is a matching onto some full subgraph of the plabic graph $\mathcal G$ of $Q$ induced by all of the internal vertices and some subset of the boundary vertices.
We use the following formulation of Hall's marriage theorem for locally finite graphs.

\begin{thm}[{\cite[Theorem 6]{XXClark}}]\label{thm:IBLFHT}
	Let $(V_1,V_2,E)$ be a bipartite graph in which every vertex has finite degree. The following are equivalent.
	\begin{enumerate}
		\item There is a matching from $V_1$ into $V_2$.
		\item Any $m$ vertices of $V_1$ have at least $m$ distinct neighbors in $V_2$.
	\end{enumerate}
\end{thm}

\begin{thm}[{\cite[Theorem 1.1]{AharoniX}}]\label{thm:aharoni}
	Let $(V_1,V_2,E)$ be a bipartite graph. Let $A\subseteq V_1$ and $B\subseteq V_2$. If there exists a matching from $A$ into $V_2$ and a matching from $B$ into $V_1$, then there exists a disjoint set of edges $\mathcal N$ in $E$ such that each vertex in $A\cup B$ is incident to precisely one edge in $\mathcal N$.
\end{thm}

\begin{thm}\label{thm:geo-consistent-then-almost-perfect-matching}
	If a simply connected dimer model $\widetilde Q$ is weakly consistent, then it has a perfect matching.
\end{thm}

\begin{proof}
	We will use the dual definition of a perfect matching. We must then show that there is a set $\mathcal N$ of edges of the plabic graph $\mathcal G$ of $\widetilde Q$ such that every vertex of $\mathcal G$ is incident to exactly one edge of $\mathcal N$.

	We first claim that it suffices to show that any collection of $m$ internal black vertices is connected to at least $m$ white vertices and that any collection of $m$ internal white vertices is connected to at least $m$ black vertices. 
	Suppose this is true. 
		By applying Theorem~\ref{thm:IBLFHT} to the internal black vertices, we see that there is a matching from the set of internal black vertices into the white vertices. Symmetrically, we get a matching from the set of internal white vertices into the black vertices. Then Theorem~\ref{thm:aharoni} shows that there exists a perfect matching. This ends the proof of the claim.

		We show that any collection of $m$ internal white vertices is connected to at least $m$ black vertices. The remaining case is symmetric. 
		Take a set $S$ of $m$ internal white vertices of $\mathcal G_{\widetilde Q}$. These correspond to $m$ internal faces of $\widetilde Q$. Let $Q'$ be a disk submodel of $\widetilde Q$ containing all faces of $S$. Such a disk submodel must exist since $\widetilde Q$ is simply connected. Since $D_{\widetilde Q}$ has no bad configurations and $D_{Q'}$ is a restriction of $D_{\widetilde Q}$, the latter also has no bad configurations. By~\cite[Proposition 2.15]{CKP}, $Q'$ is a path-consistent dimer model. By~\cite[Corollary 4.6]{CKP}, $\mathcal G_{Q'}$ has a perfect matching. In particular, by Hall's marriage theorem (Theorem~\ref{thm:IBLFHT})
 the set $S$ of white vertices considered as vertices of $\mathcal G_{Q'}$ has at least $m$ neighbors in $\mathcal G_{Q'}$, hence $S$ has at least $m$ neighbors in $\mathcal G_{\widetilde Q}$. This completes the proof.
\end{proof}

\def\svgwidth{70pt}
\begin{figure}
	\centering
	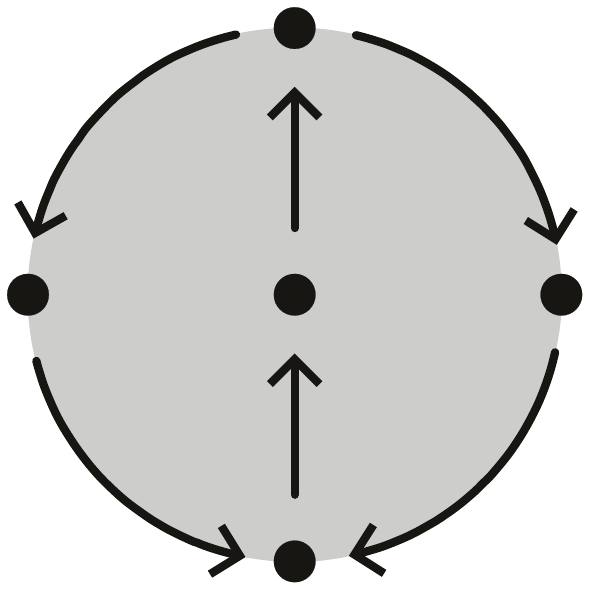
	\caption{A dimer model on a disk which is not weakly consistent but has a perfect matching.}
	\label{fig:perf-match-not-cons}
\end{figure}

Example~\ref{ex:cons-tor-no-perf} shows that Theorem~\ref{thm:geo-consistent-then-almost-perfect-matching} does not work for dimer models which are not simply connected.

\begin{example}\label{ex:cons-tor-no-perf}
	\def\svgwidth{130pt}
\begin{figure}
	\centering
\begingroup%
  \makeatletter%
  \providecommand\color[2][]{%
    \errmessage{(Inkscape) Color is used for the text in Inkscape, but the package 'color.sty' is not loaded}%
    \renewcommand\color[2][]{}%
  }%
  \providecommand\transparent[1]{%
    \errmessage{(Inkscape) Transparency is used (non-zero) for the text in Inkscape, but the package 'transparent.sty' is not loaded}%
    \renewcommand\transparent[1]{}%
  }%
  \providecommand\rotatebox[2]{#2}%
  \newcommand*\fsize{\dimexpr\f@size pt\relax}%
  \newcommand*\lineheight[1]{\fontsize{\fsize}{#1\fsize}\selectfont}%
  \ifx\svgwidth\undefined%
    \setlength{\unitlength}{532.02801514bp}%
    \ifx\svgscale\undefined%
      \relax%
    \else%
      \setlength{\unitlength}{\unitlength * \real{\svgscale}}%
    \fi%
  \else%
    \setlength{\unitlength}{\svgwidth}%
  \fi%
  \global\let\svgwidth\undefined%
  \global\let\svgscale\undefined%
  \makeatother%
  \begin{picture}(1,0.99999805)%
    \lineheight{1}%
    \setlength\tabcolsep{0pt}%
    \put(0,0){\includegraphics[width=\unitlength,page=1]{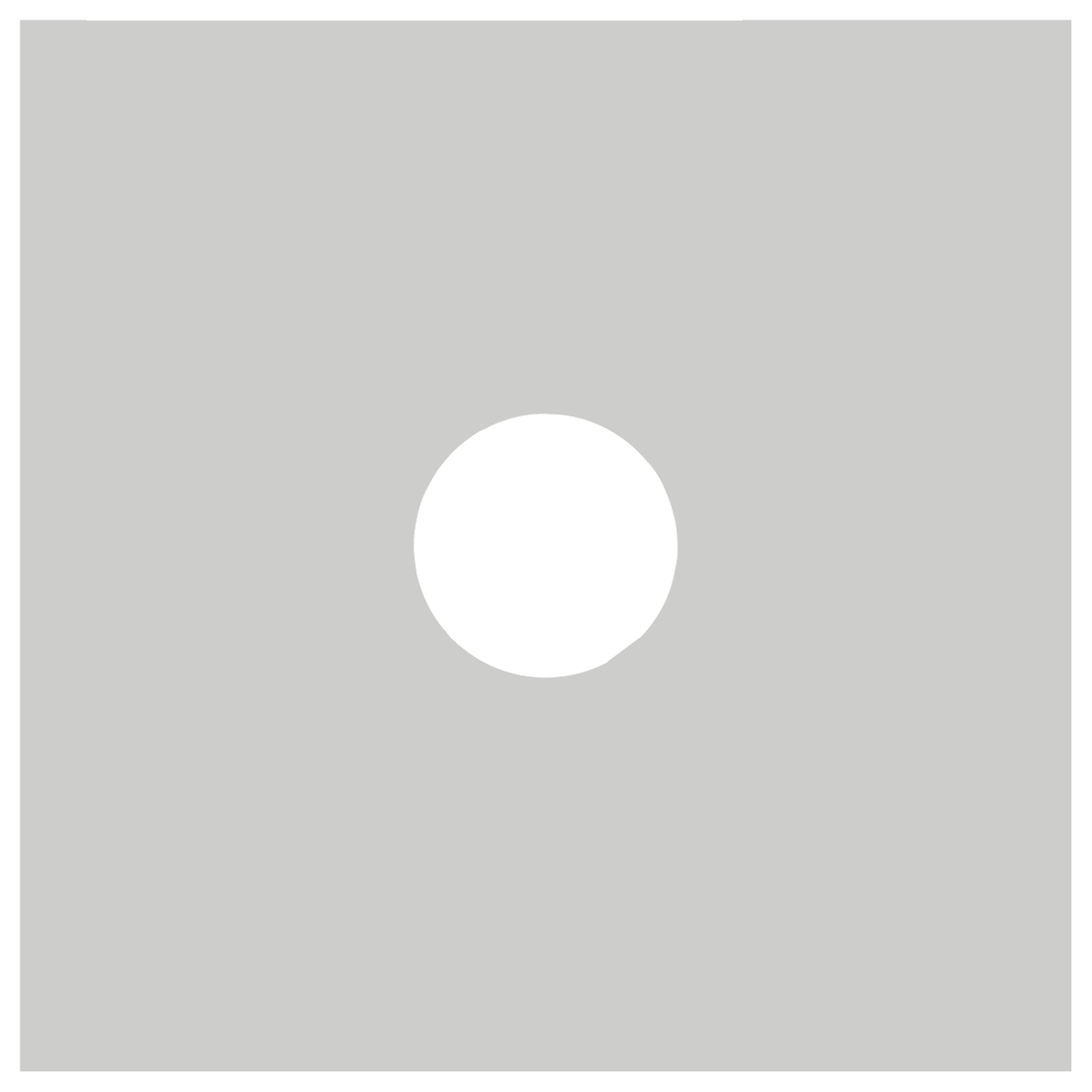}}%
    \put(0.76304948,0.4706857){\color[rgb]{0.09019608,0.08627451,0.07058824}\makebox(0,0)[lt]{\lineheight{1.25}\smash{\begin{tabular}[t]{l}$\gamma$\end{tabular}}}}%
    \put(0.46325324,0.76390334){\color[rgb]{0.09019608,0.08627451,0.07058824}\makebox(0,0)[lt]{\lineheight{1.25}\smash{\begin{tabular}[t]{l}$\delta$\end{tabular}}}}%
    \put(0.1935306,0.4716255){\color[rgb]{0.09019608,0.08627451,0.07058824}\makebox(0,0)[lt]{\lineheight{1.25}\smash{\begin{tabular}[t]{l}$\alpha$\end{tabular}}}}%
    \put(0.47453084,0.19250486){\color[rgb]{0.09019608,0.08627451,0.07058824}\makebox(0,0)[lt]{\lineheight{1.25}\smash{\begin{tabular}[t]{l}$\beta$\end{tabular}}}}%
    \put(0,0){\includegraphics[width=\unitlength,page=2]{nopX.pdf}}%
  \end{picture}%
\endgroup%

	\caption{A weakly consistent dimer model on a torus with a disk taken out which has no perfect matching. Opposite dashed edges are identified.}
	\label{fig:cons-tor-no-perf}
\end{figure}
	Consider the dimer model on a torus pictured in Figure~\ref{fig:cons-tor-no-perf}. It is immediate that any perfect matching of this dimer model must contain one of the arrows of its digon.  A short check verifies that a perfect matching must also contain the arrows $\{\alpha, \beta\}$ or the arrows $\{\gamma, \delta\}$.  However, this prevents any arrow of the face appearing in a corner of the diagram from being in a perfect matching.  Hence, this dimer model has no perfect matching.
	We remark that this dimer model is obtained by taking a weakly consistent dimer model on a torus, and replacing a counter-clockwise square with a variant of the dimer model of Figure~\ref{fig:non-red-annulus}.
\end{example}

Example~\ref{ex:cons-tor-no-perf} raises a question: what sort of conditions may we impose on a weakly consistent dimer model to necessitate the existence of some perfect matching?
In particular, does any weakly consistent dimer model with no digons have a perfect matching?

\subsection{Nondegeneracy}
\label{sec:nondeg}

In the disk and torus case, an important idea is \textit{nondegeneracy} of dimer models. We will define nondegeneracy and prove a simple result which will be used in Section~\ref{sec:3CY}.
\begin{defn}\label{defn:nondeg}
	A dimer model $Q$ is \textit{nondegenerate} if every arrow is contained in a perfect matching. Otherwise, it is \textit{degenerate}.
\end{defn}

In the disk and torus case, nondegeneracy is implied by weak consistency~\cite[Proposition 6.2]{XIUM}~\cite{CKP}. In the general case, this is not true.
For example, the weakly consistent dimer model in Figure~\ref{fig:cons-tor-no-perf} has no perfect matchings, hence is certainly degenerate. The middle of Figure~\ref{fig:noeth-ann} shows a weakly consistent dimer model which has a perfect matching but is still degenerate. Nondegeneracy will feature prominently in Section~\ref{sec:3CY} and Section~\ref{sec:reduce}. 
Figure~\ref{fig:perf-match-not-cons} gives an example of a disk model which is nondegenerate but not weakly consistent.
See Example~\ref{ex:noeth-fin} for multiple examples of nondegenerate weakly consistent dimer models on annuli.
In the rest of this paper, we will see that nondegeneracy is a useful condition that allows us to generalize results from the disk and torus case, motivating the following definition.

\begin{defn}
	A dimer model is \textit{strongly consistent} if it is weakly consistent and nondegenerate.
\end{defn}

The following lemma generalizes the well-known situation in the torus and disk literature. See, for example,~\cite[\S2.3]{XBroomhead2009} in the torus case and~\cite[Proposition 3.1]{XPressland2019} in the disk case.
\begin{lemma}\label{lem:posgrad}
	If $Q$ is {finite and} strongly consistent, then $A_Q$ (and hence $\widehat A_Q$) admits a $\mathbb Z$-grading such that
	\begin{enumerate}
		\item Every nonconstant path has a positive degree, and
		\item Every face-path has the same degree.
	\end{enumerate}
\end{lemma}
\begin{proof}
	Let $C$ be a collection of all perfect matchings on $Q$. Since $Q$ is finite, $C$ has finite cardinality. Given a path $p$ of $Q$, we give $p$ the grading
	\[G(p)=\sum_{\mathcal M\in C}\mathcal M(p),\]
	where $\mathcal M(p)$ is the number of arrows of $p$ which are in $\mathcal M$.
	Note that, for any perfect matching $\mathcal M$, the quantity $\mathcal M(p)$ is unchanged by applying a basic morph to $p$. This means that the quantity $G(p)$ is a well-defined number of the equivalence class of $p$. It is clear that if $p$, $q$, and $qp$ are paths, then $G(p)+G(q)=G(qp)$. It follows that $G$ gives a positive $\mathbb Z$-grading on $A$ through which every arrow is given a positive degree.
	The second statement follows because the degree of any face-path is equal to the number of perfect matchings on $Q$.
\end{proof}
\begin{lemma}\label{lem:mabel}
	Let $Q$ be finite and strongly consistent. Then an arbitrary element $x$ of the completed dimer algebra $\widehat A_Q$ may be represented as
	\[x=\sum_{v,w\in Q_0}\sum_{C:v\to w}\sum_{m\geq0}a_{C,m}[r_Cf^m],\]
	where the second sum iterates over all homotopy classes of paths from $v$ to $w$, and $a_{C,m}\in \mathbb{C}$. Moreover, this element is zero if and only if all coefficients $a_{C,m}$ are zero.
\end{lemma}
\begin{proof}
Since $Q$ is strongly consistent, we may fix a positive $\mathbb Z$-grading $G$ of the completed dimer algebra $\widehat A_Q$ as in Lemma~\ref{lem:posgrad}.
	We first show that an arbitrary $x\in\widehat A_Q$ is of the desired form.
	By definition, an element $x\in\widehat A_Q$ is of the form 
	\begin{equation}\label{ghqw}x=\sum_{p\text{ a path of }Q}a_p[p],\end{equation}
	for some coefficients $a_p\in\mathbb C$.
	Define $G_{\min}:=\textup{min}\{G(\alpha)\ :\ \alpha\in Q_1\}$.

	Let $p$ be any path of $Q$.
Let $C$ be the homotopy class of $p$. By path-consistency, write $[p]=[r_Cf^{m}]$ for some value $m$ and a minimal path $r_C$ in homotopy class $C$.
	Let $M_{C,m}$ be an integer greater than $\frac{G(p)}{G_{\min}}$. Then any path of $Q$ with at least $M_{C,m}$ arrows must have a grading greater than $G(p)$, hence every path equivalent to $p$ must have less than $M_{C,m}$ arrows. Such paths are finite in number, hence the path equivalence class of $p$ is finite.
	Then we may define $a_{C,m}:=\sum_{q\text{ equivalent to }p}a_q$. The desired
\[x=\sum_{v,w\in Q_0}\sum_{C:v\to w}\sum_{m\geq0}a_{C,m}[r_Cf^m]\]
	now follows from rearranging the terms of equation~\eqref{ghqw}.

	Now suppose that some $a_{C',m'}$ is nonzero. Define a real number $G'>G(r_C'f^{m'})$. We show that the graded part of $x$ below $G'$ must be nonzero, and hence that $x$ itself is nonzero. Define for any homotopy class $C$ and $m\geq0$
	\[a'_{C,m}:=\begin{cases}
		a_{C,m}&G(r_Cf^m)<G' \\
		0&\text{else.}
	\end{cases}\]
	Then the sum
	\begin{equation}\label{greatscott}x_{G'}:=\sum_{v,w\in Q_0}\sum_{C:v\to w}\sum_{m\geq0}a'_{C,m}[r_Cf^m]\end{equation}
	gives the graded part of $x$ below $G'$.

	Set $M'>\frac{G'}{G_{\min}}$. Then any path of $Q$ with grading less than $G'$ has less than $M'$ arrows. So, the sum~\eqref{greatscott} has a finite number of nonzero summands. Then the sum~\eqref{greatscott} is in the noncompleted path algebra of $Q$. Since $a_{C',m'}$ is nonzero and in this sum, then $x_{G'}$ is not in the ideal $I_Q$ of the noncompleted dimer algebra.
	Moreover, since every path with at least $M'$ arrows has a grading greater than $G'$, the sum $x_{G'}$ cannot be in the completion of $I_Q$ with respect to the arrow ideal. This shows that $x_{G'}$ is nonzero in the completed dimer algebra, hence $x$ is nonzero in the completed dimer algebra.
\end{proof}

\section{Bimodule Internal 3-Calabi-Yau Property}
\label{sec:3CY}\label{sec:3cy}

We show that finite strongly consistent (completed or noncompleted) dimer models are bimodule internally 3-Calabi-Yau with respect to their boundary idempotent in the sense of~\cite{XPressland2015}. As an application, we use a result from~\cite{AIRX} to show that the Gorenstein-projective module category over the completed boundary algebra of a finite strongly consistent dimer model $Q$ satisfying some extra conditions categorifies the cluster algebra given by the ice quiver of $Q$. We give new examples of suitable dimer models.

The technical part of this section follows~\cite[\S3]{XPressland2019} (in the disk case) and~\cite[\S7]{XBroomhead2009} (in the torus case).
Note that the former writes $A$ for the completed dimer algebra (or Jacobian algebra) $\widehat A_Q$, and the latter deals only with the noncompleted dimer algebra. 

Throughout this section, let $Q$ be a finite weakly consistent dimer model. We will eventually pass to the case when $Q$ is in addition nondegenerate. We write $\A$ to denote simultaneously the noncompleted dimer algebra $A_Q$ and the completed dimer algebra $\widehat A_Q$, since the arguments are the same.

\subsection{One-Sided and Two-Sided Complexes}

We begin by defining some $(\A,\A)$-bimodules. If $v$ is a vertex of $Q$, define $T_v$ (respectively $H_v$) to be the set of all arrows with tail (respectively head) $v$. Define $Q_0^m$ 
to be the set of internal 
vertices of $Q$.  Let $Q_1^m$ 
be the set of internal 
arrows of $Q$.
Define $T_v^m$ and $H_v^m$ to be the internal arrows of $T_v$ and $H_v$, respectively.
We define vector spaces
\[ T_3:=\oplus_{v\in Q_0^m}\mathbb C\omega_v,\ \  
T_2:=\oplus_{\alpha\in Q_1^m}\mathbb C\rho_\alpha,\ \ 
T_1:=\oplus_{\alpha\in Q_1}\mathbb C\alpha,\ \ 
T_0:=\oplus_{v\in Q_0}\mathbb C e_v.\]
The $(\A,\A)$-bimodule structures are given by
\[e_v\cdot\omega_v\cdot e_v=\omega_v,\ \ e_{t(\alpha)}\cdot\rho_\alpha\cdot e_{h(\alpha)}=\rho_\alpha,\ \ 
e_{h(\alpha)}\cdot\alpha\cdot e_{t(\alpha)}=\alpha,\ \ e_v\cdot e_v\cdot e_v=e_v.\]
All other products with the generators of $T_i$ are zero.
In this section, all tensor products are over $T_0$ unless otherwise specified. We consider the following complex, which we will call the \textit{two-sided complex}.
\begin{equation}\label{eqn:cmp2}0\to \A\otimes T_3\otimes \A\xrightarrow{\mu_3}\A\otimes T_2\otimes \A\xrightarrow{\mu_2}\A\otimes T_1\otimes \A\xrightarrow{\mu_1}\A\otimes T_0\otimes \A\xrightarrow{\mu_0}\A\to0\end{equation}
	We consider $\A\otimes T_0\otimes \A$ to be the degree-zero term of the complex and
we define the maps $\mu_i$ as follows.
First, define a function $\partial$ on the arrows of $Q$ by
\[\partial(\alpha)=\begin{cases}R_\alpha^{cc}-R_\alpha^{cl}&\alpha\text{ is an internal arrow}\\
	R_\alpha^{cc}&\alpha\text{ is a boundary arrow in a counter-clockwise face}\\
-R_\alpha^{cl}&\alpha\text{ is a boundary arrow in a clockwise face}.\end{cases}\]
For a path $p=\alpha_m\dots\alpha_1$, we define 
\[\Delta_\alpha(p)=\sum_{\alpha_i=\alpha}\alpha_m\dots\alpha_{i+1}\otimes\alpha\otimes\alpha_{i-1}\dots\alpha_1\]
and extend by linearity and continuity to obtain a map $\Delta_\alpha:\mathbb C\langle\langle Q\rangle\rangle\to \A\otimes T_1\otimes \A$. Then we define
\[\mu_3(x\otimes\omega_v\otimes y)=\sum_{\alpha\in T_v^{m}}x\otimes\rho_\alpha\otimes \alpha y-\sum_{\beta\in H_v^{m}}x\beta\otimes\rho_\beta\otimes y,\]
\[\mu_2(x\otimes\rho_\alpha\otimes y)=\sum_{\beta\in Q_1}x\Delta_\beta\big(\partial(\alpha)\big)y,\textup{ and }\]
\[\mu_1(x\otimes\alpha\otimes y)=x\otimes e_{h(\alpha)}\otimes\alpha y-x\alpha\otimes e_{t(\alpha)}\otimes y.\]
Since the tensor products are over $T_0$, there is a natural isomorphism $\A\otimes T_0\otimes \A\cong \A\otimes \A$. This may be composed with the multiplication map to obtain $\mu_0$.

The following was shown for Jacobian ice quivers, and completed dimer algebras are special case of these.  Nevertheless, the same proof applies in the noncompleted case, so we cite it here without this limitation. 

\begin{thm}[{\cite[Theorem 5.6]{XPressland2015}}]
	\label{thm:bim3cy}
	If the complex~\eqref{eqn:cmp2} is exact, then $\A$ is bimodule internally 3-Calabi-Yau with respect to the idempotent given by the sum of all frozen vertex simples.
\end{thm}

Our goal is to show that when $Q$ is strongly consistent, the complex~\eqref{eqn:cmp2} is exact. 
To do this, we will first define a version of~\eqref{eqn:cmp2} which is merely a complex of modules, rather than bimodules, and prove that exactness of this one-sided complex is equivalent to exactness of the two-sided complex~\eqref{eqn:cmp2} in the non-degenerate case. We will then show that the one-sided complex is exact to finish the proof.

Use the quotient map $\A\to \A/\textup{rad}\ \A\cong T_0$ to consider $T_0$ as an $(\A,\A)$-bimodule. Using this bimodule structure we consider the functor $\mathcal F=-\otimes_\A T_0$ from the category of $(\A,\A)$-bimodules to itself. We apply this to the complex~\eqref{eqn:cmp2} and note that 
$T_i\otimes\A\otimes_\A T_0\cong T_i$ and $\A\otimes_{T_0}T_0\cong\A$ to get the complex

\begin{equation}\label{eqn:cmp1}0\to \A\otimes T_3\xrightarrow{\mathcal F(\mu_3)}\A\otimes T_2\xrightarrow{\mathcal F(\mu_2)}\A\otimes T_1\xrightarrow{\mathcal F(\mu_1)} \A\xrightarrow{\mu_0}T_0\to0\end{equation}
We forget the right $\A$-module structure and treat this as a complex of left $\A$-modules, which we refer to as the \textit{one-sided complex}.

\begin{remk}\label{rem:bimodule}
The two-sided sequence \eqref{eqn:cmp2} is indeed a complex of $\A$-modules, which is exact in degrees 1, 0, -1.  This can be seen as follows.  In the noncompleted case $\A = \A_Q$  this follows by the work of Ginzburg \cite[Proposition 5.1.9 and Theorem 5.3.1]{Ginz}, see also the exposition in \cite[Section 7.1]{XBroomhead2009}.  In the completed case $\A = \widehat \A_Q$ this statement appears in  \cite[Lemma 5.5]{XPressland2015}, however as noted by Pressland \cite{PRESSPR} there is a gap in the proof.  Namely,  the citation of the results by Butler and King for the exactness in degrees 1, 0, -1 applies only for the noncompleted algebra.  Instead, following Pressland \cite{PRESSPR} we can deduce exactness of the one-sided sequence \eqref{eqn:cmp1} from \cite[Proposition 3.3]{BIRS} in degrees 1, 0, -1.  Then the same holds for the two-sided complex by taking inverse limits as in the proof of \cite[Lemma 4.7]{XPressland2017}, which works degree by degree. 
\end{remk}

We would like to show exactness of the one-sided complex in degrees $-3$ and $-2$, so we explicitly write the maps $\F(\mu_3)$ and $\F(\mu_2)$.
\begin{align*}
	\F(\mu_3):x\otimes\omega_v&\mapsto
		-\sum_{\beta\in H_v^m}x\beta\otimes\rho_\beta\\
	\F(\mu_2):x\otimes\rho_\alpha&\mapsto\sum_{\beta\in Q_1}x\Delta_\beta^r(\partial(\alpha))
\end{align*}
where, for $p=\alpha_m\dots\alpha_1$,
\[\Delta_\beta^r(p)=\sum_{\alpha_i=\beta}\alpha_m\dots\alpha_{i+1}\otimes\alpha_i.\]

We now do some calculations for $\mu_3$.
\begin{align*}
	\mu_3:x\otimes\omega_v\otimes y&\mapsto \sum_{\alpha\in T_v^{m}}x\otimes\rho_\alpha\otimes\alpha y-\sum_{\beta\in H_v^{m}}x\beta\otimes\rho_\beta\otimes y\\
	&=\sum_{\alpha\in T_v^{m}}x\otimes\rho_\alpha\otimes\alpha y-\left(\sum_{\beta\in H_v^m}x\beta\otimes\rho_\beta\right)\otimes y\\
	&=\left(\sum_{\alpha\in T_v^{m}}x\otimes\rho_\alpha\otimes\alpha\right)y+(\F(\mu_3)(x\otimes\omega_v))\otimes y
\end{align*}
We perform a similar calculation for $\mu_2$.
\begin{align*}
	\mu_2:x\otimes\rho_\alpha\otimes y&\mapsto\sum_{\beta\in Q_1}x\Delta_\beta\big(\partial(\alpha)\big)y\\
	&=\left(\sum_{\beta\in Q_1}x\Delta_\beta^r(p)\right)\otimes y\\
	&=(\F(\mu_2)(x\otimes\rho_\alpha))\otimes y
\end{align*}
We have shown that, for $j\in\{2,3\}$, we have
\begin{equation}\label{eqn:idd}\mu_j:u\otimes y\mapsto((\F(\mu_j)(u))\otimes y+\left(\sum_{v,y'}v\otimes y'\right)y,\end{equation}
	where 
	\begin{itemize}
		\item $u$ is in either $\A\otimes T_3$ (if $j=3$) or $\A\otimes T_2$ (if $j=2$), 
		\item $v$ ranges across some elements of $\A\otimes T_2$ (if $j=3$) or $\A\otimes T_1$ (if $j=2$), and
		\item $y'$ ranges across some arrows of $\A$.
	\end{itemize}

\subsection{Proving 3-Calabi-Yau Property for Strongly Consistent Models}

We now show that exactness of the bimodule complex~\eqref{eqn:cmp2} is equivalent to exactness of the one-sided complex~\eqref{eqn:cmp1} when $Q$ is strongly consistent. We will then show that the one-sided complex is exact, and hence that the completed dimer algebra is bimodule internally 3-Calabi-Yau with respect to its boundary idempotent.

We consider the bimodules $T_i$ to be $\mathbb Z$-graded as follows. All elements of $T_3$ and $T_0$ have degree 0. An element $\rho_\alpha$ in $T_2$ corresponding to an arrow $\alpha\in Q$ is given the negative of the grading of $\alpha$ in $Q$. An element $\alpha$ in $T_1$ corresponding to an arrow $\alpha\in Q$ is given the grading of $\alpha$ in $Q$.

We extend the grading on $\A$ given by Lemma~\ref{lem:posgrad} with the gradings on $T_i$ described above to a $\mathbb Z$-grading on the bimodule complex $\A\otimes T_*\otimes \A$ by adding the grading in each of the three positions. 

We remark that the grading on $\A\otimes T_2\otimes \A$ is not positive. The minimum possible degree of an element of $\A\otimes T_2\otimes \A$ is $-m$, where $m$ is the maximum possible degree of an arrow of $\A$.
Moreover, every face-path has degree equal to the number of perfect matchings on $Q$.

\begin{lemma}\label{lem:gradres}
	The maps $\mu_3$ and $\mu_2$ are maps of graded bimodules. In other words, they map homogeneous elements to homogeneous elements. 
\end{lemma}
\begin{proof}
	First, we consider $\mu_3$. Any summand of $\mu_3(x\otimes\omega_v\otimes y)$ is of the form $x\otimes\rho_\alpha\otimes\alpha y$ or $x\alpha\otimes\rho_\alpha\otimes y$ for some arrow $\alpha$. The $\alpha$ on the left or right summands adds some number $m$ to the degree, and the $\rho_\alpha$ in the middle subtracts that same degree, so the degree of $\mu_3(x\otimes\omega_v\otimes y)$ is the same as the degree of $x\otimes\omega_v\otimes y$. 

	We now consider $\mu_2$. Any summand of $\mu_2(x\otimes\rho_\alpha\otimes y)$ is of the form $xR''\otimes\beta\otimes R'y$ for some arrow $\beta$ and some paths $R'$ and $R''$ such that $R''\beta R'$ is some return path $R_\alpha$ of $\alpha$. Compared to $x\otimes\rho_\alpha\otimes y$, we are replacing the (negative) grading of middle term $\rho_\alpha$ with the (positive) grading of the path $R_\alpha=R''\beta R'$. The result is that the grading has increased by the grading of $\alpha R_\alpha$ in $\A$. This is a face-path, and all face-paths have the same grading. It follows that $\mu_2$ has the effect of increasing the grading of a homogeneous element by the grading of a face-path in $\A$.
\end{proof}
\begin{thm}[{\cite[Lemma 4.7]{XPressland2017}}]\label{thm:wherebefore}
	If the one-sided complex~\eqref{eqn:cmp1} is exact for $\A=\widehat A_Q$, then the two-sided complex~\eqref{eqn:cmp2} is exact for $\A=\widehat A_Q$.
\end{thm}
To show a version of Theorem~\ref{thm:wherebefore} when $\A$ is the noncompleted dimer algebra $A_Q$, we assume in addition nondegeneracy.
\begin{prop}\label{prop:1iff2}
	Suppose $Q$ is strongly consistent. If the one-sided complex~\eqref{eqn:cmp1} is exact, then the two-sided complex~\eqref{eqn:cmp2} is exact.
\end{prop}
\begin{proof}
	We follow the proof of~\cite[Proposition 7.5]{XBroomhead2009}.
	Suppose the one-sided complex~\eqref{eqn:cmp1} is exact. 
	We know by Remark~\ref{rem:bimodule} that the bimodule complex~\eqref{eqn:cmp2} is exact in every degree except for, possibly, $-3$ and $-2$. It remains to show that $\ker\mu_2\subseteq\textup{\textup{im}}\mu_3$ and $\ker\mu_3\subseteq\textup{im}\mu_4=0$, where $T_4:=0$ and $\mu_4:T_4\to \A\otimes T_3\otimes \A$ is the zero map.

	Let $j\in\{3,2\}$.
	Let $\phi_0$ be a nonzero element of $\A\otimes T_j\otimes \A$ which is in the kernel of $\mu_j$. We show that $\phi_0$ is in the image of $\mu_{j+1}$. Since the grading is respected by the map $\mu_j$ (Lemma~\ref{lem:gradres}), we may assume that $\phi_0$ is homogeneous of some grade $d$. We may organize the terms of $\phi_0$ by the degree of the term in third position.
	\[\phi_0=\sum_{y\in Y}u_y\otimes y+\{\textup{terms with strictly higher degree in the third position}\},\]
	where $Y$ is a nonempty linearly independent set of monomials in the graded piece $\A^{(d_0)}$ with least possible degree, and $u_y\in(\A\otimes T_j)^{(d-d_0)}$. 
	Applying the map $\mu_j$ and using~\eqref{eqn:idd} and Lemma~\ref{lem:posgrad}, we see that $\mu_j(\phi_0)=0$ is equivalent to the condition
	\[0=\sum_{y\in Y}(\F(\mu_j)(u_y))\otimes y+\{\textup{terms with strictly higher degree in the third position}\}.\]
	Since the monomials $y\in Y$ are linearly independent, this implies that for all $y\in Y$ we have $\mathcal F(\mu_j)(u_y)=0$. 
	Using the exactness of the one-sided complex, we conclude that there exist elements $v_y\in(\A\otimes T_{j+1})^{(d-d_0)}$ such that $\F(\mu_j)(v_y)=u_y$ for each $y\in Y$. We construct an element
	\[\psi_1=\sum_{y\in Y}v_y\otimes y\in(\A\otimes T_{j_1}\otimes\A)^{(d)}\]
	and apply $\mu_{j+1}$ to get
	\[\mu_{j+1}(\psi_1)=\sum_{y\in Y}u_y\otimes y+\{\textup{terms with strictly higher degree in the third position}\}.\]
	We observe that $\phi_1:=\phi_0-\mu_{j+1}(\psi_1)$ is in the kernel of $\mu_{j}$ and that its terms have strictly higher degree in the third position than $\phi_0$. We iterate the procedure, noting that the degree in the third position is strictly increasing but is bounded above by the total degree $d$ if $j=3$, and by $d$ plus the maximum degree of an arrow of $\A$ if $j=2$. Hence, after a finite number of iterations we get $\phi_r=\phi_0-\sum_{i=1}^r\mu_{j+1}(\psi_i)=0$. We conclude that $\phi_0=\mu_{j+1}(\sum_{i=1}^r\psi_i)$ and that the complex~\eqref{eqn:cmp2} is exact at $\A\otimes T_j\otimes \A$.
\end{proof}

We now know that in order to show exactness of the bimodule complex~\eqref{eqn:cmp2}, it suffices to show exactness of the one-sided complex~\eqref{eqn:cmp1}.
We follow~\cite[\S3]{XPressland2019} and consider exactness of~\eqref{eqn:cmp1} vertex by vertex.
If $v$ is a vertex, let $S_v:=e_vT_0e_v$ be the simple module at $v$.

We may consider
the complex~\eqref{eqn:cmp1} as a complex of $(\A,T_0)$-bimodules, which we denote by $\mathbf P^1$. Since $T_0=\oplus_{v\in Q_0}S_v$, we have
\[\mathbf P^1=\oplus_{v\in Q_0}\mathbf P^1e_v\]
as a complex of left $\A$-modules.
Hence, in order to show exactness of the complex of left $\A$-modules~\eqref{eqn:cmp1}, it suffices to show that for any vertex $v$ of $Q$, the complex of left $\A$-modules $\mathbf P^1e_v$ is exact. We rewrite this complex $\mathbf P^1e_v$ as

\begin{equation}
	\label{eqn:cmp4}
	0\to X_3\xrightarrow{\bar\mu_3}X_2\xrightarrow{\bar\mu_2}X_1\xrightarrow{\bar\mu_1}A\otimes T_0\otimes S_v\to S_v\to0,
\end{equation}
where the spaces $X_i$ are defined as
\begin{align*}
	X_1&:=\bigoplus_{\beta\in T_v}\A e_{h(\beta)},\\
	X_2&:=\bigoplus_{\alpha\in H_v^m}\A e_{t(\alpha)},\\
	X_3&:=\begin{cases}
		\A e_v&v\in Q_0^m\\
	0&\textup{else}\end{cases},
\end{align*}
and the maps $\bar\mu_j$ are induced by $\mathcal F(\mu_j)$ under the relevant isomorphisms. We will make explicit the maps $\bar\mu_3$ and $\bar\mu_2$ after introducing some notation. We write a general element $x$ of $\oplus_{\alpha\in H_v^{m}}\A e_{t(\alpha)}$ as $x=\sum_{\alpha\in H_v^{m}}x_\alpha\otimes[\alpha]$, where for any $\alpha\in H_v^{m}$, the summand $x_\alpha\otimes[\alpha]$ refers to the element $x_\alpha\in\A e_{t(\alpha)}$ in the summand of $\oplus_{\alpha\in H_v^{m}}\A e_{t(\alpha)}$ indexed by $\alpha$. Similarly, a general element of $\oplus_{\beta\in T_v}\A e_{h(\beta)}$ will be written as $y=\sum_{\beta\in T_v}y_\beta\otimes[\beta]$.

We define the \textit{right derivative $\partial_\beta^r$ with respect to $\beta$} on a path $\gamma_k\dots\gamma_1$ by
\[\partial_\beta^r(\gamma_k\dots\gamma_1)=\begin{cases}\gamma_k\dots\gamma_2&\gamma_1=\beta\\0&\gamma_1\neq\beta\end{cases}\]
and extend linearly and continuously. Similarly, there is a \textit{left derivative}, defined on paths by
\[\partial_\beta^l(\gamma_k\dots\gamma_1)=\begin{cases}\gamma_{k-1}\dots\gamma_1&\gamma_k=\beta\\0&\gamma_k\neq\beta.\end{cases}\]
Given two arrows $\alpha$ and $\beta$ of $Q$, we observe that
\[\partial_\beta^r(\partial(\alpha))=\partial_\alpha^l(\partial(\beta)).\]
We now calculate:
\begin{align*}
	\bar\mu_2(x)&=\sum_{\beta\in T_v}\left(\sum_{\alpha\in H_v^m}x_\alpha\partial_\beta^r(\partial(\alpha))\right)\otimes[\beta]\\
	\bar\mu_3(x)&=\sum_{\alpha\in H_v^{m}}x\alpha\otimes[\alpha]
\end{align*}

We now finally prove the main result of this section after citing the disk version proven by Pressland.
Note that in~\cite{XPressland2019}, dimer models on a disk are required to have at least three boundary vertices to avoid degenerate cases. This condition is not used in the cited result (see~\cite[Remark 2.2]{XPressland2019}), so we cite it without this limitation. {The results and proofs of~\cite[Theorem 5.6]{XPressland2015} and~\cite[Theorem 3.7]{XPressland2019} are stated only for completed Jacobian algebras, but work also for their noncompleted variants. Hence, we cite this result for the completed and noncompleted algebras simultaneously.}

\begin{thm}[{\cite[Theorem 3.7]{XPressland2019},~\cite[Theorem 5.6]{XPressland2015}}]\label{thm:Calabi-Yau-disk}
	If $Q$ is a path-consistent dimer model on a disk, then the sequence~\eqref{eqn:cmp4} is exact for all $v$, and hence $\A$ is bimodule internally 3-Calabi-Yau with respect to its boundary idempotent.
\end{thm}
In fact, Pressland's proof of Theorem~\ref{thm:Calabi-Yau-disk} works in the general setting, with the stipulation that all computations must be performed as a sum over homotopy classes. We give a shorter proof here, which uses Pressland's result for disk models as well as the theory of dimer submodels developed in Section~\ref{sec:ds}. 
\begin{thm}\label{thm:Calabi-Yau}
	Let $Q$ be a strongly consistent finite dimer model. Then $A_Q$ and $\widehat A_Q$ are bimodule internally 3-Calabi-Yau with respect to their boundary idempotents.
\end{thm}
\begin{proof}
	By Theorem~\ref{thm:bim3cy}, it suffices to prove exactness of the two-sided complex~\eqref{eqn:cmp2}.
	By Proposition~\ref{prop:1iff2} or Theorem~\ref{thm:wherebefore} (depending on whether $\A=A_Q$ or $\A=\widehat A_Q$), it suffices to prove exactness of the one-sided complex~\eqref{eqn:cmp1}. 
	As argued in the text following Proposition~\ref{prop:1iff2}, we may prove this by showing that the complex~\eqref{eqn:cmp4} is exact for any choice of $v\in Q_0$.
	We need only show that $\bar\mu_3$ is injective and that $\ker\bar\mu_2\subseteq\textup{im}\ \bar\mu_3$.

	First, we show injectivity of $\bar\mu_3$. If $v$ is boundary, then this is trivial, so suppose $v\in Q_0^{m}$. Suppose there is some nonzero $x\in\A e_v$ with $0=\bar\mu_3(x)=\sum_{\alpha\in H_v^{m}}x\alpha\otimes[\alpha]$. Write $x=\sum_Cx_C$, where the sum is over homotopy classes of paths in $Q$ starting at $v$ and 
		$x_C=\sum_{m=0}^\infty a_{C,m}r_Cf^m$, where $r_C$ is a minimal path in the homotopy class $C$ and $a_{C,m}\in\mathbb C$. 

		If $C$ and $C'$ are different homotopy classes, then $C\alpha$ and $C'\alpha$ are different homotopy classes for any arrow $\alpha$. In particular, the summands of $\bar\mu_3(x_C)$ and $\bar\mu_3(x_{C'})$ corresponding to each arrow $\alpha\in H_v^{m}$ are in different homotopy classes. 
		Since $\bar\mu_3(x)=0$, this means that $\bar\mu_3(x_C)=0$ for all homotopy classes $C$. Since $x\neq0$, we may choose a homotopy class $C$ and $m\geq0$ such that $a_{C,m}\neq0$.
		{Then $0=\bar\mu_3(x_C)=\sum_{\alpha\in H_v^{m}}x_C\alpha\otimes[\alpha]$. Then $0=x_C\alpha=\sum_{m=0}^\infty a_{C,m}r_C\alpha f^m$ for all $\alpha\in H_v^{m}$.
		Hence, $a_{C,m}=0$ for all $m\geq0$ (if $\A=A_Q$, this is immediate by cancellativity, and if $\A=\widehat A_Q$ this follows from Lemma~\ref{lem:mabel}) and we have $x_C=0$. This contradicts our choice of $C$ and completes the proof of injectivity of $\bar\mu_3$.}

	We now prove that the image of $\bar\mu_3$ contains the kernel of $\bar\mu_2$. 
	Take a nonzero element $x=\sum_Cx_C$ of $\ker\bar\mu_2$, where the sum is over homotopy classes of paths in $Q$ starting at the tail of some arrow $\alpha\in H_v^{m}$ and $x_C=\sum_{\alpha\in H_v^{m}}\left(a_{C,\alpha,m}r_Cf^m\otimes[\alpha]\right)$, where $r_C$ is a minimal path in $Q$ homotopic to $C$ and $a_{C,\alpha,m}$ are some coefficients in $\mathbb C$.
	We wish to find $y\in X_3$ such that $\bar\mu_3(y)=x$. It suffices to find $y_C$ for each homotopy class $C$ with $\bar\mu_3(y_C)=x_C$ (note that paths in $y_C$ will not be in the homotopy class $C$), so fix a homotopy class $C$.
	Pick a vertex $\tilde v$ of $\widetilde Q$ corresponding to $v$. For $\alpha\in H_v^{m}$, let $\tilde\alpha$ be the corresponding arrow of $\widetilde Q$ ending at $\tilde v$. Choose a (finite) disk submodel $Q^C$ of $\widetilde Q$ containing $\tilde v$ and a minimal path $r_{C\alpha}$ in $Q$ homotopic to $C\alpha$ for each $\alpha\in H_v^{m}$. Lift each $r_{C\alpha}$ to a minimal path $\tilde r_{C\alpha}$ beginning at $t(\tilde\alpha)$.

	{Similarly, lift $x_C$ to $\tilde x_C:=\sum_{\alpha\in H_v^{m}}(a_{C,\alpha,m}\tilde r_Cf^m\otimes[\alpha])$, where $\tilde r_C$ is the lift of $r_C$ to $\widetilde Q$ beginning at $\tilde v$. We claim that $\tilde x_C$ is in the kernel of the lift $\widetilde{\bar\mu_2}$.}

	Choose coefficients $b_{\beta,D,m}$ such that $\bar\mu_2(x_C)=\sum_{\beta\in T_v}\sum_{D:h(\beta)\to?}\sum_{m\geq0}b_{\beta,D,m}[r_Df^m]\otimes[\beta]$, where the second sum iterates over homotopy classes of paths of $Q$ starting at $h(\beta)$. For $\beta\in T_v$, let $\tilde\beta\in T_{\tilde v}\subseteq \widetilde Q_1$ be the corresponding arrow of $\widetilde Q_1$ starting at $\tilde v$. Similarly, for any $r_D$ starting at some $h(\beta)$, lift it to $\tilde r_D$ starting at $h(\tilde\beta)$. It follows from the definition of $\bar\mu_2$ and $\widetilde{\bar\mu_2}$, then, that
	\[\widetilde{\bar\mu_2}(\tilde x_C)=\sum_{\beta\in T_{\tilde v}}\sum_{D:h(\tilde\beta)\to?}\sum_{m\geq0}b_{\beta,D,m}[\tilde r_Df^m]\otimes[\tilde\beta].\]
	But $\bar\mu_2(x_C)=0$, so for each $\beta\in T_v$, we have
	$\sum_{D:h(\beta)\to?}\sum_{m\geq0}b_{\beta,D,m}[r_Df^m]=0$. Then each coefficient $b_{\beta,D,m}=0$ (again this is immediate if $\A=A_Q$ and follows from Lemma~\ref{lem:mabel} if $\A=\widehat A_Q$). Then the above formula shows that $\widetilde{\bar\mu_2}(\tilde x_C)=0$.
	Then $\tilde x_C$ is in the kernel of $\widetilde{\bar\mu_2}$.

	 By Corollary~\ref{cor:dimer-submodel-consistent}, $Q^C$ is weakly consistent. By Theorem~\ref{thm:submodel-path-equivalence}, two paths of $Q^C$ are equivalent in $Q^C$ if and only if they are equivalent as paths of $\widetilde Q$. By Theorem~\ref{thm:Calabi-Yau-disk}, the exact sequence corresponding to~\eqref{eqn:cmp4} of $Q^C$ must be exact. In particular, we get some $\tilde y_C$ in $\tilde X_3$ which maps to $\tilde x_C$ through $\widetilde{\bar\mu_3}$. Then $\tilde y_C$ considered as a sum of paths in $\widetilde Q$ descends to some $y_C$ in $X_3$ with $\bar\mu_3(y_C)=x_C$ and the proof is complete.
\end{proof}

\subsection{Categorification}

In certain special cases Theorem~\ref{thm:Calabi-Yau} provides examples of categorifications of cluster algebras. We loosely model this section after~\cite[\S4]{XPressland2019}. We avoid defining technical terms used in this subsection, and instead refer to~\cite{XPressland2019} for more information.

\begin{thm}[{\cite[Theorem 4.1 and Theorem 4.10]{AIRX}\label{AIR}\label{thm:cat-lemma}}]
	Let $A$ be an algebra and $e\in A$ an idempotent. If $A$ is Noetherian, $\underline A=A/\langle e\rangle$ is finite dimensional, and $A$ is bimodule 3-Calabi-Yau with respect to $e$, then
	\begin{enumerate}
		\item $B=eAe$ is Iwanaga-Gorenstein with Gorenstein dimension at most 3,
		\item $eA$ is a cluster-tilting object in the Frobenius category of Gorenstein projective modules $\textup{GP}(B)$,
		\item The stable category $\underline{\textup{GP}}(B)$ is a 2-Calabi-Yau triangulated category, and
		\item The natural maps $A\to\textup{End}_B(eA)^{op}$ and $\underline{A}\to\underline{\textup{End}}_B(eA)^{op}$ are isomorphisms.
	\end{enumerate}
\end{thm}

If $Q$ is a strongly consistent dimer model and $e$ is its boundary idempotent, then $\widehat A_Q$ is bimodule internally 3-Calabi-Yau with respect to $e$ by Theorem~\ref{thm:Calabi-Yau}. 
Hence, in order to apply Theorem~\ref{thm:cat-lemma}, it suffices to check that $\widehat A_Q$ is Noetherian and that $\widehat A/\langle e\rangle$ is finite dimensional as a vector space.

\begin{cor}\label{thm:cat-thm}
	Let $Q$ be a finite strongly consistent dimer model 
	with boundary idempotent $e$. Suppose that $\widehat A_Q/\langle e\rangle$ is finite dimensional (hence that $\widehat A_Q/\langle e\rangle\cong A_Q/\langle e\rangle$), and that $\widehat A_Q$ is Noetherian. Write $\widehat B_Q:=e\widehat A_Qe$ for the completed boundary algebra of $Q$.
	\begin{enumerate}
		\item $\widehat B_Q$ is Iwanaga-Gorenstein with Gorenstein dimension at most 3,
		\item $e\widehat A_Q$ is a cluster-tilting object in the Frobenius category of Gorenstein projective modules $\textup{GP}(\widehat B_Q)$,
		\item The stable category $\underline{\textup{GP}}(B_Q)$ is a 2-Calabi-Yau triangulated category, and
		\item \label{en:4}The natural maps $\widehat A_Q\to\textup{End}_{\widehat B_Q}(e\widehat A_Q)^{op}$ 
		and $A_Q/\langle e\rangle\to\widehat A_Q/\langle e\rangle\cong\underline{\textup{End}}_{\widehat B_Q}(e\widehat A_Q)^{op}$ are isomorphisms.
	\end{enumerate}

\end{cor}

	If $Q$ has no 1-cycles or 2-cycles, then 
it follows from Corollary~\ref{thm:cat-thm}~\eqref{en:4}
that ${\textup{GP}}(\widehat B_Q)$ is a categorification of the cluster algebra whose seed is given by the underlying ice quiver of $Q$.
	Section~\ref{sec:red} gives us tools to \textit{reduce} a dimer model by removing digons while preserving the dimer algebra.

\begin{remk}\label{remk:gggg}
	In the setting of Theorem~\ref{thm:cat-lemma}, if $e''$ is an idempotent orthogonal to $e$ then $\widehat A$ is 3-Calabi-Yau with respect to $e':=e+e''$ by~\cite[Remark 2.2]{XPressland2015}. This means that if $Q$ is a strongly consistent dimer model and $\widehat A_Q$ is Noetherian, then even if $\widehat A_Q/\langle e\rangle$ is not Noetherian we may use a larger idempotent $e'$, where some of the internal vertices are seen to be frozen, in order to apply Theorem~\ref{thm:cat-lemma}. In this way, we can get categorifications even from, for example, a consistent dimer model on a torus. See~\cite[\S6]{AIRX}.
\end{remk}

\begin{defn}\label{defn:bf}
	We say that a dimer model $Q$ is \textit{boundary-finite} if $\widehat A_Q/\langle e\rangle$ is finite-dimensional. We say that $Q$ is \textit{Noetherian} if $\widehat A_Q$ is Noetherian.
\end{defn}

Hence, if $Q$ is a finite strongly consistent dimer model, then in order to apply Corollary~\ref{thm:cat-thm} we must only check Noetherianness and boundary-finiteness of $Q$.
The authors are not aware of an example of a (weakly or strongly) consistent Noetherian dimer model on a surface other than a disk, annulus, or torus.
Dimer models on tori and disks have been studied extensively. 
\begin{itemize}
	\item Any consistent dimer model on a torus is Noetherian~\cite{Beil2021}, but since the boundary idempotent is zero such a dimer model is never boundary-finite. On the other hand, as in Remark~\ref{remk:gggg} we may use a larger idempotent to apply Theorem~\ref{thm:cat-lemma}.
	\item If $Q$ is a consistent dimer model on a disk with no digons, Pressland showed in~\cite[Proposition 4.4]{XPressland2019} that $Q$ is Noetherian and boundary-finite. Hence, Corollary~\ref{thm:cat-thm} can be applied to any consistent dimer model on a disk.
Moreover, in the same paper it is shown that $(\textup{GP}(\widehat B_Q),e\widehat A_Q)$ is a Frobenius 2-Calabi-Yau realization of the cluster algebra $\AA_Q$ given by the ice quiver $Q$.
\end{itemize}

On the other hand, dimer models on annuli have seen comparatively little attention.
We give some examples pertaining to Corollary~\ref{thm:cat-thm} in the following subsection.

\subsection{Examples on the Annulus}

We first prove a technical lemma which we will use to verify Noetherianness of some examples from this section.  Recall that $r_C$ denotes a minimal path in the homotopy class $C$ of paths on the surface of $Q$.

\begin{lemma}\label{lem:noeth}
	Let $Q$ be a strongly consistent dimer model on an annulus.
	Pick a vertex $v$ of the dimer model and a generator of the homotopy group of cycles on the annulus at $v$.  Let $C_v$ be the homotopy class of this generator. For any vertex $w$ of $Q$, pick a path $p:w\to v$ and let $C_w$ be the homotopy class of $p^{-1}C_vp$ (note this does not depend on the choice of $p$).
	Suppose there exists a perfect matching $\mathcal M$ of $Q$ such that, for any vertex $w$ of $Q$, we have $\mathcal M(r_{C_v})=\mathcal M(r_{C_w})$ and $\mathcal M(r_{C_v^{-1}})=\mathcal M(r_{C_w^{-1}})$. Then $A_Q$ and $\widehat A_Q$ are Noetherian.
\end{lemma}
\begin{proof}
	Define $\A$ to be either $A_Q$ or $\widehat A_Q$. We must show that $\A$ is Noetherian.
	
	We first claim that for any vertex $w$ and any $m\geq1$, the path $(r_{C_w})^m$ is the minimal path in the homotopy class $C_w^m$. We show this by induction on $m$. The base case is true by hypothesis.
	Now choose minimal cycle $p$ in some homotopy class $C_w^m$ (for $m\in\mathbb N_{>1}$) and suppose we have shown the claim for smaller values of $m$.
	Since $S(Q)$ is an annulus and $m>1$, we must be able to factor $p=p_2 q p_1$, where $q$ is a cycle and either $p_1$ or $p_2$ (or both) is nonconstant. Since $p$ is minimal, the paths $p_1,q,p_2$ must be minimal, so by the induction hypothesis we know that $[q]=[r_{C_{t(q)}}^a]$ and $[p_2p_1]=[r_{C_w}^b]$ 
	for some $a,b\in\mathbb N_{\geq1}$ such that $a+b=m$. 
	{Then $\mathcal M(p)=\mathcal M(p_2 qp_1)=\mathcal M(r_{C_{t(q)}}^a)+\mathcal M(r_{C_{w}}^b)=\mathcal M(r_{C_w}^m)$}
	and hence $[p]=[p_2qp_1]=[r_{C_w}^m]$ by Proposition~\ref{prop:perfect-matching-intersection-num}.

	Symmetrically, for any vertex $w$ and $m\geq1$, we have that $(r_{C_w})^{-m}$ is the minimal path in the homotopy class $C_w^{-m}$.

	Define $M:=\mathcal M(r_{C_v}r_{C_v^{-1}})$.
	Note that by assumption for any $w\in Q_0$ we have 
	$\mathcal M(r_{C_w}r_{C_{w}^{-1}})=\mathcal M(r_{C_v}r_{C_{v}^{-1}})=M$. Then for any $w\in Q_0$ we have $r_{C_w}r_{C_w^{-1}}=f_w^M$, hence $\left(\sum_{w\in Q_0}r_{C_w}\right)\left(\sum_{w\in Q_0}r_{C_w^{-1}}\right)=f^M$.

	Define the subalgebra $\mathcal Z$ of $\mathcal A$ as the set of (possibly infinite if $\mathcal A=\widehat A_Q$) linear combinations of elements of the form $\sum_{w\in Q_0}f^ar_{C_w}^b$ or $\sum_{w\in Q_0}f^ar_{C_w^{-1}}^b$, for $a,b\in\mathbb N_{\geq0}$.
	Then $\mathcal Z$ is the polynomial ring (if $\mathcal A=A_Q$) or ring of formal power series (if $\mathcal A=\widehat A_Q$) in three variables $f$ and $\sum_{w\in Q_0}r_{C_w}$ and $\sum_{w\in Q_0}r_{C_w^{-1}}$, modulo the relation
	$\left(\sum_{w\in Q_0}r_{C_w}\right)\left(\sum_{w\in Q_0}r_{C_w^{-1}}\right)=f^M$.
	In either case, $\mathcal A$ is Noetherian by standard results.

	Let $S$ be the set of paths of $Q$ containing no cycles. Since $Q$ is finite, $S$ is also finite.
	We claim that $\mathcal A$ is generated as a $\mathcal Z$-algebra by $S$.

	Indeed, take any path $p$ of $Q$. 
	{We may write $p$ in the form $p_m l_m \dots l_2p_2l_1p_1$ for some $m\in\mathbb N_{\geq0}$ and paths $p_i$ and $l_j$ such that each $l_j$ is a cycle and $p_m\dots p_2p_1$ is a path which contains no cycles. This expression is not unique.}
	By the claim, each $l_j$ is equivalent to $r_{C_{t(l_j)}^{\varepsilon_j}}^{a_j}f^{b_j}$ for some $a_j,b_j\in\mathbb N_{\geq0}$ and $\varepsilon_j\in\{1,-1\}$.
	Then
	\begin{align*}
		[p]&=[p_ml_m\dots l_2p_2l_1p_1] \\
		&=\left[p_m\left(
		r_{C_{t(l_m)}^{\varepsilon_m}}^{a_m}f^{b_m}\right)\dots p_2
		\left(r_{C_{t(l_1)}^{\varepsilon_1}}^{a_1}f^{b_1}\right)p_1\right]\\
		&=\left[
			\left(\sum_{j\in[m] :\ \varepsilon_j=1}r_{C_{h(p_m)}}^{a_j}f^{b_j}\right)
			\left(\sum_{j\in[m] :\ \varepsilon_j=-1}r_{C_{h(p_m)}^{-1}}^{a_j}f^{b_j}\right)\right][p_m\dots p_1].
	\end{align*}
	Note that $p_m\dots p_1\in S$ and 
$\left[ \left(\sum_{j\in[m] :\ \varepsilon_j=1}r_{C_{h(p_m)}}^{a_j}f^{b_j}\right) \left(\sum_{j\in[m] :\ \varepsilon_j=-1}r_{C_{h(p_m)}^{-1}}^{a_j}f^{b_j}\right)\right]$ is in $\mathcal Z$. Since every path is of this form, it follows that $\mathcal A$ is generated by the finite set of paths $S$ as a $\mathcal Z$-module. Since $\mathcal Z$ is Noetherian, the algebra $\mathcal A$ is also Noetherian, completing the proof.
\end{proof}

Before giving positive examples, we give three examples to which Corollary~\ref{thm:cat-thm} may \textit{not} be applied in Figure~\ref{fig:noeth-ann}, explained in the following two example environments. These examples show that neither the assumption of Noetherianness nor boundary-finiteness is vacuous in Corollary~\ref{thm:cat-thm}.

\def\svgwidth{410pt}
\begin{figure}
	\centering
	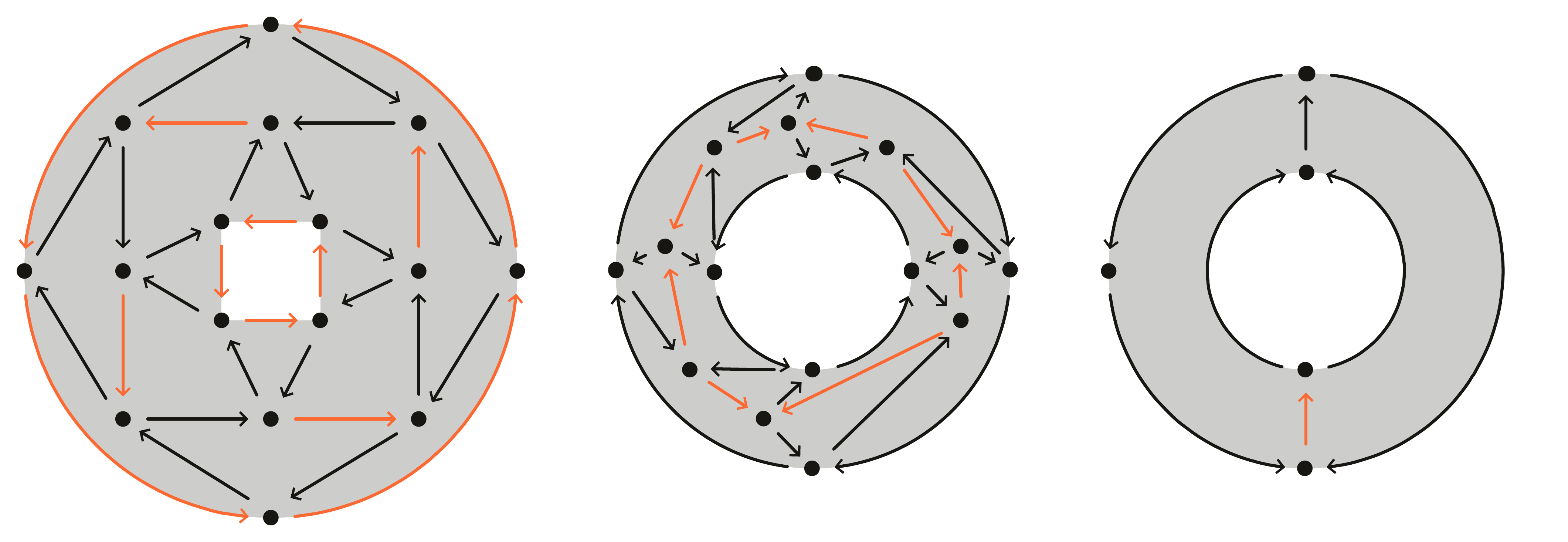
	\caption{On the left is a strongly consistent dimer model on an annulus which is Noetherian but not boundary-finite. In the middle is a weakly consistent dimer model on an annulus which is boundary-finite but not Noetherian. On the right is a strongly consistent dimer model on an annulus which is boundary-finite but not Noetherian.}
	\label{fig:noeth-ann}
\end{figure}

\begin{example}\label{ex:left}
	Consider the dimer model $Q$ on the left of Figure~\ref{fig:noeth-ann}. It may be checked using the strand diagram that $Q$ is weakly consistent, and it is not hard to see that $Q$ is in addition nondegenerate. Let $\mathcal M$ be the perfect matching consisting of the arrows drawn in red.
	For any vertex $w\in Q_0$, let $C_w$ be the homotopy class of cycles at $w$ winding once counter-clockwise around the annulus as embedded in Figure~\ref{fig:noeth-ann}. Then it may be checked that, for any $w$, we have $\mathcal M(r_{C_w})=0$ and $\mathcal M(r_{C_w^{-1}})=4$. Then Lemma~\ref{lem:noeth} shows that $A_Q$ and $\widehat A_Q$ are Noetherian.
	On the other hand, if $w$ is an internal vertex, then any power of $r_{C_w}$ is the only path in its equivalence class, showing that $Q$ is not boundary-finite. Then $Q$ satisfies all requirements for Corollary~\ref{thm:cat-thm} except for boundary-finiteness.
\end{example}
\begin{example}\label{ex:right}
	Consider the dimer model $Q$ in the middle of Figure~\ref{fig:noeth-ann}.
	As in Lemma~\ref{lem:noeth}, let $\mathcal A$ be $A_Q$ or $\widehat A_Q$.
	It may be verified using the strand diagram that $Q$ is weakly consistent, though $Q$ is not strongly consistent because no perfect matching contains any of its outer boundary arrows.
	Since every path with at least two arrows passes through a boundary vertex, the model $Q$ is boundary-finite.

	Let $\mathcal M$ be the perfect matching consisting of the arrows drawn in red.
	We have $\mathcal M(r_{C_{v_1}})=0$ but $\mathcal M(r_{C_{v_2}})=1$.
	Now for any $j\in\mathbb N_{\geq0}$, define $I_j$ to be the left ideal of $\mathcal A$ generated by $\{\alpha r_{C_{v_1}}^{i}\ :\ i\in[j]\}$, where $\alpha$ denotes the arrow $v_1\to v_2$. 
	For any $j$, we have $\mathcal M(\alpha r_{C_{v_1}}^j)=0$, hence $\alpha r_{C_{v_1}}^j$ is minimal. Since this path has no morphable arrows, it is the only minimal path in its homotopy class. This shows that $[\alpha r_{C_{v_1}}^j]$ is nonzero in $\widehat A_Q$ and that, moreover, if $j>0$ then $\alpha r_{C_{v_1}}^j$ is not in the left ideal $I_{j-1}$. Then $I_0\subsetneq I_1\subsetneq I_2\subsetneq \dots$ is an infinite increasing chain of left ideals of $\mathcal A$ which never stabilizes, hence $\mathcal A$ is not Noetherian.
	{This example is notable as it comes from a triangulated annulus as in~\cite[\S13]{BKMX}. The fact that such models are not nondegenerate or Noetherian indicates that they may not be the nicest dimer models on annuli from our perspective.}
\end{example}
\begin{example}\label{ex:rightest}
	Consider the dimer model $Q$ on the right of Figure~\ref{fig:noeth-ann}. It may be verified using the strand diagram that $Q$ is weakly consistent, and it is not hard to see that $Q$ is in addition nondegenerate. Boundary-finiteness of $Q$ is immediate since all vertices are boundary.

	Let $\mathcal A$ be $A_Q$ or $\widehat A_Q$.
	For $j\in\mathbb N_{\geq0}$, let $I_j$ be the left ideal of $\mathcal A$ generated by $\{\alpha_1(\gamma_1\beta_1\gamma_2\alpha_3)^i\ :\ i\in[j]\}$. As in Example~\ref{ex:right}, we may check that $[\alpha_1(\gamma_1\beta_1\gamma_2\alpha_3)^{j+1}]\not\in I_j$, hence $I_0\subsetneq I_1\subsetneq\dots$ is an increasing chain of left ideals of $\mathcal A$ which never stabilizes, hence $\mathcal A$ is not Noetherian.
	This dimer model then satisfies all assumptions of Corollary~\ref{thm:cat-thm} except for Noetherianness.
\end{example}

Example~\ref{ex:left} shows that a strongly consistent and Noetherian dimer model may fail to be boundary finite. Example~\ref{ex:rightest} shows that a strongly consistent and boundary-finite dimer model may have a non-Noetherian (completed and noncompleted) dimer algebra. So, no conditions of Corollary~\ref{thm:cat-thm} are vacuous.
Further work is required to understand how the Noetherian condition interacts with weak and strong consistency, as in~\cite{Beil2021} on the torus. It would be of interest to obtain a condition on the strand diagram for a strongly consistent dimer model (on an annulus or in general) to be Noetherian.
For now, we give examples of annulus models satisfying the conditions of Corollary~\ref{thm:cat-thm}.

\begin{example}\label{ex:noeth-fin}
	Figure~\ref{fig:noeth-fin} shows two strongly consistent dimer models on annuli satisfying Noetherianness and boundary finiteness.
	Consider the model $Q$ on the right. Let $\mathcal M$ be the perfect matching consisting of the six arrows drawn vertically in Figure~\ref{fig:noeth-fin}.
	For any vertex $w\in Q_0$, we let $C_w$ be the homotopy class of paths winding once counter-clockwise around the model. It can then be checked that $\mathcal M(r_{C_w})=4=\mathcal M(r_{C_w^{-1}})$ for any vertex $w\in Q_0$.
	Then Lemma~\ref{lem:noeth} shows that $A_Q$ and $\widehat A_Q$ are Noetherian.
	The model $Q$ is boundary-finite because any path $p$ of length at least four is equivalent to a path which factors through a boundary vertex. We have now seen that $Q$ is strongly consistent, Noetherian, and boundary-finite.
	The models of Figure~\ref{fig:noeth-fin} suggest a general construction. Note that the internal subquiver of the model on the left is an alternating affine type $\widetilde A_4$ quiver, and on the right the internal subquiver is two layers of affine type $\widetilde A_4$ quivers which are connected to each other. One may obtain similar models whose internal subquiver is any number of layers of an alternating affine type $\widetilde A_{2n}$ quiver.
	\def\svgwidth{300pt}
	\begin{figure}
		\centering
		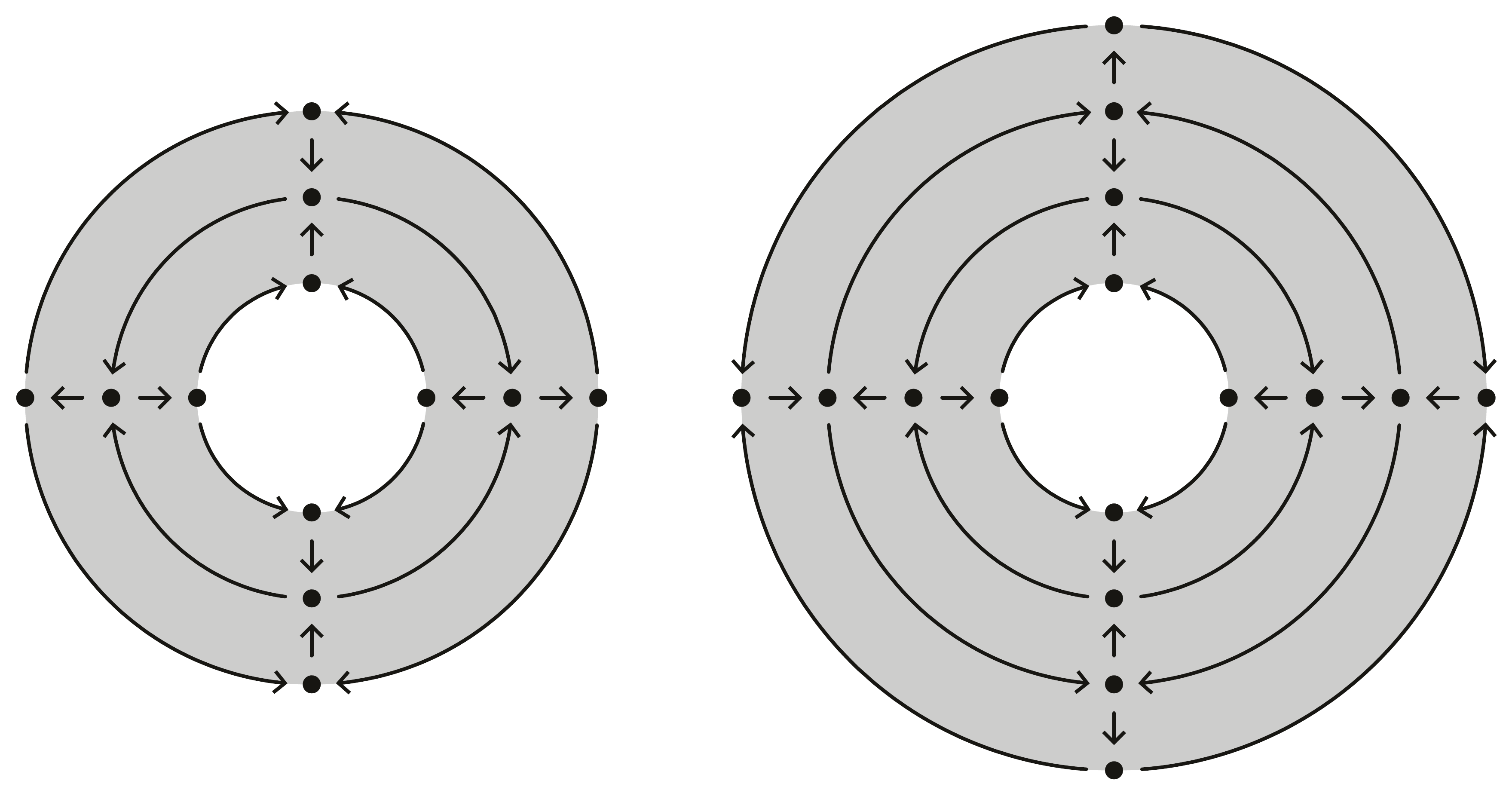
		\caption{Two strongly consistent dimer models on annuli satisfying Noetherianness and boundary finiteness.}
		\label{fig:noeth-fin}
	\end{figure}
\end{example}
\begin{example}\label{ex:noeth-finn}
	Let $Q$ be the dimer model of Figure~\ref{fig:noeth-finn}. One may check similarly to Example~\ref{ex:noeth-fin} that $Q$ is strongly consistent, Noetherian, and boundary-finite. Similarly to the previous example, the internal subquiver of this dimer model consists of two layers of some orientation of an affine type $\widetilde A_{11}$ quiver, stitched together with some vertical and diagonal arrows. Again, a more general construction is indicated, and one may obtain a similar strongly consistent, Noetherian, and boundary-finite model for any number of layers of any orientation of an affine type quiver.
	\def\svgwidth{300pt}
	\begin{figure}
		\centering
		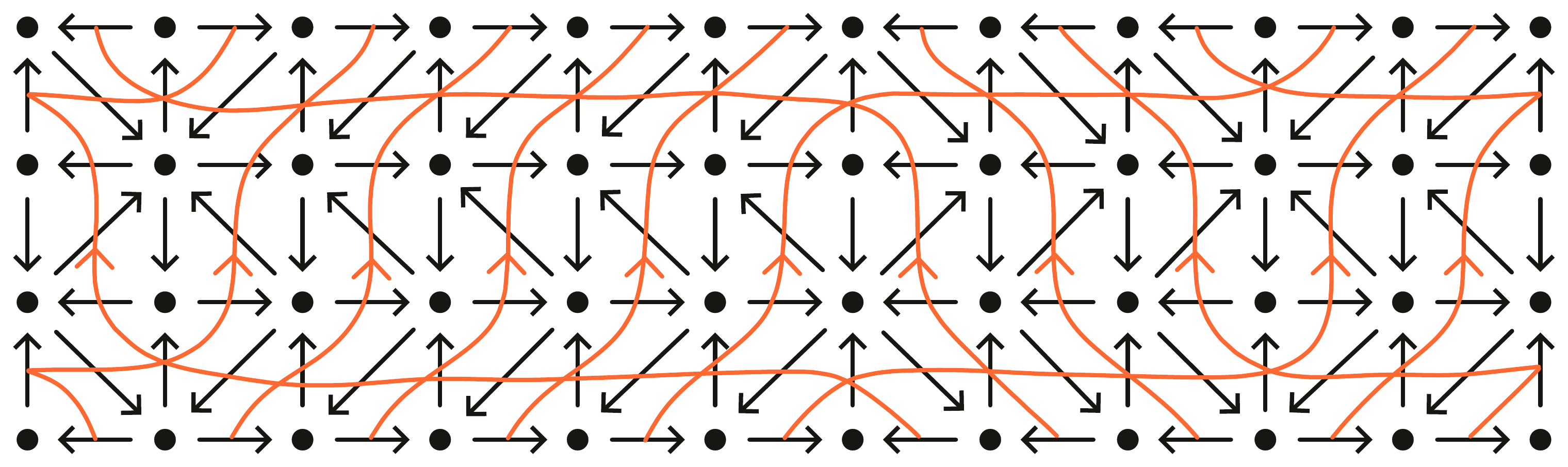
		\caption{A strongly consistent dimer model on an annulus satisfying Noetherianness and boundary finiteness. The left and right sides should be identified. The strands going up are shown in red, while the strands going down are left out for readability.}
		\label{fig:noeth-finn}
	\end{figure}
\end{example}

Corollary~\ref{thm:cat-thm} may be applied to the annulus models of Examples~\ref{ex:noeth-fin} and~\ref{ex:noeth-finn} to show that the Gorenstein-projective module categories over their boundary algebras categorify the cluster algebra given by their underlying quivers.

\section{Extra Results about Equivalence Classes}\label{sec:thin-quivers}\label{sec:eraec}

For this section, unless otherwise specified we suppose that $Q$ is a weakly consistent dimer model and $\widetilde{Q}$ is a simply connected dimer model that is not on a sphere. We use Theorem~\ref{thm:submodel-path-equivalence} to get some interesting results about the path equivalence classes in weakly consistent quivers. These results are not used anywhere else in the paper. 

\begin{prop}\label{prop:catch-all-right-left-c-value-thing}
    Let $p$ and $q$ be distinct elementary paths in a weakly consistent dimer model $\widetilde Q$ with the same start and end vertices. If $p$ is to the right of $q$, then either $p$ has a left-morphable arrow or $q$ has a strictly higher c-value than $p$.
\end{prop}
\begin{proof}
	We may reduce to the case where $p$ and $q$ share no vertices except for the start and end vertices. Then $q^{-1}p$ defines a disk submodel $Q'$ of $\widetilde Q$ in which $p$ consists only of boundary arrows in counter-clockwise faces and $q$ consists of boundary arrows in clockwise faces. If $p$ has a strictly higher c-value than $q$ in $\widetilde Q$, then we are done. Suppose the c-value of $p$ is less than or equal to that of $q$ in $\widetilde Q$. Then $[p]=[qf^m]$ for some $m$ in $\widetilde Q$. By Theorem~\ref{thm:submodel-path-equivalence}, the same is true in $Q'$. It follows that $p$ has a left-morphable arrow in $Q'$, and hence in $\widetilde Q$. This proves the desired result.
\end{proof}

\begin{cor}\label{cor:lrcool}
	Let $Q=\widetilde Q$ be a simply connected dimer model.
	Let $v$ and $w$ be vertices of $Q$. Suppose that there is a leftmost minimal path $p$ from $v$ to $w$. Then $p$ is to the left of every minimal path and to the right of every leftmost path.
\end{cor}
\begin{proof}
	If $p'$ is any other minimal path from $v$ to $w$, Proposition~\ref{prop:catch-all-right-left-c-value-thing} shows that $p$ is to the left of $p'$. If $q$ is any other leftmost path from $v$ to $w$, Proposition~\ref{prop:catch-all-right-left-c-value-thing} shows that $q$ is to the left of $p$.
\end{proof}

If $Q$ is not simply connected, we may still pass to the universal cover and apply Corollary~\ref{cor:lrcool} to gain insight into the equivalence classes of paths of $Q$.

\begin{cor}\label{cor:thin-leftmost-path-unique}
    
	Suppose $Q$ is a weakly consistent dimer model. If $p$ and $q$ are minimal homotopic leftmost paths with the same start and end vertices, then $p=q$.
\end{cor}
\begin{proof}
	Suppose $p$ and $q$ are homotopic minimal leftmost paths with the same start and end vertices. Lift them to minimal leftmost paths $\tilde p$ and $\tilde q$ of $\widetilde Q$. Since $p$ and $q$ are homotopic, these paths have the same start and end vertices. It suffices to show that $\tilde p=\tilde q$. Suppose this is not the case. By taking subpaths, we may reduce to the case where $\tilde p$ and $\tilde q$ share only their start and end vertices. Say $\tilde p$ is to the right of $\tilde q$. Since both paths have the same c-value, Proposition~\ref{prop:catch-all-right-left-c-value-thing} shows that $\tilde p$ has a left-morphable arrow, a contradiction.
\end{proof}

Corollary~\ref{cor:thin-leftmost-path-unique} indicates that there is at most one minimal leftmost path between any two vertices of $Q$.
In general, leftmost paths need not exist. Figure~\ref{fig:non-red-annulus} gives an example of a finite weakly consistent dimer model on an annulus where a minimal path has an infinite equivalence class. Let $\alpha$ be one of the arrows of the digon of $Q$. Then an arbitrarily long series of left-morphs at the other arrow of the digon may be applied to $\alpha$. In other words, there is no leftmost path from the outer boundary ring to the inner boundary ring.
The dimer model in the middle of Figure~\ref{fig:noeth-ann} has no 2-cycles and displays the same behavior: a minimal path from a vertex of the inner boundary to a vertex of the outer boundary may be left-morphed or right-morphed indefinitely, depending on the path.

On the other hand, by assuming nondegeneracy we guarantee the existence of leftmost and rightmost paths between any given vertices.
\begin{thm}\label{cor:c}
	Let $Q$ be a strongly consistent dimer model. Let $v$ and $w$ be distinct vertices of $Q$ and let $C$ be a homotopy class of paths from $v$ to $w$. Then there is a unique minimal leftmost path $p$ from $v$ to $w$ in $C$. 
\end{thm}
\begin{proof}
	We show the existence of a leftmost path from $v$ to $w$. Uniqueness follows from Corollary~\ref{cor:thin-leftmost-path-unique}. Lift $v$ and $w$ to vertices $\tilde v$ and $\tilde w$ of the universal cover model $\widetilde Q$ such that a path from $\tilde v$ to $\tilde w$ descends to a path of $Q$ in homotopy class $C$.
	We consider paths on $\widetilde Q$ to have the grading of their corresponding paths on $Q$.
	By path-consistency, all minimal paths from $\tilde v$ to $\tilde w$ are equivalent and hence have the same grading $k$. Then if $\tilde q=\alpha_m\dots\alpha_1$ is a minimal path from $\tilde v$ to $\tilde w$, then since the degree of each arrow is positive we must have $m\leq k$. This shows that the length of a minimal path from $\tilde v$ to $\tilde w$ is bounded by $k$. This condition guarantees that there are a finite number of minimal paths from $\tilde v$ to $\tilde w$.

	Now consider again the minimal path $\tilde q$ from $\tilde v$ to $\tilde w$. Let $\tilde r$ be any path from $\tilde w$ to $\tilde v$. 
	If $\tilde q$ is not leftmost, then left-morph it at some arrow $\alpha_1$ to get some $\tilde q_1$. Then Lemma~\ref{lem:rotation-number-formula} gives that
	$\wind(\tilde r\tilde q_1,F)<\wind(\tilde r\tilde q,F)$
		if $F$ is one of the faces containing $\alpha_1$, and
	$\wind(\tilde r\tilde q_1,F)=\wind(\tilde r\tilde q,F)$
	otherwise.
		Continue to left-morph to get some sequence of paths $\tilde q,\tilde q_1,\tilde q_2,\dots$. By the above inequalities, for any $j$ the inequality $\label{fdsdf}\wind(\tilde r\tilde q_{j+1},F)\leq\wind(\tilde r\tilde q_j,F)$ holds for all faces $F$, and is strict for some choice of $F$. Repeating this argument shows that $\tilde q_i\neq\tilde q_j$ for $i\neq j$.
	Since there are a finite number of minimal paths from $\tilde v$ to $\tilde w$ and the sequence never repeats itself, the sequence must terminate with some leftmost path $\tilde p:=\tilde q_l$ from $\tilde v$ to $\tilde w$. This descends to a leftmost path $p$ from $v$ to $w$ in homotopy class $C$.
\end{proof}

Theorem~\ref{cor:c} and Corollary~\ref{cor:lrcool} are useful tools to understand equivalence classes of paths in nondegenerate dimer models.
{In particular, if $Q$ is a consistent (hence also nondegenerate) dimer model in a disk and $e$ is the boundary idempotent of $Q$, the \textit{boundary algebra} $eA_Qe$ may be used to obtain an additive Frobenius categorification of the ice quiver of $Q$~\cite{XPressland2019}. In a future paper, we will use these results to study boundary algebras of consistent dimer models on disks.}

\section{Reduction of a Dimer Model}\label{sec:red}
\label{sec:reduce}

There has been some interest, particularly in the disk case, in \textit{reducing} a dimer model by removing digons. See~\cite{XPressland2019} and~\cite[Remark 2.9]{XBroomhead2009}.
In particular, if a weakly consistent, Noetherian, boundary-finite dimer model $Q$ has no 1-cycles or 2-cycles, our categorification result Corollary~\ref{thm:cat-thm} shows that $\textup{GP}(Q)$ categorifies the cluster algebra $\AA_Q$ given by the underlying quiver of $Q$. On the other hand, if $Q$ has 1-cycles or 2-cycles, then it may not be immediately clear what cluster algebra is being categorified. Reducing a dimer model helps to avoid this issue.
In~\cite{XPressland2019}, Pressland works with consistent dimer models on a disk with more than three boundary vertices. He observes that, in this case, the removal of digons corresponds to untwisting moves in the strand diagram and hence their removal is straightforward. Moreover, the resulting \textit{reduced dimer models} do not have any 2-cycles, and the disk version of Corollary~\ref{thm:cat-thm} can be used to prove results about categorification. The situation for more general weakly consistent dimer models is more complicated, as there may be copies similar to Figure~\ref{fig:non-red-annulus} contained in the dimer model whose digons may not be removed. However, such occurrences will force the dimer model to be degenerate. Then in the nondegenerate case we may freely remove internal digons from a dimer model, that is digons where both of its arrows are internal.

We now consider the reduction of a dimer model by removing digons.
First, we consider weakly consistent dimer models.
\begin{prop}\label{prop:reduce-dimer-model}
	Let $Q$ be a weakly consistent dimer model with a finite number of digons. There exists a \textit{reduced dimer model $Q_{red}$ of $Q$} satisfying the following:
	\begin{enumerate}
		\item\label{prdm:0} $S(Q_{red})=S(Q)$,
		\item\label{prdm:1} $A_{Q_{red}}\cong A_Q$,
		\item\label{prdm:2}$Q_{red}$ is weakly consistent,
		\item\label{prdm:55}If $Q$ is nondegenerate, then $Q_{red}$ is nondegenerate, and
		\item\label{prdm:-1} Either $Q_{red}$ is a dimer model on a disk composed of a single digon, or every digon of $Q_{red}$ is an internal face incident to only one other face. 
	\end{enumerate}

\end{prop}
The condition~\eqref{prdm:-1} means that, as long as $Q_{red}$ is not composed of a single digon, every digon of $Q_{red}$ is in a configuration like that of Figure~\ref{fig:non-red-annulus}. See also Figure~\ref{fig:cons-tor-no-perf}. In contrast, the configuration of Figure~\ref{fig:tr-annulus} is not possible in a weakly consistent dimer model.
\begin{proof}
	Note that~\eqref{prdm:2} follows from~\eqref{prdm:1}. Suppose that $Q$ is not merely a digon and let $\alpha\beta$ be a digon of $Q$ which is either a boundary face or is incident to two distinct faces. Then $\alpha$ and $\beta$ may not both be boundary arrows. Suppose $\alpha\beta$ is a clockwise face; the counter-clockwise case is symmetric. 

	If $\alpha$ and $\beta$ are both internal arrows, then by assumption the faces $F_\alpha^{cc}$ and $F_\beta^{cc}$ are the distinct neighbors of the digon $\alpha\beta$. 
	Let $Q^1=(Q^1_0,Q^1_1,Q^1_2)$ be the dimer model whose underlying quiver is $Q$ without $\alpha$ and $\beta$ and whose set of faces is the same as $Q_2$, but with the faces $\{\alpha\beta,F_\alpha^{cc},F_\beta^{cc}\}$ replaced with one face whose arrows are those of $R_\alpha^{cc}$ and $R_\beta^{cc}$. Since the dimer algebra relations of $Q$ give $[\alpha]=[R_\beta^{cc}]$ and $[\beta]=[R_\alpha^{cc}]$, the dimer algebra is not changed by this operation. The surface is also unchanged by this operation. 
	If $\mathcal M$ is a perfect matching of $Q$, then $\mathcal M$ contains either $\alpha$ or $\beta$, and the removal of this arrow from $\mathcal M$ gives a perfect matching of $Q$. Hence, if $Q$ is nondegenerate then $Q^1$ is nondegenerate.
		See Figure~\ref{fig:digon-remove1}.

		On the other hand, suppose without loss of generality that $\alpha$ is a boundary arrow and $\beta$ is internal. Then let $Q^1$ be the dimer model obtained by removing the arrow $\alpha$ from $Q_1$ and the face $\alpha\beta$ from $Q_2$. As above, $[\alpha]=[R_\beta^{cc}]$ in $A_Q$ and hence the dimer model is unchanged by this operation. If $\mathcal M$ is a perfect matching of $Q$, then $\mathcal M$ contains $\alpha$ or $\beta$. If $\mathcal M$ contains $\alpha$, then removing this arrow gives a perfect matching of $\mathcal M$. If $\mathcal M$ contains $\beta$, then $\mathcal M$ is also a perfect matching of $Q$. Then if $Q$ is nondegenerate, then $Q^1$ is nondegenerate. Furthermore, $S(Q^1)=S(Q)$. 
		See Figure~\ref{fig:digon-remove2}. 

		\def\svgwidth{200pt}
		\begin{figure}\centering
			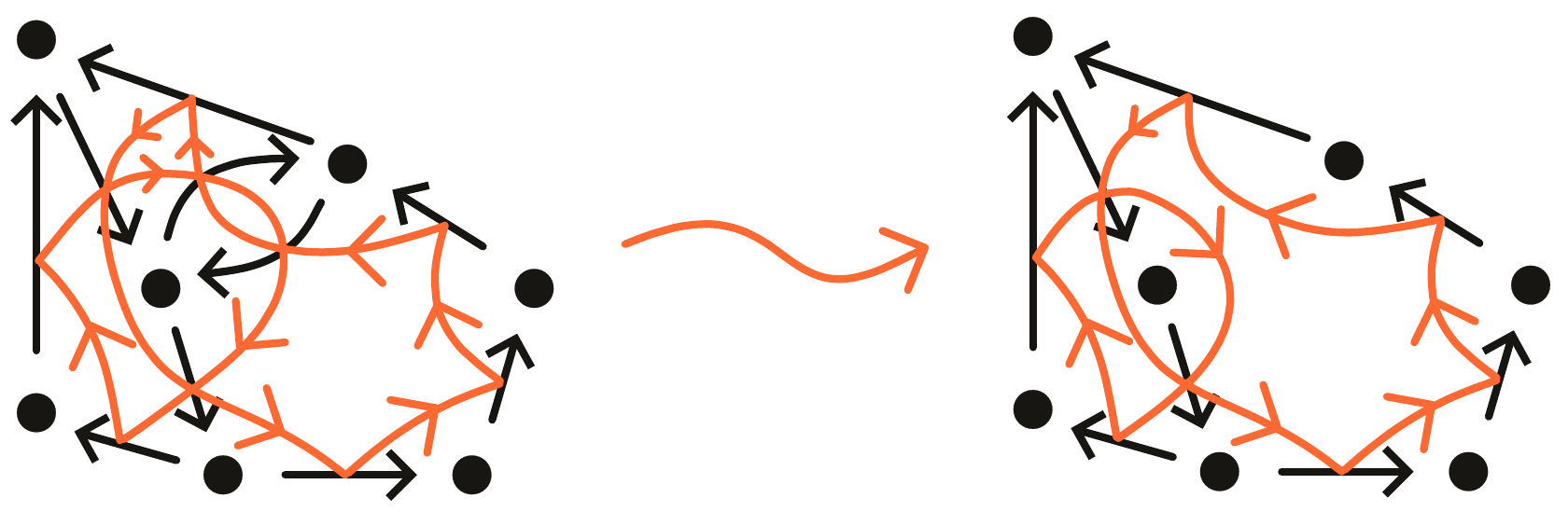
			\caption{Removing an internal digon.}
			\label{fig:digon-remove1}
		\end{figure}
		\def\svgwidth{200pt}
		\begin{figure}\centering
			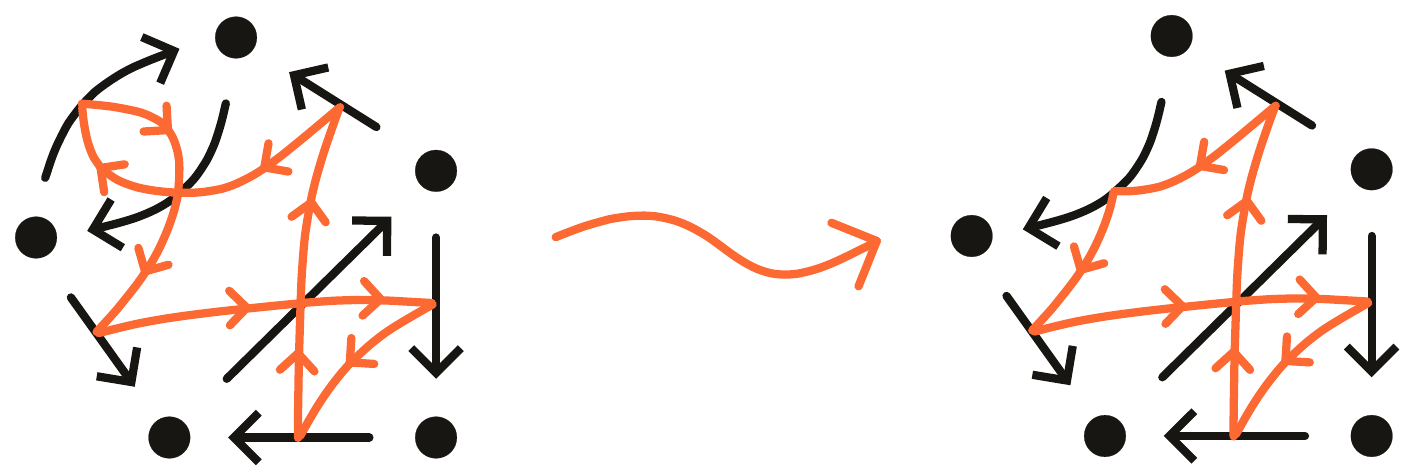
			\caption{Removing a boundary digon.}
			\label{fig:digon-remove2}
		\end{figure}

		In either case, we have defined a quiver $Q^1$ such that $S(Q^1)=S(Q)$ and $A_{Q^1}\cong A_Q$. Furthermore, $Q^1$ has strictly less digons than $Q$. We may now apply this process repeatedly to remove all such digons of $Q$.  Note that a digon cannot be removed only if it is internal and incident to only one face or it is incident to no other faces.  In the latter case, such a digon constitutes the entire dimer model.
		Therefore, since $Q$ has a finite number of digons, this process must terminate with some $Q_{red}$ such that $Q_{red}$ is a dimer model on a disk composed of a single digon, or every digon of $Q_{red}$ is internal and incident to only one other face. 
\end{proof}

\def\svgwidth{70pt}
\begin{figure}
	\centering
	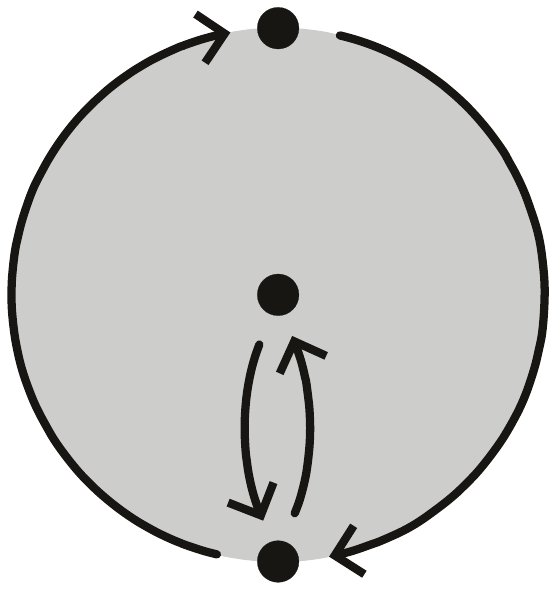
	\caption{One way that that a digon could be incident to only one other face. The strand diagram has a bad lens, hence this configuration is not possible in a weakly consistent dimer model.}
	\label{fig:tr-annulus}
\end{figure}

\def\svgwidth{120pt}
\begin{figure}
	\centering
	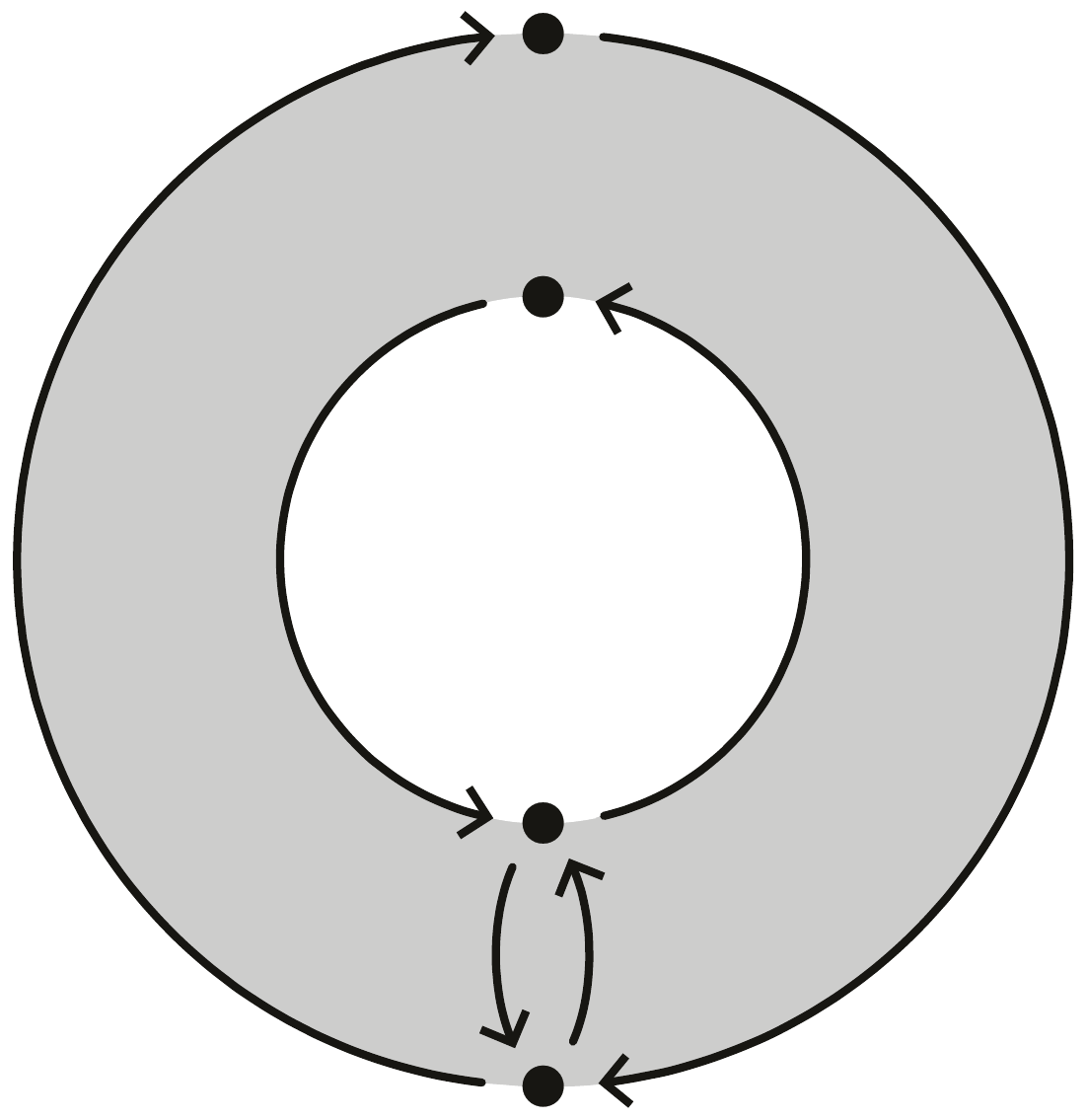
	\caption{A reduced dimer model on an annulus with a digon which may not be removed.} 
	\label{fig:non-red-annulus}
\end{figure}

Figure~\ref{fig:non-red-annulus} shows a weakly consistent model with an internal digon which may not be removed by the process of the above theorem. Indeed, if the digon is removed then the resulting ``dimer model'' would have a face which is not homeomorphic to an open disk. On the level of strand diagrams, removing the digon corresponds to an untwisting move that disconnects the strand diagram. 
If $Q$ has an infinite number of digons, a new problem appears. See Figure~\ref{fig:cant-reduce}, which shows an infinite model with an infinite number of digons. While any finite number of digons may be removed, all of them may not be removed at once. The universal cover of the dimer model of Figure~\ref{fig:non-red-annulus} displays the same behavior.

\begin{remk}
	In~\cite[\S3]{XPressland2020}, Pressland outlines a method of removing digons from general Jacobian ice quivers without changing the completed algebra. If this process is applied to the dimer model of Figure~\ref{fig:non-red-annulus}, the resulting ice quiver is not a dimer model.
\end{remk}

\def\svgwidth{120pt}
\begin{figure}
	\centering
	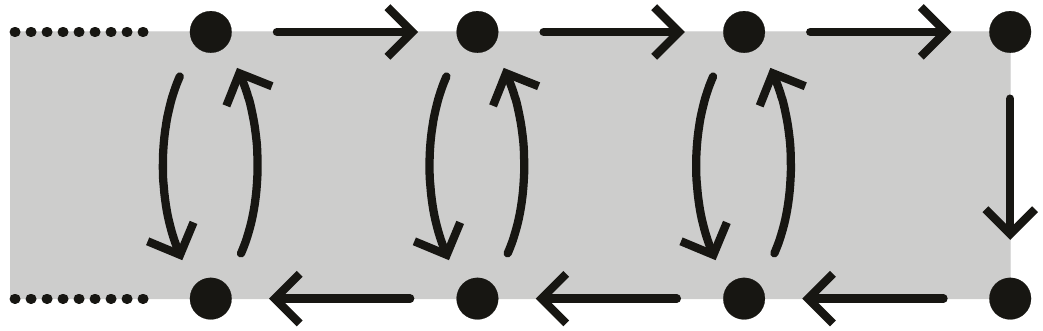
	\caption{Shown is a dimer model on an infinite half-strip. If all digons were removed, then there would only be one infinite face making up the entire non-compact surface, which is impossible.}
	\label{fig:cant-reduce}
\end{figure}

Note that neither Figure~\ref{fig:non-red-annulus} nor Figure~\ref{fig:cant-reduce} is nondegenerate. In fact, the following result shows that any strongly consistent dimer model may be reduced to a dimer model with no digons. 

\begin{cor}\label{cor:degenreduce}
	Let $Q$ be a strongly consistent dimer model. There exists a \textit{reduced dimer model $Q_{red}$ of $Q$} satisfying the following:
	\begin{enumerate}
		\item $S(Q_{red})=S(Q)$,
		\item $A_{Q_{red}}\cong A_Q$,
		\item $Q_{red}$ is strongly consistent, and
		\item Either $Q_{red}$ is a dimer model on a disk composed of a single digon, or $Q_{red}$ has no digons.
	\end{enumerate}
\end{cor}
\begin{proof}
	Apply Proposition~\ref{prop:reduce-dimer-model}. If $Q_{red}$ is not a dimer model on a disk composed of a single digon, then $Q_{red}$ has an internal digon $\alpha\beta$ which is incident to a single other face $F$. Let $\gamma$ be an arrow of $F$ which is not $\alpha$ or $\beta$. Any perfect matching $\mathcal M$ must contain $\alpha$ or $\beta$, since $\alpha\beta$ is a face. Then $\mathcal M$ cannot contain $\gamma$, since $\gamma$ shares a face with these arrows. We have shown that no perfect matching contains $\gamma$. Then $Q_{red}$, and by extension $Q$, is degenerate.
\end{proof}

We remark that $Q_{red}$ may have 1-cycles and 2-cycles, even if it has no digons. Consider the dimer model on a torus pictured on the left of Figure~\ref{fig:dignotdig}. While the quiver $Q$ has 2-cycles, it has no \textit{null-homotopic} 2-cycles. As we see by looking at the universal cover model on the right, this means that there are no digons in $Q$, hence $Q=Q_{red}$ is reduced.

\def\svgwidth{120pt}
\begin{figure}
	\centering
	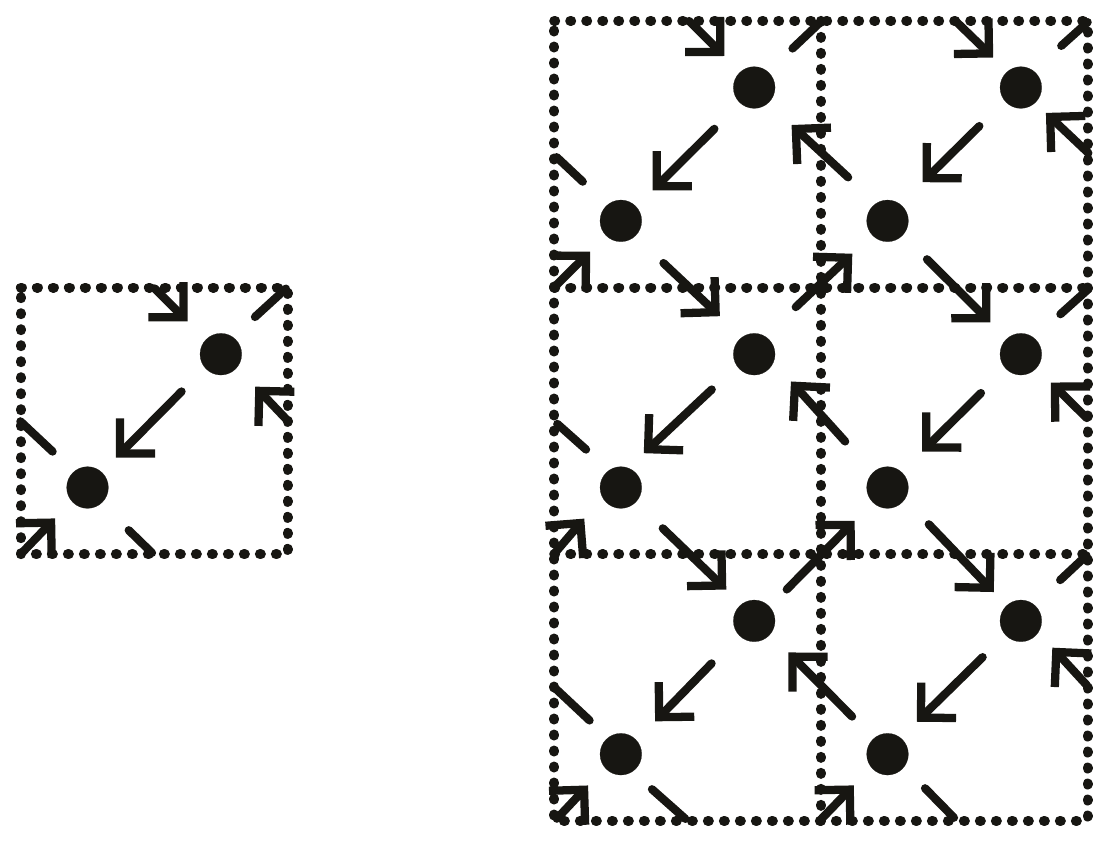
	\caption{Shown on the left is a dimer model on a torus. Opposite dashed edges should be identified. Shown on the right is a piece of its universal cover model.}
	\label{fig:dignotdig}
\end{figure}

\bibliographystyle{plain}
\bibliography{biblio}

\end{document}